\documentclass[reqno]{amsart}
\usepackage[normalem]{ulem}
\usepackage{amsmath,amsthm,amsfonts,amscd,amssymb, tikz-cd}
\usepackage{verbatim}       
\usepackage{mathtools}
\usepackage{epsfig}         
\usepackage{url} 
\usepackage{epic,eepic,latexsym}
\usepackage{graphicx}
\usepackage{color}
\usepackage{lcd}
\usepackage{mathrsfs}
\usepackage[centering,text={15.5cm,22cm},
        marginparwidth=20mm,dvips]{geometry}
\usepackage[hidelinks]{hyperref}
\usepackage[capitalize]{cleveref}
\usepackage{lscape}
\usepackage{tabularx}
\usepackage{marginnote}
\usepackage[all]{xy}
\usepackage{enumerate}
\usepackage{stmaryrd}
\usepackage{tipa}
\usepackage{upgreek}
\usepackage{tikz-cd}
\usepackage{adjustbox}
\usepackage{bbm}
\usepackage{chngcntr}
\usepackage[title]{appendix}
\usepackage{tikz}
\usetikzlibrary{positioning}
\usepackage{subcaption}
\usepackage{float}

\counterwithin{equation}{section}

\let\m\mathbb
\let\c\mathcal
\let\f\mathfrak

\let\b\textbf
\let\r\textnormal
\let\i\textit

\theoremstyle{plain}
 \newtheorem{thm}{Theorem}[section]
 
 \newtheorem{lem}[thm]{Lemma}
  
  \newtheorem{prop}[thm]{Proposition}
  \newtheorem{conj}[thm]{Conjecture}

\theoremstyle{definition}
 
 \newtheorem{dfn*}{Definition}[section]

 \newtheorem{eg}[thm]{Example}
 
 \newtheorem{rem}[thm]{Remark}
 
\theoremstyle{remark}
 
  \newtheorem*{rmk*}{Remark}

\title{Jacobi algebra presentations for fundamental group algebras}
\author{Vivek Mistry}
\address{School of Mathematics, University of Edinburgh, Edinburgh, United Kingdom}
\email{s1829507@ed.ac.uk}

\begin{document}

\maketitle

\begin{abstract}
    We prove a special case of a conjecture of Davison which pertains to superpotential descriptions of fundamental group algebras $k[\pi_1(X)]$. We consider the case in which the manifold $X$ is the mapping torus $M_{g, \varphi}$ of a genus $g$ Riemann surface $\Sigma_g$ and a finite order automorphism $\varphi$, and the superpotential structure is given by the Jacobi algebra of a quiver with potential.
\end{abstract}

\tableofcontents

\section{Introduction}
Given a $d$-dimensional manifold $X$ we can define its fundamental group algebra over a field $k$ as the group ring $k[\pi_1(X)].$ In \cite{gin} it is stated that if $X$ is compact, orientable and has a contractible universal cover then $k[\pi_1(X)]$ is a Calabi-Yau algebra of dimension $d$ (see [\cite{gin} Corollary 6.1.4] and [\cite{d3} Proposition 5.2.6]). Many Calabi-Yau algebras turn out to be what are known as superpotential algebras and so it was conjectured in \cite{gin} that in the dimension 3 case these fundamental group algebras which were Calabi-Yau were also superpotential algebras.

However in \cite{d3} Davison showed that this was not the case in general (for $d \geq 2$), and that in order to have a superpotential structure the algebra $k[\pi_1(X)]$ (and hence the manifold $X$) had to have certain specific properties. In particular a superpotential structure implies a stronger notion of an \emph{exact} Calabi-Yau structure which Davison showed required non-trivial central units in $k[\pi_1(X)]$. Therefore Davison proposed an updated conjecture [\cite{d3} Conjecture 7.1.1] regarding the possibility of a superpotential structure on $k[\pi_1(X)]$ when $X$ is a circle bundle, precisely because in such a manifold $X$ non-trivial central units in $k[\pi_1(X)]$ are easily found. In this paper we focus on the 3-dimensional case of this conjecture ie. when $X$ is a Seifert fibre bundle.

The ultimate goal of such a superpotential description of $k[\pi_1(X)]$ is in regards to calculating the Donaldson-Thomas (DT) invariants of $X$. Given a superpotential structure the DT invariants can then be though of in terms of vanishing cycles and Milnor fibres of a globally defined function on a smooth space which gives us a much stronger possibility of calculating them concretely. We remark on how this may be accomplished towards the end of the paper.

\subsection*{Acknowledgements}
I would like to thank my supervisor Ben Davison for proposing this project and many helpful discussions and ideas while writing this paper. This research was supported by the Royal Society studentship RGF\textbackslash R1\textbackslash 180093.

\section{Background}
Fix a field $k$ of characteristic 0 and let $A$ be a $k$-algebra. Henceforward all unadorned tensor products will be taken over the base field $k$ while tensor products over any other ring will be adorned with the respective ring, unless otherwise stated.

\begin{dfn*}
The \emph{enveloping algebra} of $A$ is the algebra
$$A^e= A \otimes A^{\r{op}}.$$
\end{dfn*}

Note that the category $A$-\b{Bimod} of $A$-bimodules is naturally isomorphic to $A^e$-\b{Mod} the category of left $A^e$-modules and we will freely swap between the two. From now on we shall work in the derived category and drop all such extra notation on the relevant derived functors. Define a functor
$$-^\vee : D(A\r{-\b{Bimod}}) \longrightarrow D(A\r{-\b{Bimod}})$$
by sending the bimodule $M$ to $\r{Hom}_{A^e}(M,A \otimes A)$ where $A\otimes A$ has the outer $A$-bimodule structure and $M^\vee$ has the $A$-bimodule structure induced by the inner $A$-bimodule structure on $A \otimes A$.

\subsection{Calabi-Yau algebras and Ginzburg differential graded algebras}
The Calabi-Yau structure is essential for defining DT invariants and suggests a superpotential description so we give an overview here following \cite{gin} and \cite{d3}. The subsequent definition for Calabi-Yau algebras is due to Ginzburg.

\begin{dfn*}
An object $M^\bullet$ in the derived category $D(A\r{-\b{Mod}})$ is called \emph{perfect} if it is isomorphic to a bounded complex of projective $A$-modules.
\end{dfn*}

\begin{dfn*}
An algebra $A$ is called \emph{homologically finite} if it is perfect as an $A^e$-module.
\end{dfn*}

\begin{dfn*}
A \emph{Calabi-Yau structure of dimension $d$} on a homologically finite algebra $A$ is an isomorphism in $D(A\r{-\b{Bimod}})$
$$f:A \xrightarrow{\,\,\sim\,\,} A^\vee[d] \quad \r{such that} \quad f=f^\vee[d].$$
\end{dfn*}
If $A$ is homologically finite we have an isomorphism
$$\r{HH}_d(A) \cong \r{Ext}^{-d}_{A^e}(A^\vee,A) = \r{Hom}_{A^e}(A^\vee[d],A)$$
where $\r{HH}_d(A)$ is the $d$th Hochschild homology of $A.$ So an isomorphism $A \cong A^\vee[d]$ corresponds to an element in $\r{HH}_d(A)$ (we call such elements \emph{non-degenerate}), and the self-duality $f=f^\vee[d]$ corresponds to the element being fixed by the induced map on homology of the flip isomorphism $\beta: A \otimes_{A^e} A \xrightarrow{\sim} A \otimes_{A^e} A$ that swaps the copies of $A$ in the tensor product. However it turns out that the map on Hochschild homology induced by $\beta$ is just the identity (see [\cite{vdb} Proposition C.1]) and so any non-degenerate element in $\r{HH}_d(A)$ or indeed any isomorphism $A \cong A^\vee[d]$ gives rise to a $d$-dimensional Calabi-Yau structure on $A$.

For an algebra $A$ recall we have the following long exact sequence in cyclic and Hochschild homology

$$\cdots \rightarrow \r{HC}_{n+1}(A) \rightarrow \r{HC}_{n-1}(A) \xrightarrow{\partial} \r{HH}_n(A) \rightarrow  \r{HC}_n(A) \rightarrow \cdots $$

\begin{dfn*}
An \emph{exact Calabi-Yau structure of dimension $d$} on an algebra $A$ is a non-degenerate element $\nu \in \r{HH}_d(A)$ that is in the image of the boundary map $\partial$.
\end{dfn*}

Now let $A$ be a differential graded algebra (dga). Its \emph{bimodule of 1-forms} is
$$\Omega^1A = \r{Ker}(A \otimes A \xrightarrow{\,\,\,m\,\,\,} A)$$
where $m$ is the multiplication map. $\Omega^1A$ inherits a grading from $A$ making it a differential graded $A$-bimodule.

\begin{dfn*}
A finitely generated dga $A$ is \emph{smooth} if $\Omega^1A$ is projective as an $A$-bimodule.
\end{dfn*}

Given a finitely generated negatively graded bimodule $V$ over a smooth algebra $A$ we define $T_A(V)$ to be the tensor-algebra generated by $V$ over $A$. If $V$ is free as a bimodule then we call $T_A(V)$ a \emph{noncommutative vector bundle} over $A.$

\begin{lem}[\cite{cq} Proposition 5.3 (3)]\label{lem2.1}
Let $A$ be smooth and let $V$ be a finitely generated negatively graded projective bimodule over $A$. Then the algebra $T_A(V)$ is smooth.
\end{lem}

For a bimodule $M$ over a dga $A$ let $\r{Der}(A,M)$ denote the graded vector space of super-derivations from $A$ to $M$ and set $\r{Der}(A)=\r{Der}(A,A)$. Giving $A \otimes A$ the outer A-bimodule structure we let $\m D\r{er}(A) =  \r{Der}(A, A \otimes A)$ be the bimodule of \emph{double derivations} on $A$, where the $A$-bimodule structure on $\m D\r{er}(A)$ is induced via the inner bimodule structure on $A \otimes A$.  Note there is a natural isomorphism $\m{D}\r{er}(A) \cong (\Omega^1 A)^\vee$. Let $D:A \rightarrow \Omega^1A$ be the canonical derivation that sends $a \mapsto a \otimes 1 - 1 \otimes a.$ Then we can define the dga $(\Omega^\bullet A,D)$ of \emph{noncommutative differential forms} of $A$ by
$$\Omega^\bullet A = T_A(\Omega^1 A)$$
with differential induced by $D$. We next define the super-commutator quotient
$$\r{DR}(A) = \Omega^\bullet A/ [\Omega^\bullet A, \Omega^\bullet A]$$
called the \emph{cyclic quotient} of $\Omega^\bullet A$, which is a differential graded vector space. Note that both $\Omega^\bullet A$ and $\r{DR}(A)$ are bigraded- they have the usual tensor grading as well as a grading induced from $A.$ Denote by $\Omega^i A$ and $\r{DR}^i(A)$ the $i$th graded parts with respect to the tensor grading.

\begin{lem}\label{lem2.2}
Let $A$ be a finitely generated algebra over $k$ and let $V$ be a finitely generated $A$-bimodule. Then $\Omega^1 T_A(V)$ is generated by homogeneous elements of degree 0 and 1 as a graded $T_A(V)$-bimodule. 
\end{lem}

\begin{proof}
Because $T_A(V)$ is a free algebra $\Omega^1 T_A(V)$ is generated as a vector space by homogeneous elements of the form
$$x= (x_1 \otimes_A \ldots \otimes_A x_i) \otimes (x_{i+1} \otimes_A \ldots \otimes_A x_n) - (x_1 \otimes_A \ldots \otimes_A x_{i-1}) \otimes (x_i \otimes_A \ldots \otimes_A x_n) \in (\Omega^1 T_A(V))^n$$
for $x_m \in V$ and $1 \leq i \leq n$. Then
$$(x_1 \otimes_A \ldots \otimes_A x_{i-1}) \cdot (x_i \otimes 1 -1 \otimes x_i) \cdot (x_{i+1} \otimes_A \ldots \otimes_A x_n) = x$$
and so because $(x_1 \otimes_A \ldots \otimes_A x_{i-1}), \,(x_{i+1} \otimes_A \ldots \otimes_A x_n) \in T_A(V)$ and $(a_i \otimes 1 -1 \otimes a_i) \in (\Omega^1 T_A(V))^1$ we get the result.
\end{proof}

For $\theta \in \r{Der}(A)$ define the degree 0 derivation $L_\theta: \Omega^\bullet \rightarrow \Omega^\bullet$ by
\begin{align*}
    A \ni a &\mapsto \theta(a) \\
   \Omega^1A \ni D(a) &\mapsto D(\theta(a))
\end{align*}
extended to the rest of $\Omega^\bullet A$ using the Leibniz rule and linearity. We also define the \emph{contraction mapping} $i_\theta: \Omega^\bullet \rightarrow \Omega^\bullet$ as the super-derivation given by
\begin{align*}
    A \ni a &\mapsto 0 \\
   \Omega^1A \ni D(a) &\mapsto \theta(a).
\end{align*}
$i_\theta$ has degree $-1$ with respect to the tensor grading and degree 0 with respect to the induced grading by $A$. The derivations $L_\theta$ and $i_\theta$ are related by the \emph{Cartan identity}
\begin{align}
    L_\theta= D \circ i_\theta + i_\theta \circ D \,. \label{eq:1.001}
\end{align}
Both derivations descend to maps $\r{DR}(A) \rightarrow \r{DR}(A)$; indeed if $x \in \Omega^m A$ and $y \in \Omega^n A$ then we have
\begin{align*}
    L_\theta([x,y]) &= L_\theta\big(xy - (-1)^{mn}\, yx\big) \\
    &= L_\theta(x)y + x\, L_\theta(y) - (-1)^{mn}\, L_\theta(y)x - (-1)^{mn}\, y\, L_\theta(x) \\
    &= L_\theta(x)y -(-1)^{mn}\, y\, L_\theta(x) + x\, L_\theta(y) - (-1)^{mn}\, L_\theta(y)x \\
    &= [L_\theta(x),y] + [x, L_\theta(y)]
\end{align*}
and
\begin{align*}
    i_\theta([x,y]) &= i_\theta\big(xy - (-1)^{mn}\, yx\big) \\
    &= i_\theta(x)y + (-1)^m\, x\,i_\theta(y) - (-1)^{mn}\,i_\theta(y)x - (-1)^{mn}(-1)^n\, y\, i_\theta(x) \\
    &= i_\theta(x)y - (-1)^{(m+1)n}\, y\, i_\theta(x) + (-1)^m\, x\,i_\theta(y) - (-1)^{mn}\,i_\theta(y)x  \\
    &= i_\theta(x)y - (-1)^{(m-1)n}\, y\, i_\theta(x) + (-1)^m\Big( x\,i_\theta(y) - (-1)^{m(n-1)}\,i_\theta(y)x\Big) \\
    &= [i_\theta(x),y] + (-1)^m[x, i_\theta(y)].
\end{align*}
For $\lambda \in \m D\r{er}(A)$ we can also define a \emph{contraction mapping} $i_\lambda : \Omega^\bullet A \rightarrow \Omega^\bullet A \otimes \Omega^\bullet A$ as the double super-derivation of degree $-1$ given by
\begin{align*}
    a \in\, A &\mapsto 0 \\
     D(a) \in \, \Omega^1A &\mapsto \lambda(a).
\end{align*}
The \emph{reduced contraction mapping} $\iota_\lambda: \Omega^\bullet A \rightarrow \Omega^\bullet A$ is then defined as $$\iota_\lambda = m(\beta \circ i_\lambda)$$
where $m: \Omega^\bullet A \otimes \Omega^\bullet A \rightarrow \Omega^\bullet A$ is the multiplication map and $\beta$ is the flip isomorphism on $\Omega^\bullet A \otimes \Omega^\bullet A$ which sends homogeneous elements $x \in \Omega^m A$ and $y \in \Omega^n A$ to
$$x \otimes y \longmapsto (-1)^{mn} y \otimes x.$$
The reduced contraction mapping $\iota_\lambda$ descends to a map
$$\iota_\lambda: \r{DR}(A) \rightarrow \Omega^\bullet A$$
(see [\cite{cbeg} Lemma 2.8.6 (i)]).

\begin{dfn*}
Let $\omega \in \r{DR}^2(A)$ be closed with respect to the differential $D$. Define a map\\ 
\mbox{$i^\omega: \r{Der}(A) \rightarrow \r{DR}^1(A)$} by sending
$$\r{Der}(A) \ni \theta \mapsto i_\theta(\omega).$$
We call $\omega$ \emph{symplectic} if $i^\omega$ is an isomorphism. Similarly we can define a map $\iota^\omega : \m D\r{er}(A) \rightarrow \Omega^1 A$ and call $\omega$ \emph{bisymplectic} if $\iota^\omega$ is an isomorphism.
\end{dfn*}

From [\cite{cbeg} section 4.2] we have the following result.

\begin{lem}\label{lem2.3}
Let $A$ be smooth and $\omega$ be a bisymplectic 2-form. Then $\omega$ is symplectic.
\end{lem}

For $\omega$ bisymplectic and any $a \in A$ we can define a derivation $\{a,-\}_\omega: A \rightarrow A$ by
$$b \mapsto m((\iota^\omega)^{-1}(D(a))(b))$$
which makes sense because, as $\omega$ is bisymplectic, $(\iota^\omega)^{-1}: \Omega^1 A \rightarrow \m D\r{er}(A)$ is a well-defined map. Similarly if $\omega$ is symplectic then for $W \in A/[A,A]$ we can define a derivation in $\r{Der}(A)$ by
$$\{W,-\}_\omega := (i^\omega)^{-1}(D(W)).$$

\begin{dfn*}
An algebra $A$ is \emph{connected} if the sequence
$$0 \rightarrow k \rightarrow \r{DR}^0(A) \rightarrow \r{DR}^1(A)$$
is exact.
\end{dfn*}

Let $\Delta:A \rightarrow A \otimes A$ be the distinguished double derivation that sends $a \mapsto a \otimes 1 - 1 \otimes a$. Then there exists a map
$$\mu_{\r{nc}}: \overline{\r{DR}^2(A)} \rightarrow A/k$$
where $\overline{\r{DR}^2(A)} \subset \r{DR}^2(A)$ are the closed cyclic 2-forms, having the property that
$$D(\mu_{\r{nc}}(\omega))= \iota_\Delta(\omega)$$
in $\Omega^1 A$. $\mu_{\r{nc}}$ is defined as the composition
$$\overline{\r{DR}^2(A)} \xrightarrow{\iota_\Delta} \overline{\Omega^1 A} \xrightarrow{=} \widetilde{\Omega^1 A} \xrightarrow{D^{-1}} A/k$$
where $\widetilde{\Omega^1 A}$ are the exact 1-forms (see [\cite{cbeg} section 4]). Let $\b{w} \in A$ be a representative of $\mu_{\r{nc}}(\omega)$.

\begin{dfn*}
Let $(A,\omega,\xi)$ be a triple of a non-positively graded smooth algebra $A$ such that $\Omega^1 A$ is generated by homogeneous elements of degree $0,-1, \ldots, c$ as a graded $A$-bimodule, a bisymplectic 2-form $\omega \in \r{DR}^2(A)$ which is homogeneous of degree $c$ with respect to the induced grading of $A$, and a super-derivation $\xi \in \r{Der}(A)$ of degree 1 such that $\xi^2=0$, \,  $L_\xi(\omega)=0$, and $\xi(\b{w})=0$. We define the \emph{Ginzburg differential graded algebra} (Gdga) $\f D(A,\omega,\xi)$ of this triple to be the following data:
\begin{itemize}
    \item The underlying algebra of $\f D(A,\omega,\xi)$ is the free product of algebras $A * k[t]$ where $t$ has degree $c-1.$
    \item The differential $d$ is given by
    \begin{align*}
        A \ni a &\mapsto \xi(a) \\
        t &\mapsto \b{w}.
    \end{align*}
\end{itemize}
\end{dfn*}

\begin{rem}\label{rem2.4}
\begin{enumerate}
    \item[]
    \item As $\mu_{\r{nc}}(\omega) \in A/k$ the representative $\r{\b{w}}$ is only determined up to some constant in $k$. Thus we choose $\r{\b{w}}$ to be homogeneous (of degree $\r{deg}(\omega) = c$) so it is unique. This can cause a small issue when $c=0$ but this will not be relevant to us; see [\cite{d3} Remark 4.3.1] for more details. \\
    \item That $d^2=0$ follows from the assumptions $\xi^2=0$ and $\xi(\b{w})=0$. One way of ensuring $\xi(\r{\b{w}}) =0$ is to assume the additional condition $k \cap \xi(A)= 0$. By [\cite{cbeg} Proposition 4.1.3] we have for any $\theta \in \r{Der}(A)$
    $$\mu_{\r{nc}}(L_\theta(\omega))=L_\theta(\mu_{\r{nc}}(\omega))$$
    hence
    $$\xi(\mu_{\r{nc}}(\omega)) = L_\xi(\mu_{\r{nc}}(\omega)) = \mu_{\r{nc}}(L_\xi(\omega)) = \mu_{\r{nc}}(0) =0.$$
    It follows that $\xi(\b{w}) \in k$ and so if $k \cap \xi(A)= 0$ we get $\xi(\b{w}) =0$.
    \item Since $A$ is smooth $\omega$ is also symplectic by \cref{lem2.3}.
\end{enumerate}
\end{rem}

\begin{dfn*}
A \emph{superpotential algebra} $B$ is an algebra (viewed as a dga concentrated in degree 0) that is quasi-isomorphic to a Gdga $\f D(A,\omega, \xi)$ in which $A$ is connected and $\xi=\{W,-\}_\omega$ for some $W \in A/[A,A].$ W is called a \emph{potential}.
\end{dfn*}

We shall see some examples of Gdgas when we look at quivers and Jacobi algebras.

\begin{thm}[\cite{gin} Theorem 3.6.4]\label{thm2.5}
Let $\f D(A,\omega, \xi)$ be a Gdga with $\omega$ of degree $c$ and suppose that $H^i(\f D(A,\omega,\xi)) =0$ for all $i \neq 0.$ Then $H^0(\f D(A,\omega,\xi))$ is Calabi-Yau of dimension $-c+2$.
\end{thm}

\begin{thm}[\cite{d3} Theorem 4.3.8]\label{thm2.6}
Let $\f D(A,\omega, \xi)$ be as in \cref{thm2.5} and also suppose that $A$ is connected. Then $H^0(\f D(A,\omega,\xi))$ is exact Calabi-Yau of dimension $-c+2$. In particular a superpotential algebra is exact Calabi-Yau.
\end{thm}

\subsection{Ginzburg differential graded algebras in the relative case}
We can strengthen \cref{thm2.5} by looking at a smaller class of Gdgas arising in dimension 3. First we must extend the previous story of noncommutative algebras over a field $k$ to noncommutative algebras over a finite dimensional semisimple $k$-algebra $R$. Almost everything follows in a natural way to the relative case. We introduce notation with ``$R$" in the subscript to distinguish this change. Recall unadorned tensor products are to be taken over the base field $k$. We shall explore examples of these relative Gdgas in the next section on quivers and Jacobi algebras.

Let $A$ be an $R$-algebra, then
$$\Omega^1_R A = \r{Ker}(A \otimes_R A \xrightarrow{m} A)$$
is the bimodule of \emph{relative 1-forms}, and similarly
$$\Omega_R^\bullet A = T_A(\Omega_R^1 A)$$
is the dga of \emph{relative noncommutative differential forms}, with its cyclic quotient
$$\r{DR}_R(A) = \Omega_R^\bullet A/[\Omega_R^\bullet A,\Omega_R^\bullet A].$$
An important difference is in the bimodule of \emph{relative double derivations} $\m D\r{er}_R(A)$ which are derivations $\lambda : A \rightarrow A \otimes A$ such that $\lambda \circ (R \rightarrow A) = 0.$ This is defined so that we again have a natural isomorphism $\m D\r{er}_R(A) \cong (\Omega^1_R A)^\vee$. From now on we let $R=k I$ be the algebra of $k$-valued functions on a finite indexing set $I$. For each $i \in I$ let $1_i \in kI$ be the function that sends $i \mapsto 1$ and everything else to 0, and let $e= \sum_{i \in I} 1_i \otimes 1_i \in R \otimes R$. Define $\delta \in \m D\r{er}_R(A)$ as the distinguished relative double derivation given by
$$A \ni a \longmapsto a \cdot e - e \cdot a \in A \otimes A.$$
Then $\delta$ is the appropriate relative counterpart to the distinguished double derivation $\Delta$ we had in Section 2.1.

As in the non-relative case we have a map $\mu_{\r{nc}} : \overline{\r{DR}_R^2(A)} \rightarrow (A/R)^R$, where $(A/R)^R$ is the centraliser of $R$ in $A/R$, defined by the property that
$$D(\mu_{\r{nc}}(\omega)) = \iota_\delta(\omega)$$
in $\Omega_R^1 A$. $\mu_{\r{nc}}$ is given by the composition
$$\overline{\r{DR}_R^2(A)} \xrightarrow{\iota_\Delta} (\overline{\Omega^1 A})^R \xrightarrow{=} (\widetilde{\Omega^1 A})^R \xrightarrow{D^{-1}} (A/R)^R$$
(again see [\cite{cbeg} section 4]). Taking $\b{w} \in A$ to be a representative of $\mu_{\r{nc}}(\omega)$ we can define a relative Gdga analogously to the non-relative case.

\begin{dfn*}\label{def2.11}
Let $A$ be a  non-positively graded smooth $R$-algebra such that $\Omega_R^1 A$ is generated by homogeneous elements of degree $0,-1, \ldots, c$ as a graded $A$-bimodule, $\omega \in \r{DR}_R^2(A)$ a bisymplectic 2-form of degree $c$, and $\xi \in \r{Der}_R(A)$ a super-derivation of degree 1 such that $\xi^2=0$,  $L_\xi(\omega)=0$ and $\xi(\b{w})=0$. The \emph{Ginzburg differential graded algebra} $\f D(A,\omega,\xi)$ is the free product of $R$-algebras $A *_R R[t]$ with $t$ in degree $c-1$ and differential $d$ given by
\begin{align*}
    A \ni a &\mapsto \xi(a) \\
    t &\mapsto \b{w}.
\end{align*}
\end{dfn*}

Following [\cite{gin} section 5] let $F$ be a smooth algebra over $R$ and let $\alpha \in \r{DR}_R^1(F)$ be a cyclic 1-form such that
\begin{align}
    D(\alpha)=0 \quad \r{and} \quad \iota_\delta(\alpha)=0. \label{eq:2.1}
\end{align}
Set $A$ to be the tensor-algebra $T_F(\m D\r{er}_R(F))$ with non-positive grading given by negating the natural tensor grading. $A$ is smooth by [\cite{cbeg} Theorem 5.1.1]. There is a canonical closed cyclic bisymplectic 2-form $\omega \in \r{DR}_R^2(A)$ of degree -1 (see [\cite{cbeg} Theorem 5.1.1 and Proposition 5.4.1]) and by \cref{lem2.2} $\Omega_R^1 A$ is generated by homogeneous elements of degree 0 and $-1$. Let $\xi_\alpha$ be the super-derivation which is the image of $\alpha$ under the following composition 
$$\r{DR}_R^1(F) \hookrightarrow \r{DR}_R^1(A) \xrightarrow[\sim]{(i^\omega)^{-1}} \r{Der}_R(A)$$
where the first inclusion is induced by the map $F \xrightarrow{\sim} A^0 \hookrightarrow A.$

\begin{lem}\label{lem2.7}
With the algebra A, bisymplectic 2-form $\omega$ and super-derivation $\xi_\alpha$ as above we have
$$\xi_\alpha^2=0 \quad \r{and} \quad L_{\xi_\alpha}(\omega)=0.$$
\end{lem}

\begin{proof}
In \cite{gin} in the proof of Proposition 5.2.4 we are given that $\xi_\alpha$ sends
\begin{align}
    A^0=F \ni f &\mapsto 0 \label{eq:1.002}\\
    A^1= \m D \r{er}_R(F) \ni \theta &\mapsto \iota_\theta(\alpha) \in A^0. \nonumber
\end{align}
It follows that (omitting the $F$ subscript in the tensor product to reduce clutter)
$$A^n \ni \theta_1 \otimes \ldots \otimes \theta_n \xmapsto{\,\,\, \xi_\alpha \,\,\,} \sum_{i=1}^n (-1)^{i-1} \theta_1 \otimes \ldots \otimes \theta_{i-1} \otimes
\iota_{\theta_i}(\alpha)\, \theta_{i+1} \otimes \ldots \otimes \theta_n \in A^{n-1}$$
hence
\begin{align*}
    \xi_\alpha^2(\theta_1 \otimes \ldots \otimes \theta_n) &= \sum_{i=1}^n (-1)^{i-1} \bigg( \sum_{j < i} (-1)^{j-1}\, \theta_1 \otimes \ldots \iota_{\theta_j}(\alpha)\, \theta_{j+1} \ldots \iota_{\theta_i}(\alpha)\, \theta_{i+1} \ldots \otimes \theta_n \\
    &\qquad \qquad + \sum_{j > i} (-1)^{j}\, \theta_1 \otimes \ldots \iota_{\theta_i}(\alpha)\, \theta_{i+1} \ldots \iota_{\theta_j}(\alpha)\, \theta_{j+1} \ldots \otimes \theta_n \bigg) \\
    &=\sum_{i=1}^{n-1} \bigg( \sum_{j>i}^n (-1)^{i+j-1} \, \theta_1 \otimes \ldots \iota_{\theta_i}(\alpha)\, \theta_{i+1} \ldots \iota_{\theta_j}(\alpha)\, \theta_{j+1} \ldots \otimes \theta_n \\
    &\qquad \qquad + (-1)^{i+j-2}\, \theta_1 \otimes \ldots \iota_{\theta_i}(\alpha)\, \theta_{i+1} \ldots \iota_{\theta_j}(\alpha)\, \theta_{j+1} \ldots \otimes \theta_n \bigg) \\
    &=0.
\end{align*}
Then because the elements $\theta_1 \otimes \ldots \otimes \theta_n$ generate $A^n$ over $F$ we have that $\xi_\alpha^2(x)=0$ for any $x \in A^n$ and any $n \in \m N$.

To show $L_{\xi_\alpha}(\omega)=0$, by definition we have $\xi_\alpha=(i^\omega)^{-1}(\alpha)$ and so $i_{\xi_\alpha}(\omega)= \alpha$. Hence by the Cartan identity \eqref{eq:1.001} we get
$$L_{\xi_\alpha}(\omega)= D(i_{\xi_\alpha}(\omega)) + i_{\xi_\alpha}(D(\omega)) = D(\alpha) + 0 =0$$
using the assumption in \eqref{eq:2.1} that $D(\alpha)=0$ and the fact that $\omega$ is closed.
\end{proof}

\begin{dfn*}\label{def2.12}
Let $F$ be a smooth $R$-algebra and $\alpha \in \r{DR}_R^1(F)$ a cyclic 1-form satisfying the conditions in \eqref{eq:2.1}. Define the \emph{Gdga associated to the data $F, \alpha$} to be
$$\f D(F,\alpha) =\f D(A, \omega, \xi_\alpha)$$
where $A  = T_F(\m D\r{er}_R(F))$, and $\omega$ and $\xi_\alpha$ are as above.
\end{dfn*}

\begin{rem}\label{rem2.8}
As per the definition of a Gdga we still need to establish the condition $\xi_\alpha(\b{w}) = 0$. From [\cite{gin} Proposition 5.2.4 and Remark 5.2.5] we do not necessarily have the additional condition $R \cap \xi_\alpha(A) = 0$ analogous to $k \cap \xi(A)=0$ which we had in the non-relative setting in \cref{rem2.4}. However $\xi_\alpha(\b{w}) = 0$ instead follows from the assumption $\iota_\delta(\alpha)=0$ in \eqref{eq:2.1}. Indeed, from \eqref{eq:1.002} we have that $\xi_\alpha(\delta)=\iota_\delta(\alpha) = 0$. Then by [\cite{cbeg} Theorem 5.1.1] $\delta$ is a representative of $\mu_{\r{nc}}(\omega)$ giving the required result.
\end{rem}

This smaller class of Gdgas have some nice properties which were studied by Ginzburg in [\cite{gin} section 5.3]. In particular he showed that the zeroth cohomology $H^0(\f D(F, \alpha))$ of the Gdga is Calabi-Yau of dimension 3 if and only if the so-called \emph{completed} dga $\widehat{\f D}(F, \alpha)$ is acyclic. Alternatively if $\f D(F, \alpha)$ has an additional strictly positive grading (i.e. $(\f D(F, \alpha))_0 = k$) which is preserved by the differential $d$, then $H^0(\f D(F, \alpha))$ is Calabi-Yau of dimension 3 if and only if $\f D(F, \alpha)$ is acyclic. Later on it was shown by Keller and Van den Bergh that $\f D(F, \alpha)$ is in fact itself Calabi-Yai of dimension 3.

\begin{thm}[\cite{kel} Theorem A.12]\label{thm2.9}
The Gdga $\f D(F, \alpha)$ is Calabi-Yau of dimension 3.
\end{thm}

\subsection{Quivers and Jacobi algebras}
A \i{quiver} $Q$ is a pair $(Q_0,Q_1)$ consisting of a finite set of vertices $Q_0$, and a finite set of directed edges or arrows $Q_1$ between those vertices. We allow loops and cycles in our quivers. For each quiver we have two maps, $s,t: Q_1 \rightarrow Q_0$, the \i{source} and \i{target} maps that send an arrow to its source vertex or target vertex respectively.
For $i \in Q_0$ let $e_i$ denote the constant path in $Q$ at the vertex $i$. 

\begin{dfn*}
The \i{path algebra} $k Q$ of the quiver $Q$ is the algebra obtained by taking the free $k$-algebra over all the constant paths $e_i$ and all the arrows in $a \in Q_1$, modulo the relations that if two paths do not concatenate then their product is 0.
\end{dfn*}

\begin{dfn*}
A \i{representation} of a quiver $Q$ is a vector space $V$ that comes with a decomposition $V= \bigoplus_{i \in Q_0} V_i$ and linear maps $f_a: V_{i} \rightarrow V_{j}$ for each arrow $a:i \rightarrow j \in Q_1$.
\end{dfn*}

We shall focus on finite dimensional representations which can be grouped via a dimension vector $n = (n_i) \in \m N^{Q_0}$ given by the decomposition, i.e. $n_i = \r{dim}(V_i)$ for each $i$. Let $\r{Rep}_n(Q)$ denote the stack of $n$-dimensional representations of $Q$. Explicitly we have the following global quotient stack description
$$\r{Rep}_n(Q) \cong \prod_{a:i \rightarrow j \in Q_1} \r{Mat}_{n_i\times n_j} (k)/\prod_{l \in Q_0} \r{GL}_{n_l}$$
where $(g_l) \in \prod_{l \in Q_0} \r{GL}_{n_l}$ acts on $(f_a) \in \prod_{a\in Q_1} \r{Mat}_{n_i\times n_j} (k)$ via $(g_l)  \cdot (f_a) = (g_j f_a g_i^{-1}).$

To get more interesting stacks of representations we may also consider quivers with potential. Recall a \i{potential} is an element $W \in kQ/[kQ, kQ]$ and in this case can be written as a sum of cycles in $Q$ up to cyclic permutation. Given a potential $W$ we can consider its ``noncommutative derivatives" $\partial W/\partial a$ which can be defined by taking all the cycles in $W$ that contain the arrow $a$, cyclically permuting $a$ to the front of those cycles and then deleting $a$. Concretely let $\f d: kQ/[kQ,kQ] \rightarrow kQ$ be the map that sends a cycle 
$$a_na_{n-1} \ldots a_1 \mapsto \sum_j a_ja_{j-1} \ldots a_1 a_n \ldots a_{j+1},$$
and for $a \in Q_1$ let $a^{-1}:kQ \rightarrow kQ$ send the path
$$p=b_m b_{m-1} \ldots b_1 \mapsto 
\begin{cases}
b_{m-1} \ldots b_1 &\r{if} \,\, b_m=a \\
0 &\r{otherwise}
\end{cases}$$
then we define $\partial W/ \partial a$ as $a^{-1} \circ \f d(W)$. These derivatives give us an elements in the path algebra and so we obtain the ideal 
$$I_W = (\partial W/\partial a \,:\, a \in Q_1).$$

\begin{dfn*}
The \i{Jacobi algebra} of the quiver $Q$ with potential $W$ is the quotient algebra
$$\r{Jac}(Q,W) = kQ/I_W.$$
\end{dfn*}

The potential $W$ also describes a natural map $\r{Tr}(W)_n : \r{Rep}_n(Q) \rightarrow k$ that sends a representation $(f_a)$ to the trace of the matrix $W(f_a)$.

\begin{prop}[\cite{seg} Proposition 3.8]\label{prop2.10}
There is an isomorphism of stacks
$$\r{Rep}_n(\r{Jac}(Q,W)) \cong \r{crit}(\r{Tr}(W)_n).$$
\end{prop}

Unfortunately not all Jacobi algebras are superpotential algebras or even Calabi-Yau algebras but there is a natural Gdga associated to them. The following comes from [\cite{gin} section 4.2]. We work in the relative setting and let $R=kQ_0$ be the $k$-algebra over the vertex set of the quiver $Q$ or equivalently the $k$-algebra over the constant paths. The constant paths $e_i$ then take the role of the functions $1_i$ mentioned in Section 2.2.

\begin{lem}\label{lem2.11}
Let $V$ be the vector space $k Q_1$ viewed in the natural way as an $R$-bimodule. Then we have an isomorphism of algebras
\begin{align}
    kQ \cong T_R(V). \label{eq2.2}
\end{align}
In particular $kQ$ is a smooth $R$-algebra.
\end{lem}

\begin{proof}
The isomorphism is immediate from the definition of the multiplication on the path algebra versus on the tensor algebra. For smoothness, clearly $R$ is smooth and as $V$ is finite dimensional it is projective as an $R$-bimodule hence by \cref{lem2.1} $T_R(V)$ is smooth. 
\end{proof}

Given a quiver $Q$ with potential $W$ consider the new quiver $\widehat Q$ which has vertex set $\widehat Q_0 = Q_0$ and arrows
\begin{align*}
    &a : i \rightarrow j \quad\,\, \r{of degree  0 for any}\,\,\,\, a : i \rightarrow j \in Q_1 \\
    &a^* : j \rightarrow i \quad \r{of degree -1 for any}\,\,\,\, a : i \rightarrow j \in Q_1 \\
    \r{loops} \,\,\, &t_i : i \rightarrow i \quad\,\, \r{of degree -2 for any}\,\,\,\, i \in Q_0.
\end{align*}

\begin{dfn*}
The \emph{Gdga associated to $(Q,W)$}, denoted by $\Gamma_{Q,W}$, is the path algebra $k \widehat Q$ along with the differential $d$ that sends
\begin{align*}
    &a \mapsto 0 \\
    &a^* \mapsto \partial W/\partial a \\
    &t_i \mapsto e_i\Big(\sum_{a \in Q_1} [a,a^*]\Big)e_i.
\end{align*}
\end{dfn*}

\begin{prop}
$\Gamma_{Q,W}$ is a Gdga in the sense of \cref{def2.11}.
\end{prop}

\begin{proof}
Let $\overline{Q}$ be the double quiver of $Q$, then the underlying algebra of $\Gamma_{Q,W}$ is $k\overline{Q} *_R R[t]$ where $\sum_{i \in Q_0} t_i = t$. As $k \overline{Q}$ is a path algebra of a quiver it is smooth by \cref{lem2.11}. If $p$ is a path in $k \overline{Q}$ and we write $p=p_1xp_2$ for some arrow $x \in \overline{Q}_1$, then $\Omega^1_R k \overline{Q}$ is generated as a $R$-bimodule by
$$p_1 x \otimes_R p_2-p_1 \otimes_R xp_2$$
for all paths $p$ and arrows $x$. Hence, because
$$p_1 \cdot (x \otimes_R 1 -1 \otimes_R x) \cdot p_2 = p_1 x \otimes_R p_2-p_1 \otimes_R xp_2$$
$\Omega^1_R k \overline{Q}$ is generated as a $k \overline{Q}$-bimodule by the homogeneous elements of degree 0 and -1. From [\cite{cbeg} section 8.1 and Proposition 8.1.1] we have that the 2-form $\omega$ is \footnote{note this is $-1 \cdot$ the 2-form given in \cite{cbeg} due to a sign error in [\cite{cbeg} Lemma 3.1.1 (ii)]}
$$-\sum_{a \in Q_1} D(a) \cdot D(a^*) = \sum_{a \in Q_1} D(a^*) \cdot D(a)$$
which has degree $-1$, and the super-derivation is $\xi=\{W,-\}_\omega$.

We first show that $\omega$ is bisymplectic. Let $\lambda \in \m D\r{er}_R(k \overline{Q})$ be a double derivation defined on arrows $x \in \overline{Q}_1$ by
$$\lambda(x) = \sum_r c_r^x p_r^x \otimes q_r^x$$
where $c_r^x \in k$ and $p_r^x,\, q_r^x$ are paths in $k\overline{Q}$. Note that for each term in $\lambda(x)$ we have $s(x)=s(q_r^x)$ and $t(x)=t(p_r^x)$ for the paths $p_r^x, q_r^x$. We consider the map $\iota^\omega$ that sends $\lambda \mapsto \iota_\lambda(\omega).$ Then
\begin{align}
    \iota_\lambda(\omega) &= m(\beta \circ i_\lambda(\omega)) \nonumber \\ 
    &= m \Big(\beta \Big(\sum_{a \in Q_1} -\lambda(a) \cdot D(a^*) + D(a) \cdot \lambda(a^*) \Big) \Big) \nonumber \\
    &=  m \Big(\beta \Big(\sum_{a \in Q_1} -\Big( \sum_r c_r^a p_r^a \otimes q_r^a \Big) \cdot D(a^*) + D(a) \cdot \Big(\sum_s c_s^{a^*} p_s^{a^*} \otimes q_s^{a^*} \Big) \Big) \Big) \nonumber \\
    &= m \Big(\beta \Big(\sum_{a \in Q_1} -\Big( \sum_r c_r^a p_r^a \otimes (q_r^a D(a^*)) + \sum_s c_s^{a^*}(D(a)\, p_s^{a^*}) \otimes q_s^{a^*} \Big) \Big) \Big) \nonumber \\
    &= \sum_{a \in Q_1} -\Big( \sum_r c_r^a q_r^a D(a^*)\,p_r^a  - \sum_s c_s^{a^*}q_s^{a^*} D(a)\, p_s^{a^*}\Big) \nonumber \\
    &= \sum_{a \in Q_1} -\Big( \sum_r c_r^a (q_r^a a^* \otimes_R p_r^a -q_r^a \otimes_R a^*p_r^a)  - \sum_s c_s^{a^*}(q_s^{a^*} a \otimes_R p_s^{a^*}-q_s^{a^*} \otimes_R ap_s^{a^*})\Big). \label{eq:2.3}
\end{align}
For surjectivity of $\iota^\omega$, as $\Omega^1_R k \overline{Q}$ is generated by $p_1 x \otimes_R p_2-p_1 \otimes_R xp_2$ for all paths $p=p_1 x p_2$ and arrows $x \in k \overline{Q}_1$, if we define $\lambda_{p,x} \in \m D\r{er}_R(k \overline{Q})$ by $\lambda_{p,x}(x^*)=p_2 \otimes p_1$ and $\lambda(y)=0$ for all other $y \in \overline{Q}_1$ (\i{notation: $(x^*)^*=x$}) then we can see that $\iota_{\lambda_{p,x}}(\omega)=p_1x \otimes_R p_2 - p_1 \otimes_R xp_2$ using equation \eqref{eq:2.3}. For injectivity suppose that $\iota_\lambda(\omega)=0$. Each simple tensor element in the sum given in equation \eqref{eq:2.3} is non-zero since $s(q_r^x)=s(x)=t(x^*)$ and $t(p_r^x)=t(x)=s(x^*)$. Then because simple tensor elements give a basis of $k \overline{Q} \otimes_R k \overline{Q}$, for each $x_0 \in \overline{Q}_1$ and each $r_0$ with paths $p_{r_0}^{x_0}, q_{r_0}^{x_0}$ in $k \overline{Q}$ given by $\lambda(x_0)$ there must exist at least one $x_1 \in \overline{Q}_1$ and some $r_1$ with paths $p_{r_1}^{x_1}, q_{r_1}^{x_1}$ in $k \overline{Q}$ given by $\lambda(x_1)$ such that
$$q_{r_0}^{x_0} x_0^* \otimes_R p_{r_0}^{x_0}=q_{r_1}^{x_1} \otimes_R x_1^* p_{r_1}^{x_1}$$
in order for the overall coefficient of the term $q_{r_0}^{x_0} x_0^* \otimes_R p_{r_0}^{x_0}$ in equation \eqref{eq:2.3} to be 0. Let the length of the path $p_{r_0}^{x_0}$  be $l_1$ and the length of $q_{r_0}^{x_0}$ be $l_2$. Then the length of $p_{r_1}^{x_1}=l_1-1$ and the length of $q_{r_1}^{x_1}=l_2+1$. Now equation \eqref{eq:2.3} also contains the non-zero term $c_{r_1}^{x_1} q_{r_1}^{x_1} x_1^* \otimes_R p_{r_1}^{x_1}$. Hence we repeat this for $x_1$ to find a $x_2 \in \overline{Q}_1$ and $r_2$ such that
$q_{r_1}^{x_1} x_1^* \otimes_R p_{r_1}^{x_1}=q_{r_2}^{x_2} \otimes_R x_2^* p_{r_2}^{x_2}$ where the length of $p_{r_2}^{x_2}=l_1-2$ and the length of $q_{r_2}^{x_2}=l_2+2$. We continue to repeat this argument and after $l_1$ successive iterations we find an arrow $x_{l_1} \in \overline{Q}_1$ and some $r_{l_1}$ such that the length of $p_{r_{l_1}}^{x_{l_1}} = 0$ i.e. $p_{r_{l_1}}^{x_{l_1}}$ is just a constant path in $\overline{Q}$. However now we have the non-zero term $c_{r_{l_1}}^{x_{l_1}} q_{r_{l_1}}^{x_{l_1}} x_{l_1}^* \otimes_R p_{r_{l_1}}^{x_{l_1}}$ in equation \eqref{eq:2.3} but there cannot exist some $y \in \overline{Q}_1$ some $s$ and paths $p_s^y, q_s^y$ in $k \overline{Q}$ such that
$$q_{r_{l_1}}^{x_{l_1}} x_{l_1}^* \otimes_R p_{r_{l_1}}^{x_{l_1}}= q_s^y \otimes_R y^* p_s^y$$
because $y^* p_s^y$ cannot be a constant path. It follows that the coefficient of the term $q_{r_{l_1}}^{x_{l_1}} x_{l_1}^* \otimes_R p_{r_{l_1}}^{x_{l_1}}$ in equation \eqref{eq:2.3} is $c_{r_{l_1}}^{x_{l_1}} \neq 0$ which contradicts our initial assumption that $\iota_\lambda(\omega)=0$.

Next we check the differential $d$. In the definition of the Gdga the differential is given on $A$ by the super-derivation $\xi$, so we want to explicitly verify that for all $x \in \overline{Q}_1$ we have $d(x)=\xi(x)= \{W,x\}_\omega$. $\xi$ is defined as the unique super-derivation such that $i_\xi(\omega)=D(W)$ so write $W= \sum_{t=1}^m a_{t,n_t} \ldots a_{t,1}$ for arrows $a_{t,j} \in \overline{Q}_1$. Then
$$\frac{\partial W}{\partial a} = \sum_{\substack{t,j |\\ a_{t,j}=a}} a_{t,j-1} \ldots a_{t,1}\, a_{t,n_t} \ldots a_{t,j+1}.$$
Now because we are in the quotient $\r{DR}_R^1(k \overline{Q})$ we can cyclically permute terms (up to a sign, which in this case is always $+1$ since we are dealing with terms of degree 0 and a term of degree 1 in $\Omega_R^\bullet k \overline{Q}$), hence we get that
\begin{align*}
D(W) &= \sum_{t=1}^m \bigg(\sum_{j=1}^{n_t} a_{t,n_t} \ldots a_{t,j+1}\, D(a_{t,j})\, a_{t,j-1} \ldots a_{t,1} \bigg)\\
&= \sum_{t=1}^m \bigg(\sum_{j=1}^{n_t} D(a_{t,j})\, a_{t,j-1} \ldots a_{t,1}\, a_{t,n_t} \ldots a_{t,j+1}\bigg)\\
&= \sum_{a \in Q_1} \bigg( \sum_{\substack{t,j |\\ a_{t,j}=a}} D(a_{t,j})\, a_{t,j-1} \ldots a_{t,1}\, a_{t,n_t} \ldots a_{t,j+1} \bigg)\\
&= \sum_{a \in Q_1} D(a) \frac{\partial W}{\partial a}.
\end{align*}

Then as
$$i_\xi(\omega) =\sum_{a \in Q_1} -\xi(a) D(a^*) + D(a) \xi(a^*)$$
by comparing coefficients between these two expressions we must have that $\xi(a)=0$ and $\xi(a^*) = \partial W/\partial a$ for any $a \in Q_1$, as required. 

We must also check the differential on $t$. To do this we calculate a representative $\b{w}$ of $\mu_{\r{nc}}(\omega)$. Recall $\mu_{\r{nc}}(\omega)$ was defined such that $D(\mu_{\r{nc}}(\omega)) = \iota_\delta(\omega)$ where $\delta: k \overline{Q} \rightarrow k \overline{Q} \otimes k \overline{Q}$ was the distinguished double derivation introduced in Section 2.2. So
\begin{align*}
    \iota_\delta(\omega) &= m(\beta \circ i_\delta(\omega)) \\
    &= m\Big(\beta\Big(-\sum_{a \in Q_1} \delta(a) \cdot D(a^*) - D(a) \cdot \delta(a^*) \Big)\Big) \\
    &= m\Big(\beta\Big(-\sum_{a \in Q_1} (a \otimes e_{s(a)} - e_{t(a)} \otimes a) \cdot (a^* \otimes_R 1 - 1 \otimes_R a^*)\\
    &\qquad \qquad \quad - (a \otimes_R 1 - 1 \otimes_R a) \cdot (a^* \otimes e_{s(a^*)} - e_{t(a^*)} \otimes a^*) \Big)\Big)\\
    &= m\Big(\beta\Big(-\sum_{a \in Q_1} a \otimes (a^* \otimes_R 1- e_{s(a)} \otimes_R a^*) - e_{t(a)} \otimes (a a^* \otimes_R 1 - a \otimes_R a^*) \\
    &\qquad \qquad \quad -(a \otimes_R a^* - 1 \otimes_R a a^*) \otimes e_{t(a)} + (a \otimes_R e_{t(a^*)} - 1 \otimes_R a) \otimes a^* \Big)\Big) \\
    &= -\sum_{a \in Q_1} a^* \otimes_R a - e_{s(a)} \otimes_R a^* a - a a^* \otimes_R e_{t(a)} + a \otimes_R a^* \\
    &\qquad \qquad- a \otimes_R a^* + e_{t(a)} \otimes_R a a^* + a^* a \otimes_R e_{s(a)} - a^* \otimes_R a \\
    &= \sum_{a \in Q_1} e_{s(a)} \otimes_R a^* a - a^* a \otimes_R e_{s(a)} + a a^* \otimes_R e_{t(a)} - e_{t(a)} \otimes_R a a^*
\end{align*}
while on the other hand
\begin{align*}
   D\big([a,a^*]\big) &= D(a)a^*+aD(a^*)-D(a^*)a-a^*D(a) \\
   &= (a \otimes_R e_{s(a)} - e_{t(a)} \otimes_R a) a^* + a (a^* \otimes_R e_{s(a^*)} - e_{t(a^*)} \otimes_R a^*) \\
   &\quad \,\, - (a^* \otimes_R e_{s(a^*)} - e_{t(a^*)} \otimes_R a^*) a - a^* (a \otimes_R e_{s(a)} - e_{t(a)} \otimes_R a) \\
   &= a \otimes_R a^* - e_{t(a)} \otimes_R a a^* + a a^* \otimes_R e_{t(a)} - a \otimes_R a^* \\
   &\quad \,\, - a^* \otimes_R a + e_{s(a)} \otimes_R a^* a - a^* a \otimes_R e_{s(a)} + a^* \otimes_R a \\
   &= e_{s(a)} \otimes_R a^* a - a^* a \otimes_R e_{s(a)} + a a^* \otimes_R e_{t(a)} - e_{t(a)} \otimes_R a a^*.
\end{align*}
It follows that a homogeneous representative of degree -1 of $\mu_{\r{nc}}(\omega)$ is $\sum_{a \in Q_1} [a,a^*]$ hence 
$$d(t) = \b{w} = \sum_{a \in Q_1} [a,a^*] = \sum_{i \in Q_0} d(t_i) = d\Big(\sum_{i \in Q_0} t_i\Big).$$
Finally we must check the conditions $\xi^2=0$,\, $L_\xi(\omega)=0$ and $\xi(\b{w}) =0$. The first equality is clear from the explicit description given above of how $\xi$ acts on arrows. For the second equality we have
\begin{align*}
    L_\xi(\omega) &= \sum_{a \in Q_1} L_\xi(D(a)) \cdot D(a^*) + D(a) \cdot L_\xi(D(a^*)) \\
    &= \sum_{a \in Q_1} D(\xi(a)) \cdot D(a^*) + D(a) \cdot D(\xi(a^*)) \\
    &= \sum_{a \in Q_1}  D(a) \cdot D\left(\frac{\partial W}{\partial a}\right) \\
    &= \sum_{a \in Q_1}  D^2(a) \cdot \frac{\partial W}{\partial a} + D(a) \cdot D\left(\frac{\partial W}{\partial a}\right) \\
    &= D^2(W) =0.
\end{align*}
For the third equality again write $W= \sum_{t=1}^m a_{t,n_t} \ldots a_{t,1}$. Then
\begin{align*}
    \xi(\b{w}) &= \xi \bigg(\sum_{a \in Q_1} [a,a^*] \bigg) = \sum_{a \in Q_1} \xi(aa^*)- \xi(a^*a) \\
    &=  \sum_{a \in Q_1} \xi(a)a^* + a\, \xi(a^*) -\xi(a^*)a - (-1) a^*\, \xi(a)\\
    &= \sum_{a \in Q_1} a\, \frac{\partial W}{\partial a} - \frac{\partial W}{\partial a}\, a \\
    &= \sum_{a \in Q_1} \bigg(\sum_{\substack{t,j |\\ a_{t,j}=a}} a_{t,j}\, a_{t,j-1} \ldots a_{t,1}\, a_{t,n_t} \ldots a_{t,j+1} \\
    &\qquad \quad- \sum_{\substack{s,i |\\ a_{s,i}=a}} a_{s,i-1} \ldots a_{s,1}\, a_{s,n_s} \ldots a_{s,i+1}\, a_{s,i} \bigg) \\
    &= \sum_{t,j} a_{t,j}\, a_{t,j-1} \ldots a_{t,1}\, a_{t,n_t} \ldots a_{t,j+1} \\
    &\quad\, - \sum_{s,i} a_{s,i-1} \ldots a_{s,1}\, a_{s,n_s} \ldots a_{s,i+1}\, a_{s,i} \\
    &=0.
\end{align*}
\end{proof}

\begin{prop}\label{prop2.13}
$\r{Jac}(Q,W)=H^0(\Gamma_{Q,W})$.
\end{prop}

\begin{proof}
Since $(\Gamma_{Q,W})_0= kQ$ and $(\Gamma_{Q,W})_{-1}= kQ\{ a^* : a \in Q_1 \}$ we get
\begin{align*}
  H^0(\Gamma_{Q,W})&= \frac{\r{Ker}\,\, d_0}{\r{Im}\,\, d_{-1}}  \\
  &=\frac{kQ}{( \partial W/\partial a : a \in Q_1 )} \\
  &=\r{Jac}(Q,W).
\end{align*}
\end{proof}

There is a natural identification of graded algebras $k \overline{Q} \cong T_{k Q}(\m D\r{er}_R(k Q))$ that sends the arrow $a^*$ to the double derivation $\lambda_{a^*}$ given on arrows by
$$Q_1 \ni x \longmapsto
\begin{cases}
1 \otimes 1 &\r{if} \,\, x=a \\
0 &\r{otherwise}.
\end{cases}$$
It is easy to see that under this identification the distinguished relative double derivation $\delta \in \m D\r{er}_R(k Q)$ is sent to $\sum_{a \in Q_1} [a,a^*] \in k \overline{Q}$. Setting the cyclic 1-form $\alpha \in \r{DR}_R^1(k Q)$ to be $D(W)$ we have that the super-derivation $\{W,-\}_\omega \in \r{Der}_R(k \overline{Q})$ corresponds to the super-derivation $\xi_\alpha \in \r{Der}_R(T_{k Q}(\m D\r{er}_R(k Q)))$. Indeed it suffices to check this for degree 0 and -1 elements. The degree 0 parts of $k \overline{Q}$ and $T_{k Q}(\m D\r{er}_R(k Q))$ are both isomorphic to $k Q$ and both super-derivations send degree 0 elements to 0. Then in degree -1 we have $\{W,a^*\}_\omega = \partial W/\partial a$ which under the identification above corresponds to sending $\lambda_{a^*} \mapsto \partial W/\partial a$. Hence we must show that $\xi_\alpha(\lambda_{a^*})= \iota_{\lambda_{a^*}}(D(W))$ is equal to $\partial W/\partial a$. Then
\begin{align*}
    \iota_{\lambda_{a^*}}(D(W)) &= m \circ \beta\big(i_{\lambda_{a^*}}(D(W)\big) \\
    &= m \circ \beta\bigg(i_{\lambda_{a^*}}\bigg(\sum_{b \in Q_1}D(b)\frac{\partial W}{\partial b}\bigg)\bigg) \\
    &= m \circ \beta\bigg(\sum_{b \in Q_1}\lambda_{a^*}(b)\frac{\partial W}{\partial b}\bigg) \\
    &= m \circ \beta\bigg(1 \otimes_R 1 \cdot \frac{\partial W}{\partial a}\bigg) \\
    &= \frac{\partial W}{\partial a}.
\end{align*}
We also have $\iota_\delta(\alpha)=0$ from \cref{rem2.8}. It follows that the Gdga $\Gamma_{Q,W}$ is naturally isomorphic to the Gdga $\f D(kQ, D(W))$ described in \cref{def2.12}. The proceeding result then follows immediately from \cref{thm2.9}.

\begin{prop}
The Gdga $\Gamma_{Q,W}$ is Calabi-Yau of dimension 3.
\end{prop}

\begin{prop}[\cite{am} Proposition 2.3]\label{prop2.15}
There exists a natural $t$-structure for the derived category $D(\Gamma_{Q,W})$  whose heart is equivalent to $\r{Jac}(Q,W)$-$\r{\b{Mod}}$.
\end{prop}

\cref{prop2.15} allows us to view the moduli stack of representations $\r{Rep}_n(\r{Jac}(Q,W))$ as a moduli stack in a 3 Calabi-Yau category regardless of whether the algebra $\r{Jac}(Q,W)$ is 3 Calabi-Yau or not. This is important for us to be able to define DT invariants. If however $\r{Jac}(Q,W)$ is in fact Calabi-Yau and $W$ is homogeneous we get something stronger. There is a natural grading on the path algebra $k Q$ of a quiver given by the path length, which makes the isomorphism \eqref{eq2.2} an isomorphism of graded algebras. If $W$ is homogeneous of degree $m$ then this descends to a grading on $\r{Jac}(Q,W)$ too. This gives us a second grading on $\Gamma_{Q,W}$, which is strictly positive and preserved by the differential $d$ if we set the degree of the dual arrows $a^*$ to be $m-1$ and degree of the loops $t_i$ to be $m$. By [\cite{cbeg} Proposition 8.1.1] and its proof we have that the path algebra $k \overline{Q}$ of the doubled quiver is connected. Then by \cref{prop2.15}, and [\cite{gin} Proposition 3.7.7 and Theorem 5.3.1] we get the following: 

\begin{thm}[cf. \cite{gin} Corollary 5.4.3]\label{thm2.16}
Let $W$ be a homogeneous potential with respect to the path grading on $kQ$. Then $\r{Jac}(Q,W)$ is Calabi-Yau of dimension 3 if and only if it is a superpotential algebra.
\end{thm}

\subsection{Brane tilings}
\begin{dfn*}
A \i{brane tiling} $\Delta$ of a genus $g$ Riemann surface $\Sigma_g$ is an embedding $\Gamma \hookrightarrow \Sigma_g$ of a bipartite graph $\Gamma$ such that each connected component of $\Sigma_g \setminus \Gamma$ is simply connected. We choose a partition of the vertex set of $\Gamma$ into 2 disjoint subsets of black and white vertices, such that every edge in $\Gamma$ goes between a black vertex and a white vertex.
\end{dfn*}

From a brane tiling we can obtain a quiver with potential $(Q_\Delta, W_\Delta)$. The underlying graph of $Q_\Delta$ is the dual graph to $\Gamma$ in $\Sigma_g$ and it is directed so that arrows in $Q_\Delta$ go clockwise around a white vertex and anticlockwise around a black vertex. For a vertex $v \in \Gamma$ let $c_v$ denote the minimal cycle in $Q_\Delta$ that goes around it, i.e. $c_v$ consists of all the arrows that are dual to the edges that come out of $v$. Then we take the potential to be
$$W_\Delta = \sum_{v \, \r{white}} c_v - \sum_{u \, \r{black}} c_u.$$

\begin{eg}\label{eg2.1}
The following genus 2 example will be used throughout the following sections and will be useful to keep in mind to picture what's going on. It can also easily be extended to higher genus surfaces.
A brane tiling $\Delta$ of $\Sigma_2$ and its dual quiver $Q_\Delta$ are given in the following picture.
\begin{figure}[h]
    \centering
    \includegraphics[scale=0.5]{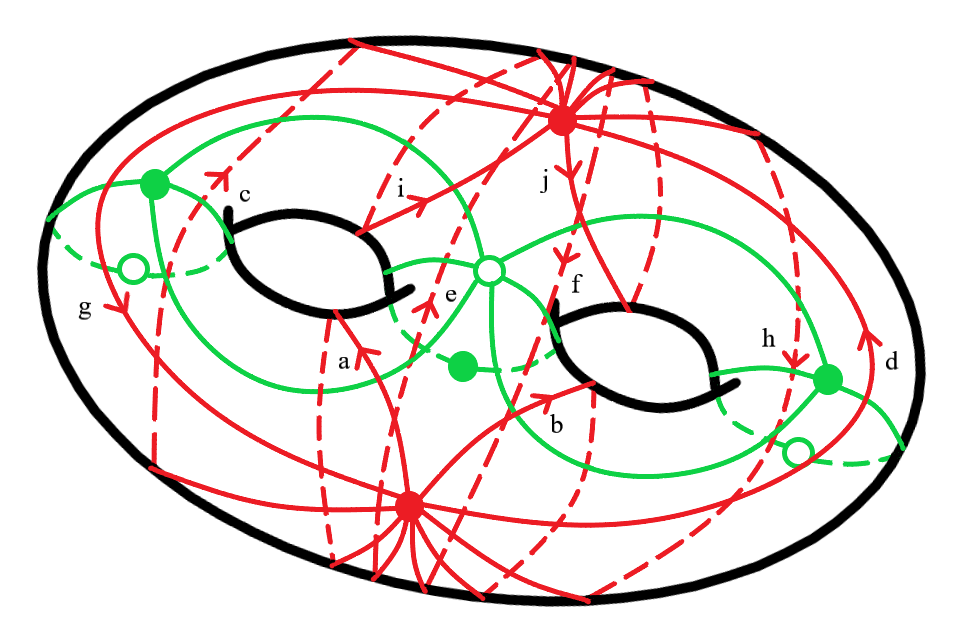}
    \caption{A brane tiling $\Delta$ of $\Sigma_2$ in green with its dual quiver $Q_\Delta$ in red.}
\end{figure}
Explicitly the dual quiver $Q_\Delta$ is
\[
\begin{tikzcd}[row sep=large, column sep=3cm]
1 \arrow[out=90, in=150, swap, loop,"a"] \arrow[out=210, in=270, swap, loop,"b"] \arrow[r, bend left= 30,"c"] \arrow[r, shift left=1.0ex, bend left=45,"d"] \arrow[r, shift left=2ex, bend left=60,"e"]
& 2 \arrow[out=90, in=30, loop,"i"] \arrow[out=330, in=270, loop,"j"] \arrow[l,bend left= 30, "h"] \arrow[l, shift left=1ex, bend left=45,"g"] \arrow[l, shift left=2ex, bend left=60,"f"]
\end{tikzcd}
\]
with potential 
$$W_\Delta=abfjie+gc+hd-agic-bhjd-fe.$$
\end{eg}

The idea behind using brane tilings in the context of fundamental groups of Riemann surfaces is as follows (from \cite{d1}). If the arrow $a \in Q_{\Delta,1}$ is dual to the edge between the white vertex $v$ and black vertex $u$ in $\Delta$ then $a \cdot \partial W_\Delta/\partial a = c_v - c_u.$ This tells us that in $\r{Jac}(Q_\Delta, W_\Delta)$ we can identify the paths $c_v$ and $c_u$, which are homotopic in $\Sigma_g$ via a homotopy which is spanned by the edge from the brane tiling. This allows us to link the algebra over the fundamental group of the Riemann surface $\Sigma_g$ with a Jacobi algebra. However not all homotopic paths in $\m C Q_\Delta$ are identified in $\r{Jac}(Q_\Delta, W_\Delta)$, for example the minimal cycles $c_v$ are null-homotopic but are not equal to any constant paths $e_i$ in $\r{Jac}(Q_\Delta, W_\Delta)$.

\begin{dfn*}
A dimer $D$ for a brane tiling $\Delta$ is a choice of edges in $\Delta$ such that every vertex in $\Delta$ is an endpoint of exactly one edge in $D$.
\end{dfn*}

Dimers are useful because after dualising they give us a collection of arrows in $Q_\Delta$ with the property that every cycle in $W_\Delta$ contains exactly one arrow from this set. This is straightforward from the definition of the dimer and the potential $W_\Delta$; if $a$ is dual to the edge in $D$ that goes between the vertices $v$ and $u$ then $a$ will appear in the minimal cycles $c_v$ and $c_u$ in $W_\Delta$, and since vertices in $\Delta$ are endpoints to exactly one edge in $D$ no other arrow in the dual of $D$ will appear in $c_v$ or $c_u$ and every term in $W$ contains one of these arrows.

\subsection{Fundamental group algebras}
This final preliminary section gives an overview of existing results relating algebras over fundamental groups to the aforementioned Calabi-Yau and superpotential properties, motivating the conjecture we aim to partially solve.

\begin{dfn*}
A manifold $X$ is called \emph{acyclic} if its universal cover is contractible.
\end{dfn*}

For the rest of this section let $X$ be a compact, acyclic, orientable manifold of dimension $n$ and let $k[\pi_1(X)]$ be its fundamental group algebra. 

\begin{thm}[\cite{d3} Theorem 5.2.2]
The fundamental group algebra $k[\pi_1(X)]$ is homologically finite.
\end{thm}

Let $LX$ denote the \emph{free loop-space} of $X$ i.e. the space of continuous maps $S^1 \rightarrow X.$ As path components of $LX$ are in one-to-one correspondence with conjugacy classes $c \in \pi_1(X)$, let $(LX)_c$ denote the path component corresponding to $c.$ Denote by $const: X \rightarrow LX$ the map that sends a point to the constant loop at that point, let $bp: LX \rightarrow X$ denote the map that sends a loop to its basepoint, and let $const_*$ and $bp_*$ denote the respective pushforwards in homology. 

\begin{thm}[\cite{jon} Theorem 6.2]
For all $d \in \m N$ we have a natural isomorphism
$$\r{HH}_d(k[\pi_1(X)]) \cong H_d(LX).$$
\end{thm}

Recall for an algebra $A$ that an element $\nu \in \r{HH}_d(A)$ is called non-degenerate if induces a (quasi-) isomorphism $A \cong A^\vee[d]$. There is an action of the centre $Z(A)$ on $\r{HH}_*(A)$ induced by $Z(A)$ acting on the first copy of $A$ in $A \otimes_{A^e} A$.

\begin{prop}[\cite{d3} Proposition 5.2.6]\label{prop2.19}
An element $\nu \in \r{HH}_n(k[\pi_1(X)])$ is non-degenerate if and only if
$\nu = z \cdot const_*([X])$ for some central unit $z \in k[\pi_1(X)].$ In particular $k[\pi_1(X)]$ is Calabi-Yau of dimension $n.$
\end{prop}

\begin{thm}[\cite{d3} Theorem 6.1.3]\label{thm2.20}
$k[\pi_1(X)]$ is exact Calabi-Yau of dimension $n$ only if it contains a non-trivial central unit $z.$
\end{thm}

The idea behind the proof of \cref{thm2.20} is as follows. First, from \cite{jon}, we have an isomorphism of long exact sequences (where $A=k[\pi_1(X)]$)
\[\begin{tikzcd}
\cdots \arrow[r] & \r{HC}_n(A) \arrow[r] \arrow[d] & \r{HC}_{n-2}(A) \arrow[r, "\partial"] \arrow[d] & \r{HH}_{n-1}(A) \arrow[r] \arrow[d] & \r{HC}_{n-1}(A) \arrow[r] \arrow[d] & \cdots \\
\cdots \arrow[r] & H^{S^1}_n(LX) \arrow[r] & H^{S^1}_{n-2}(LX) \arrow[r, "\partial'"] & H_{n-1}(LX) \arrow[r] & H^{S^1}_{n-1}(LX) \arrow[r] & \cdots
\end{tikzcd}\]
Hence an exact Calabi-Yau structure of dimension $d$ on $A$ is equivalent to a non-degenerate element $\eta \in H_{d}(LX)$ which is in the image of $\partial'.$ By \cref{prop2.19} a non-degenerate $\eta$ must be of the form $\eta=z \cdot const_*([X])$ for a central unit $z \in A$. Now there exists a natural grading of both $H^{S^1}_{d-1}(LX)$ and $H_{d}(LX)$ by conjugacy classes in $\pi_1(X)$ (the grading exists at the level of chains) which is preserved by $\partial'$, so the existence of some $\lambda \in H^{S^1}_{d-1}(LX)$ such that $\partial'(\lambda)=\eta$ is equivalent to equalities $\partial(\lambda_c) = a_cs_c \cdot const_*([X])$ where $c$ is a conjugacy class in $\pi_1(X)$, $a_c \in k$, $s_c=\sum_{g \in c} g$,  and $\lambda_c \in H^{S^1}_{d-1}((LX)_c)$ such that $\eta= \sum_c a_cs_c \cdot const_*([X])$. 
If we were to take the central unit $z=1$ then we would need to look for some $\lambda_{0} \in H^{S^1}_{d-1}((LX)_{0})$ such that $\partial'(\lambda_0)=const_*([X]).$ But, from [\cite{d3} Theorem 6.1.3] and its proof, it turns out the map $H^{S^1}_{d-1}((LX)_{0}) \rightarrow H_d(LX)$ is 0 hence 1 does not give an exact Calabi-Yau structure.

\begin{eg}[\cite{d3} Corollary 6.2.3 and Corollary 6.2.4]
Let $X$ be a compact hyperbolic manifold of dimension $> 1.$ Then $k[\pi_1(X)]$ is not a superpotential algebra. Indeed one can show that $k[\pi_1(X)]$ has trivial centre, hence by \cref{thm2.20} it is not exact Calabi-Yau. But \cref{thm2.6} says that every superpotential algebra is exact Calabi-Yau.
\end{eg}

We saw earlier that when $A$ is connected a Gdga with cohomology concentrated in degree 0 will have an exact Calabi-Yau cohomology algebra. \cref{thm2.20} describes the manifolds which cannot have an exact Calabi-Yau fundamental group algebra. This gives a rudimentary justification for looking for superpotential descriptions for the fundamental group algebras of circle bundles; often circle bundles have a natural non-trivial central element in their fundamental groups. In particular for the dimension 3 case we will be looking for isomorphisms between the fundamental group algebra and the Jacobi algebra of a quiver with potential. Then because our fundamental group algebras are always Calabi-Yau (by \cref{prop2.19}) we get using \cref{thm2.16} that the Jacobi algebra presentation is a superpotential algebra provided that the potential is homogeneous. As a starting point of this conjecture we have:

\begin{thm}[\cite{d1} Proposition 4.2]\label{thm2.22}
Let $\Delta$ be a brane tiling of a Riemann surface $\Sigma_g$ and let $\m C \widetilde{Q}_\Delta$ denote the localisation of the path algebra $\m C Q_\Delta$ with respect to all the arrows in $Q_{\Delta, 1}$. Then we have an isomorphism of algebras
$$\r{Jac}(\m C \widetilde{Q}_\Delta,W_\Delta) \cong \r{Mat}_{r\times r}( \m C[\pi_1(\Sigma_g \times S^1)])$$
where $r$ is the number of vertices in $Q_\Delta.$
\end{thm}

As remarked earlier there is a natural way to think about how the relations in the Jacobi algebra of the brane tiling identify homotopic paths in the surface $\Sigma_g$, but also not all relations in the fundamental group can be obtained from the potential. Adding in this extra $S^1$ direction fully rectifies this issue whereby paths in the Jacobi algebra that would be null-homotopic in $\Sigma_g$ are instead sent to loops around the circle in $\Sigma_g \times S^1$. Explicitly we grade the arrows in $Q_\Delta$ such that the potential $W_\Delta$ is homogeneous of degree 1 (note this is entirely separate to the path grading we already have on the path algebra) and then define a new embedding $Q_\Delta \hookrightarrow \Sigma_g \times S^1$ where this grading determines how far the arrows go around the $S^1$ direction. In particular, for any genus $g$ surface it is possible to find a tiling $\Delta$ such that the potential $W_\Delta$ is homogeneous with respect to this grading (e.g. see [\cite{d3} Figure 2]).

If in addition $W_\Delta$ is homogeneous with respect to the path grading, since \cref{thm2.22} holds for any brane tiling, we get a superpotential description for the fundamental group algebra $\m C[\pi_1(\Sigma_g \times S^1)]$.

\begin{conj}\label{conj2.1}
Let $X$ be a compact, acyclic, orientable circle bundle of dimension $n$. Then $k[\pi_1(X)]$ is a superpotential algebra.
\end{conj}

Guided by \cref{thm2.22} we shall provide a Jacobi algebra description for the case of a mapping torus of a Riemann surface by a finite-order, orientation-preserving automorphism. Thereby we provide more evidence for \cref{conj2.1} if a path-graded homogeneous potential can be found.

\section{Main results}
\subsection{A Jacobi algebra presentation for $\m C [\pi_1(M_{g,\varphi})]$}
From now on we work over the field $\m C$.

Let $\Sigma_g$ be a genus $g$ Riemann surface and let $\varphi$ be an orientation-preserving automorphism of $\Sigma_g$ of order $n$. Let $M_{g,\varphi}= \Sigma_g \times [0,1]/ \sim_\varphi$ denote the mapping torus of $\varphi$.

Consider a brane tiling $\Delta$ of $\Sigma_g$ giving rise to a dual quiver $Q=Q_\Delta$ and potential $W=W_\Delta$ as described in Section 2.4. We assume that the brane tiling and the chosen colouring of the vertices in the tiling are preserved by $\varphi.$ This implies that $\varphi$ induces an automorphism on $Q$ and hence also on the path algebra $\m C Q$ (we denote this automorphism by $\varphi$ as well). We focus on the case in which the size of the orbit of every vertex in $Q$ under $\varphi$ is $n$.

\begin{lem}[\cite{ck} Proposition 3.3 1.]\label{lem3.0}
Let $f: D^2 \xrightarrow{\sim} D^2$ be an orientation-preserving automorphism of the disc $D^2$ of order $n$. Then
\begin{enumerate}[a)]
    \item the fixed point set of $f$ is a single point in the interior of $D^2$.
    \item the fixed point set of the composition $f^i$ is equal to the fixed point set of $f$ for all $i \neq 0 \,\, \r{mod} \,\, n$.
\end{enumerate}
\end{lem}

\begin{lem}\label{lem3.1}
Let $\varphi$ be an orientation-preserving automorphism of $\Sigma_g$ of order $n$. Then there exists a brane tiling $\Delta$ which is preserved by $\varphi$ such that every vertex in the dual quiver $Q = Q_\Delta$ has an orbit of size $n$ under the induced action of $\varphi$ on $Q$.
\end{lem}

\begin{proof}
A point in $p \in \Sigma_g$ will have orbit size strictly less than $n$ under $\varphi$ if and only if $p$ is a fixed point of $\varphi^d$ for some $0<d<n$ that divides $n$. Let $\Delta_0$ be a brane tiling of $\Sigma_g$ with colouring that is preserved by $\varphi$. Then a vertex $i$ in the dual quiver $Q_{\Delta_0}$ will have orbit size strictly less than $n$ if and only if the tile $S_i$ in $\Delta_0$ that is dual to $i$ is such that $\varphi^d(S_i) = S_i$ for some $0<d<n$ that divides $n$. Suppose that there exists a vertex $i \in Q_{\Delta_0,0}$ with orbit size less than $n$ and suppose that $d$ is the smallest positive integer such that $\varphi^d(S_i) = S_i$. Then $\varphi^d|_{S_i}: S_i \xrightarrow{\sim} S_i$ is an (orientation-preserving) automorphism of $S_i$. Therefore by \cref{lem3.0} a) it has a single fixed point $p$ in the interior of $S_i$, and because by \cref{lem3.0} b) the fixed point set of $\varphi^{dj}$ is the same point for all $j=1, \ldots, n/d-1$ every other point in $S_i$ has orbit size $n$ under $\varphi$.

We add new brane tiling vertices of the same colour at $p, \varphi(p), \ldots, \varphi^{d-1}(p)$, then choose an opposite colour vertex $v$ on the boundary of $S_i$ and add an edge between the new vertex at $p$ and $v$. We then add the images of this edge under $\varphi^j$ for all $j=1, \ldots, n-1$, which go between the vertices $\varphi^j(p)$ and $\varphi^j(v)$, to the brane tiling. This gives us a new brane tiling $\Delta_1$ that is also preserved by $\varphi$, and in which the tile $S_i$ in $\Delta_0$ has been subdivided into tiles $S_{i_1}, \ldots, S_{i_{n/d}}$ (see \cref{pic:3.0}). Indeed because the vertex $v$ has orbit size $n$ we have added $n/d$ distinct edges into the tile $S_i$ that only intersect at the point $p$ and so we do in fact end up with a brane tiling (each of the distinct tiles $S_i, \varphi(S_i), \ldots, \varphi^{d-1}(S_i)$ each have $n/d$ edges added, giving us a total of $n$ new edges added to $\Delta_0$ as expected). 

We claim that $\varphi^k(S_{i_j}) \neq S_{i_j}$ for all $i_j$ and all $0<k<n$, and hence the vertex $i \in Q_{\Delta_0,0}$ which had orbit size less than $n$ has been replaced by the vertices $i_j \in Q_{\Delta_1,0}$ all of orbit size $n$. So suppose not i.e. there exists some $j$ and $k$ such that $\varphi^k(S_{i_j})= S_{i_j}$. Note that $k$ must be multiple of $d$ since otherwise
$$\varphi^{md+l}(S_{i_j}) \subset \varphi^l(S_i) \neq S_i$$
for $0<l<d$. If $k=md$ then we have that $\varphi^{md}|_{S_{i_j}}: S_{i_j} \xrightarrow{\sim} S_{i_j}$ is an automorphism of $S_{i_j}$ and so again by \cref{lem3.0} a) it has a fixed point in the interior of $S_{i_j}$. But this would also be a fixed point of $\varphi^{md}|_{S_i}$ and it would be distinct from $p$ contradicting \cref{lem3.0} b).

We repeat this procedure for all such vertices in the dual quiver with orbit size less than $n$ giving us the required brane tiling $\Delta$.
\end{proof}

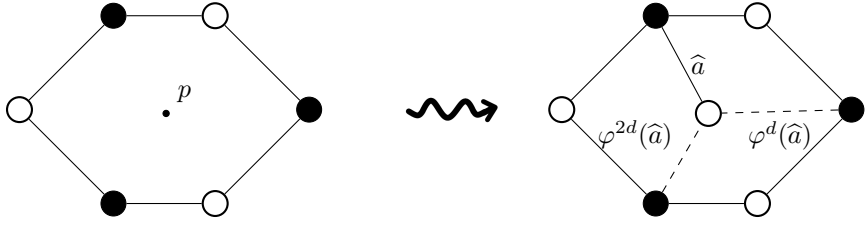
\begin{figure}[H]
\centering
\begin{subfigure}{0.3\textwidth}
\begin{tikzpicture}
[black/.style={circle, draw=black!120, fill=black!120, thin, minimum size=3mm},
white/.style={circle, draw=black!120, thick, minimum size=3mm},
empty/.style={circle, minimum size=1pt, inner sep=1pt},
blacksmall/.style={circle, draw=black!120, fill=black!120, thin, scale=0.25}]

\node[black] (1) {};
\node[white] (2) [right=of 1] {};
\node[black] (3) [below right=of 2] {};
\node[white] (4) [below left=of 3] {};
\node[black] (5) [left=of 4] {};
\node[white] (6) [below left=of 1] {};
\node[blacksmall,label=above right:{$p$}] (8) at (0.7,-1.3) {};

\draw[-] (1) -- (2);
\draw[-] (2) -- (3);
\draw[-] (3) -- (4);
\draw[-] (4) -- (5);
\draw[-] (5) -- (6);
\draw[-] (6) -- (1);
\end{tikzpicture}
\end{subfigure}
\begin{subfigure}{0.15\textwidth}
\centering
\includegraphics[scale=0.12]{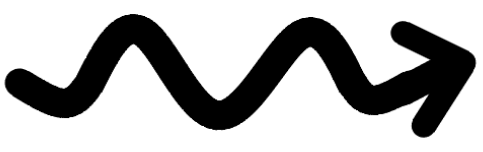}
\end{subfigure}
\begin{subfigure}{0.3\textwidth}
\begin{tikzpicture}
[black/.style={circle, draw=black!120, fill=black!120, thin, minimum size=3mm},
white/.style={circle, draw=black!120, thick, minimum size=3mm},
empty/.style={circle, minimum size=1pt, inner sep=1pt},
redsmall/.style={circle, draw=red!120, fill=red!120, thin, scale=0.6}]

\node[black] (1) {};
\node[white] (2) [right=of 1] {};
\node[black] (3) [below right=of 2] {};
\node[white] (4) [below left=of 3] {};
\node[black] (5) [left=of 4] {};
\node[white] (6) [below left=of 1] {};
\node[white] (8) at (0.7,-1.3) {};

\draw[-] (1) -- (2);
\draw[-] (2) -- (3);
\draw[-] (3) -- (4);
\draw[-] (4) -- (5);
\draw[-] (5) -- (6);
\draw[-] (6) -- (1);
\draw[-] (1) -- (8) node [right,midway] {$\widehat{a}$};
\draw[-,dashed] (3) -- (8) node [below, midway] {$\varphi^d(\widehat{a})$};
\draw[-,dashed] (5) -- (8) node [above left, midway] {$\varphi^{2d}(\widehat{a})$};
\end{tikzpicture}        
\end{subfigure}
    \caption{An example of a tile in $\Delta_0$ for which we suppose the induced action of $\varphi^d$ is clockwise rotation by $120^\circ$ and the tile has orbit size $d=n/3<n$. The fixed point of this rotation is the point $p$. We add a vertex at this point (in this case a white vertex), and then attach this to the brane tiling via an edge $\widehat{a}$. We then add all the images under $\varphi, \ldots, \varphi^{n-1}$ of both the new vertex and edge to $\Delta_0$. In the resulting brane tiling the three new tiles are no longer preserved by $\varphi^d$.}
    \label{pic:3.0} 
\end{figure}

\begin{rem}\label{rem3.2}
Take a brane tiling $\Delta$ such that every vertex in the dual quiver has an orbit of size $n$ under the induced action of $\varphi$ on $Q = Q_\Delta$. Then we may freely add additional vertices and edges to $\Delta$ along with their images under $\varphi, \ldots, \varphi^{n-1}$ and still maintain a brane tiling which is preserved by $\varphi$ such that the orbit of every vertex in the dual quiver has size $n$. Indeed, given such a $\Delta$ adding extra vertices and edges will amount to subdividing existing tiles in $\Delta$. From the proof of \cref{lem3.1} we see that if one of these new tiles $S_{i_j}$, which subdivides the old tile $S_i$, gives a dual vertex with orbit size less than $n$ then we must have that $\varphi^d(S_{i_j})=S_{i_j}$ for some $d<n$, and hence $\varphi^d|_{S_{i_j}}$ has a fixed point in $S_{i_j}$. Then as $\varphi$ and hence $\varphi^d$ preserves $\Delta$ we know that $\varphi^d(S_i)$ is also a tile in $\Delta$. But $\varphi^d$ has a fixed point inside $S_{i_j} \subset S_i$ and so $\varphi^d(S_i)$ must equal $S_i$ which contradicts $\Delta$ not having any dual vertices of orbit size less than $n$.
\end{rem}

Consider the semidirect product $S := \m C Q \rtimes_\varphi \m Z/n \m Z$, which has the following description 
$$\m CQ \rtimes_\varphi \m Z/n \m Z = \bigoplus_{l=0}^{n-1} \m CQ \times \{l\}$$ with multiplication given by
$$(a,l) \cdot (b,m) = (a\, \varphi^l(b),\,\, l+m)$$
where $a, b$ are paths in $\m CQ$ and $l,m$ and $l+m$ are elements in $\m Z/n \m Z.$

\begin{lem}\label{lem3.2}
$S$ is Morita equivalent to a localised path algebra of a quiver.
\end{lem}

\begin{proof}
A representation of $S$ is a representation of $Q$ with the extra data of the action of $\mathbbm 1 \in \m Z/n \m Z$, subject to the multiplication relations in $S$. Let $V=\bigoplus_{i\in Q_0} V_i$ and $(f_a)_{a \in Q_1}$ be the data of a representation of $Q$. The action of $\mathbbm 1$ is then given by isomorphisms $V_i \xrightarrow{r_i} V_{\varphi(i)}$ for each $i \in Q_0$ such that if we let $r:V \xrightarrow{\sim} V$ be the isomorphism given component-wise on $V_i$ by $r_i$, we have $r^n=\r{id}_V$. The data $(V,(f_a),(r_i))$ then gives a representation of $S$ if we also have that for all $a:i \rightarrow j \in Q_1, \,\,\, f_{\varphi(a)}=  r_j \circ f_a \circ r_i^{-1}.$ Equivalently if we partition the orbits of $\varphi$ in $Q$ as follows; for the vertices we have orbits $O_1=\{i_1, \varphi(i_1), \ldots, \varphi^{n-1}(i_1)\}, \, \ldots,\, O_y=\{i_y, \ldots, \varphi^{n-1}(i_y)\}$ and for the arrows we have orbits $P_1=\{a_1,\ldots, \varphi^{n-1}(a_1)\},\, \ldots,\, P_z=\{a_z, \ldots, \varphi^{n-1}(a_z)\}$, then the data of a representation of $S$ is given by the vector space $V=\bigoplus_{i\in Q_0} V_i$, linear maps $f_{a_t}$ for $t \in \{1, \ldots, z\}$ and isomorphisms $r_{i_u,0}: V_{i_u} \xrightarrow{\sim} V_{\varphi(i_u)}, \,\, r_{i_u,1}: V_{\varphi(i_u)} \xrightarrow{\sim} V_{\varphi^2(i_u)},\, \ldots,\,\, r_{i_u,n-2}: V_{\varphi^{n-2}(i_u)} \xrightarrow{\sim} V_{\varphi^{n-1}(i_u)} $ for $u \in \{1, \ldots, y\}$. 

Let $Q^\#$ denote the quiver with vertices $Q^\#_0 = Q_0$ and arrows $Q^\#_1 = \{a_1, \ldots, a_z,\, r_{i_u,t} \,:\, u=1, \ldots, y \,\, \r{and} \,\, t= 0, \ldots, n-2\}$, and let $\m CQ'$ denote the localisation of the path algebra $\m CQ^\#$ at the arrows $\{r_{i_u,t}\}$. Then it is clear that the data of a representation of $S$ is then exactly the same as the data of a representation of $\m CQ'$.
\end{proof}

We shall use the notation $Q'$ for the ``localised" quiver associated to the partially localised path algebra $\m CQ'$. We call the arrows in $Q^\# \subset Q'$ \emph{generating arrows} and the arrows $r_{i_u,t} \in Q'_1$ \emph{isomorphism arrows}.

\begin{eg}\label{eg3.1}
Consider a Riemann surface of genus 2 and let the automorphism $\varphi$ be rotation by $180^\circ$ around the $z$-axis through the centre of the surface. We take the brane tiling from \Cref{eg2.1}.
\begin{figure}[H]
    \centering
    \includegraphics[scale=0.4]{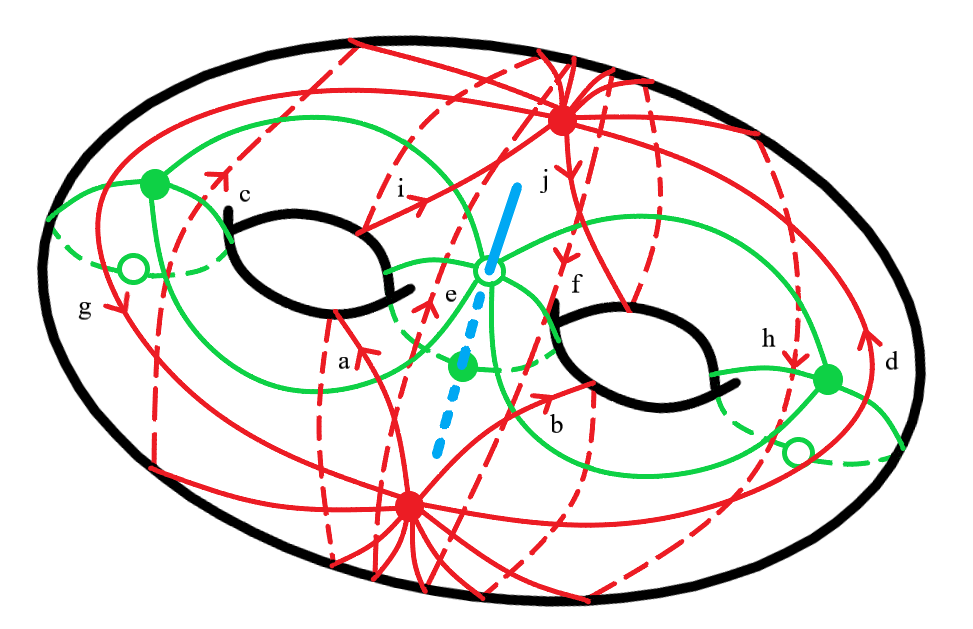}
    \caption{Let $\varphi$ be the rotation by $180^\circ$ around the blue axis.}
\end{figure}
Recall this gives us the following dual quiver $Q$
\[
\begin{tikzcd}[row sep=large, column sep=3cm]
1 \arrow[out=90, in=150, swap, loop,"a"] \arrow[out=210, in=270, swap, loop,"b"] \arrow[r, bend left= 30,"c"] \arrow[r, shift left=1.0ex, bend left=45,"d"] \arrow[r, shift left=2ex, bend left=60,"e"]
& 2 \arrow[out=90, in=30, loop,"i"] \arrow[out=330, in=270, loop,"j"] \arrow[l,bend left= 30, "h"] \arrow[l, shift left=1ex, bend left=45,"g"] \arrow[l, shift left=2ex, bend left=60,"f"]
\end{tikzcd}
\]
and potential $W=abfjie+gc+hd-agic-bhjd-fe$. $\varphi$ then swaps the two vertices 1 and 2 and also the arrows $a$ and $j$, $b$ and $i$, $c$ and $h$, $d$ and $g$, $e$ and $f$. The quiver $Q'$ corresponding to the algebra $\m CQ \rtimes \m Z/2 \m Z$ is
 \[
\begin{tikzcd}[row sep=large, column sep=3cm]
1 \arrow[out=90, in=150, swap, loop,"a"] \arrow[out=210, in=270, swap, loop,"b"] \arrow[r,"r^{-1}"] \arrow[r, bend left= 30,"c"] \arrow[r, shift left=1.0ex, bend left=45,"d"] \arrow[r, shift left=2ex, bend left=60,"e"]
& 2 \arrow[l, bend left= 30, "r"]
\end{tikzcd}
\]
as the arrows $a,b,c,d,e$ generate $Q$ under $\varphi$ and the arrows $r,r^{-1}$ give the isomorphism between the vector spaces at the two vertices.
\end{eg}

\begin{lem}\label{lem3.4}
There is an inclusion of algebras $\xi: \m CQ \hookrightarrow \m CQ'$.
\end{lem}

\begin{proof}
The map $\xi$ is induced by the natural inclusion $\m CQ \times \{0\} \hookrightarrow S$. As per the proof of \cref{lem3.2} we make a choice of a generating set of arrows $\{a_1, \ldots, a_z\}$ of $Q$ under $\varphi$ as well as a choice of isomorphism arrows $r_{i_u,t}:\varphi^t(i_u) \rightarrow \varphi^{t+1}(i_u)$ for $u = 1, \ldots, y$ and $t = 0, \ldots ,n-2$. Then define the map $\xi$ as follows; since $Q_0=Q_0'$ send the constant paths $e_i \mapsto e_i$, and then map the arrow $a \in Q_1$ to the path $p_a a_j q_a$ where $a$ is in the orbit under $\varphi$ of the generating arrow $a_j$ and $p_a: t(a_j) \rightarrow t(a), \, q_a:s(a) \rightarrow s(a_j)$ are paths comprised solely of the isomorphism arrows or their inverses (see \cref{pic:3.01} for an example).
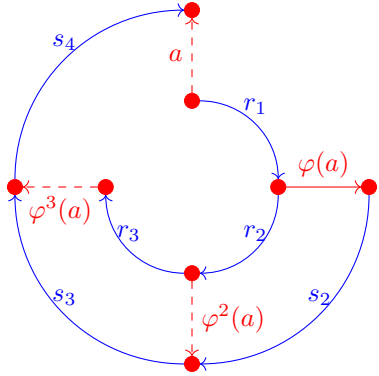
\begin{figure}[H]
    \centering
\begin{tikzpicture}
[black/.style={circle, draw=black!120, fill=black!120, thin, minimum size=3mm},
white/.style={circle, draw=black!120, thick, minimum size=3mm},
empty/.style={circle, minimum size=1pt, inner sep=1pt},
redsmall/.style={circle, draw=red!120, fill=red!120, thin, scale=0.6}]

\node[redsmall] (1) {};
\node[redsmall] (2) [above=of 1] {};
\node[redsmall] (3) [below right=of 1] {};
\node[redsmall] (4) [right=of 3] {};
\node[redsmall] (5) [below left=of 3] {};
\node[redsmall] (6) [below=of 5] {};
\node[redsmall] (7) [below left=of 1] {};
\node[redsmall] (8) [left=of 7] {};

\draw[red,dashed,->] (1) -- (2) node [left, midway] {$a$};
\draw[red,->] (3) -- (4) node [above, midway] {$\varphi(a)$};
\draw[red,dashed,->] (5) -- (6) node [right, midway] {$\varphi^2(a)$};
\draw[red,dashed,->] (7) -- (8) node [below, midway] {$\varphi^3(a)$};
\draw[blue,->] (1) to[out=0,in=90] node [above, midway] {$r_1$} (3);
\draw[blue,->] (3) to[out=-90,in=0] node [above, midway] {$r_2$} (5);
\draw[blue,->] (4) to[out=-90,in=0] node [above, midway] {$s_2$} (6);
\draw[blue,->] (5) to[out=-180,in=-90] node [above, midway] {$r_3$} (7);
\draw[blue,->] (6) to[out=-180,in=-90] node [above, midway] {$s_3$} (8);
\draw[blue,->] (8) to[out=90,in=180] node [above, midway] {$s_4$} (2);
\end{tikzpicture}
    \caption{An example of how $Q$ fits inside $Q'$. Let the red vertices and (dashed) arrows be in $Q$ with $\varphi$ being clockwise rotation by $90^\circ$. If we take the generating arrow of the orbit of $a$ to be $\varphi(a)$ and let the blue arrows denote the chosen isomorphism arrows, then $\xi(a) = s_4 s_3 s_2 \varphi(a) r_1$, $\xi(\varphi(a)) = \varphi(a)$, $\xi(\varphi^2(a)) = s_2 \varphi(a) r_2^{-1}$, and $\xi(\varphi^3(a)) = s_3 s_2 \varphi(a) r_2^{-1} r_3^{-1}$.}
    \label{pic:3.01}
\end{figure}
$\xi$ is well-defined because by construction of $Q'$ all such paths $a_j, \, p_a, \, q_a$ are unique and because both $a$ and $\xi(a)$ are paths $s(a) \rightarrow t(a)$ all the trivial relations in $\m CQ$ are satisfied. 

To see that $\xi$ is an inclusion it suffices to show that $\xi$ is injective on paths and that
$$\{\xi(p) : p \,\, \r{is a path in} \,\, \m CQ\}$$
is linearly independent. So first suppose we have paths $p,\, q \in \m CQ$ such that $\xi(p)=\xi(q).$ Write
\begin{align*}
    p &= b_n b_{n-1}\ldots b_2b_1 \\
    q &= c_m c_{m-1} \ldots c_2c_1
\end{align*} 
for arrows $b_i, c_j \in Q_1$. Then the equality $\xi(p)=\xi(q)$ implies that both $b_1$ and $c_1$ lie in the same orbit of some generating arrow $a_1$ i.e. 
$$b_1 = \varphi^{i_1}(a_1)\quad \r{and} \quad c_1=\varphi^{j_1}(a_1)$$
for some $i_1,j_1$. But as paths in $\m C Q'$ they must have the same source, and so by the assumption that the size of the orbit of each vertex in $Q$ under $\varphi$ has size $n$ we know that the source of $\varphi^{i_1}(a_1)$ and $\varphi^{j_1}(a_1)$ are equal only when $i_1=j_1$ and so $b_1=c_1$. We can then repeat this argument for the rest of the arrows in $p$ and $q$ giving us that $m=n$ and $b_i=c_i$ for all $i$, hence $p=q$ and $\xi$ is injective on paths. Then certainly we have
$$\{\xi(p) : p \,\, \r{is a path in} \,\, \m CQ\} \subseteq \{p' : p' \,\, \r{is a path in} \,\, \m CQ'\}$$
and since $\{p' : p' \,\, \r{is a path in} \,\, \m CQ'\}$ is a $\m C$-basis of $\m CQ'$ we get the required result.
\end{proof}

We grade the isomorphism arrows $r_{i_u,t}$ in $Q'$ with degree 1, their inverses $r_{i_u,t}^{-1}$ with degree $-1$, and all other arrows in $Q'$ with degree 0.

\begin{lem}\label{lem3.5}
For all arrows $a \in Q_1$ the element $\xi(a) \in \m CQ'$ either has degree 0, degree $n$ or degree $-n$.
\end{lem}

\begin{proof}
This is clear from the assumption that each orbit of vertices/arrows under $\varphi$ has size $n.$ Indeed if the arrow $a$ is in the orbit of the generating arrow $b$ then we can write $a=\varphi^k(b)$ for some $k.$ Therefore, depending upon the choice of the isomorphism arrows $r_{i_u,t}$ for each orbit of vertices, between the source vertices of $a$ and $b$ there will either be $k$ or $n-k$ isomorphism arrows and similarly between the target vertices of $a$ and $b$ there will either be $k$ or $n-k$ isomorphism arrows. Hence $\xi(a)$ is either going to be the composition of $k$ (resp. $n-k$) isomorphism arrows followed by $b$ followed by $k$ (resp. $n-k$) inverses and hence it will have degree 0, or it will be the composition of $k$ (resp. $n-k$) isomorphism arrows followed by $b$ followed by $n-k$ (resp. $k$) isomorphism arrows and hence it will have degree $n$ or it will be the composition of $k$ (resp. $n-k$) inverses arrows followed by $b$ followed by $n-k$ (resp. $k$) more inverse arrows and hence it will have degree $-n$ (see \cref{pic:3.02} for an illustration of each case).
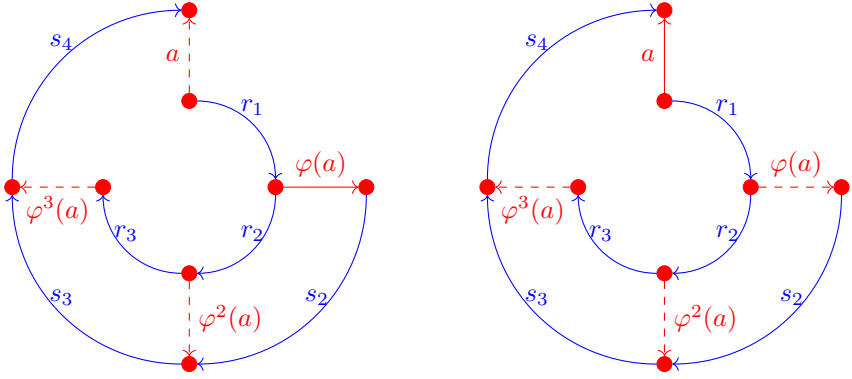
\begin{figure}
\begin{subfigure}{0.4\textwidth}
\begin{tikzpicture}
[black/.style={circle, draw=black!120, fill=black!120, thin, minimum size=3mm},
white/.style={circle, draw=black!120, thick, minimum size=3mm},
empty/.style={circle, minimum size=1pt, inner sep=1pt},
redsmall/.style={circle, draw=red!120, fill=red!120, thin, scale=0.6}]

\node[redsmall] (1) {};
\node[redsmall] (2) [above=of 1] {};
\node[redsmall] (3) [below right=of 1] {};
\node[redsmall] (4) [right=of 3] {};
\node[redsmall] (5) [below left=of 3] {};
\node[redsmall] (6) [below=of 5] {};
\node[redsmall] (7) [below left=of 1] {};
\node[redsmall] (8) [left=of 7] {};

\draw[red,dashed,->] (1) -- (2) node [left, midway] {$a$};
\draw[red,->] (3) -- (4) node [above, midway] {$\varphi(a)$};
\draw[red,dashed,->] (5) -- (6) node [right, midway] {$\varphi^2(a)$};
\draw[red,dashed,->] (7) -- (8) node [below, midway] {$\varphi^3(a)$};
\draw[blue,->] (1) to[out=0,in=90] node [above, midway] {$r_1$} (3);
\draw[blue,->] (3) to[out=-90,in=0] node [above, midway] {$r_2$} (5);
\draw[blue,->] (4) to[out=-90,in=0] node [above, midway] {$s_2$} (6);
\draw[blue,->] (5) to[out=-180,in=-90] node [above, midway] {$r_3$} (7);
\draw[blue,->] (6) to[out=-180,in=-90] node [above, midway] {$s_3$} (8);
\draw[blue,->] (8) to[out=90,in=180] node [above, midway] {$s_4$} (2);
\end{tikzpicture}
\end{subfigure}
\begin{subfigure}{0.4\textwidth}
\begin{tikzpicture}
[black/.style={circle, draw=black!120, fill=black!120, thin, minimum size=3mm},
white/.style={circle, draw=black!120, thick, minimum size=3mm},
empty/.style={circle, minimum size=1pt, inner sep=1pt},
redsmall/.style={circle, draw=red!120, fill=red!120, thin, scale=0.6}]

\node[redsmall] (1) {};
\node[redsmall] (2) [above=of 1] {};
\node[redsmall] (3) [below right=of 1] {};
\node[redsmall] (4) [right=of 3] {};
\node[redsmall] (5) [below left=of 3] {};
\node[redsmall] (6) [below=of 5] {};
\node[redsmall] (7) [below left=of 1] {};
\node[redsmall] (8) [left=of 7] {};

\draw[red,->] (1) -- (2) node [left, midway] {$a$};
\draw[red,dashed,->] (3) -- (4) node [above, midway] {$\varphi(a)$};
\draw[red,dashed,->] (5) -- (6) node [right, midway] {$\varphi^2(a)$};
\draw[red,dashed,->] (7) -- (8) node [below, midway] {$\varphi^3(a)$};
\draw[blue,->] (1) to[out=0,in=90] node [above, midway] {$r_1$} (3);
\draw[blue,->] (3) to[out=-90,in=0] node [above, midway] {$r_2$} (5);
\draw[blue,->] (4) to[out=-90,in=0] node [above, midway] {$s_2$} (6);
\draw[blue,->] (5) to[out=-180,in=-90] node [above, midway] {$r_3$} (7);
\draw[blue,->] (6) to[out=-180,in=-90] node [above, midway] {$s_3$} (8);
\draw[blue,->] (8) to[out=90,in=180] node [above, midway] {$s_4$} (2);
\end{tikzpicture}
\end{subfigure}
    \caption{Let $\varphi$ be clockwise rotation by $90^\circ$ on the quiver. The left-hand side is the same setup as we had in \cref{pic:3.01} with generating arrow for the orbit $\varphi(a)$. As we saw $\xi(a) = s_4 s_3 s_2 \varphi(a) r_1$ has degree 4 whilst $\xi(\varphi^2(a)) = s_2 \varphi(a) r_2^{-1}$ has degree 0 due to the choices of isomorphism arrows between the sources and targets of the arrows in this orbit. On the right-hand side we instead take the generating arrow of this orbit to be $a$ and choose the same isomorphism arrows. In this case we would have, for example, $\xi(\varphi(a)) = s_2^{-1} s_3^{-1} s_4^{-1} a r_1^{-1}$ which has degree $-4$.}
    \label{pic:3.02}
\end{figure}
\end{proof}

From now on when we say a path $p \in \m CQ'$ is made from isomorphism arrows we mean that the only arrows that appear in $p$ are the isomorphism arrows $r_{i_u,t}$ or their inverses $r_{i_u,t}^{-1}$.

\begin{lem}\label{lem3.6}
For all paths $p' \in \m CQ'$ we can write
$$p'= q\, \xi(p)$$
for some path $p \in \m CQ$ and a path $q \in \m CQ'$ made from isomorphism arrows. In particular if $\r{deg}(p') \equiv 0 \, \r{mod}\, n$ then $p' \in\, \r{Im}(\xi)$.
\end{lem}

\begin{rem}
Having the path of isomorphism arrows $q$ at the end is somewhat arbitrary, it is also possible to write $p'=\xi(p)q'$ or $p'=q'' \xi(p) q'''$. 
\end{rem}

\begin{proof}
Write
$$p'=r_{m+1}a_m \ldots a_2r_2a_1r_1$$
for arrows $a_i \in Q'_1$ coming from the generating set of $Q_1$ under $\varphi$, and paths of isomorphism arrows $r_i$. Let $\widetilde{a_1}$ be the arrow in $Q_1$ such that $r_1$ is the unique path of isomorphism arrows in $\m CQ'$ between $s(a_1)$ and $s(\widetilde{a_1})$ (in particular $\widetilde{a_1} = \varphi^{k_1}(a_1)$ for some $k_1$). Then $\xi(\widetilde{a_1})= s_1a_1r_1$ where $s_1$ is the unique path of isomorphism arrows in $\m CQ'$ between $t(a_1)$ and $t(\widetilde{a_1})$. Then let $\widetilde{a_2}$ be the arrow in $Q_1$ such that $r_2s_1^{-1}$ is the unique path of isomorphism arrows in $\m CQ'$ between $s(a_2)$ and $s(\widetilde{a_2})$, and write $\xi(\widetilde{a_2})=s_2a_2r_2s_1^{-1}$. Repeating this procedure for all the generating arrows $a_i$ in $p'$ it is clear that
$$p'=r_{m+1}s_m^{-1} \xi(\widetilde{a_m} \ldots \widetilde{a_1}).$$
If $\r{deg}(p') \equiv 0 \,\, \r{mod} \,\, n$ this implies that
$$\r{deg}(r_{m+1}s_m^{-1}) + \r{deg}(\xi(\widetilde{a_m} \ldots \widetilde{a_1})) \equiv 0 \,\, \r{mod} \,\, n.$$
By \cref{lem3.5} $\r{deg}(\xi(\widetilde{a_m} \ldots \widetilde{a_1})) \equiv 0 \,\, \r{mod} \,\, n$ and so the path of isomorphism arrows $r_{m+1}s_m^{-1}$ has degree a multiple of $n$. Now a path made of isomorphism arrows must go between vertices in $Q'$ in the same orbit, but in any orbit of a vertex there are only $n-1$ isomorphism arrows that connect the whole orbit. Hence the degree of $r_{m+1}s_m^{-1}$ is 0 and so it is just a constant path.
\end{proof}

It is clear that $\xi$ descends to a map $\m C Q/[\m C Q,\m CQ] \rightarrow  \m C Q'/[\m C Q', \m C Q']$ so let $W'$ be the image of $W$ under this morphism, giving a potential on $Q'$. Therefore we can construct the Jacobi algebra $\r{Jac}(Q',W')$. In our running example \Cref{eg3.1} $\xi$ sends
\begin{align*}
    a &\longmapsto a, \qquad \quad f \longmapsto rer \\
    b &\longmapsto b, \qquad \quad g \longmapsto rdr \\
    c &\longmapsto c, \qquad \quad h\longmapsto rcr \\
    d &\longmapsto d, \qquad \quad i\longmapsto r^{-1}br \\
    e &\longmapsto e, \qquad \quad j\longmapsto r^{-1}ar
\end{align*}
and 
\begin{align}
    W'= abreabre+2rdrc-2ardbrc-rere. \label{eq:3.91}
\end{align}
There is a choice involved in taking generators of $Q$ under $\varphi$ and the isomorphism arrows $r_{i_u,t}$ (or equivalently a choice of $\xi$) and hence the algebra $\m CQ'$ is not unique (but it is unique up to canonical isomorphism). We choose $\xi$ so that $W'$ does not contain both $r_{i_u,t}$ and $r_{i_u,t}^{-1}$, so that we may view it as a potential on the un-localised quiver algebra and hence the noncommutative derivative description of $\r{Jac}(Q',W')$ makes sense.

\begin{lem}\label{lem3.8}
There exists a choice of $\xi$ such that for each isomorphism arrow $r_{i_u,t} \in Q'_1$ the potential $W'$ does not contain both $r_{i_u,t}$ and $r_{i_u,t}^{-1}$. 
\end{lem}

\begin{proof}
Consider the orbit of vertices $\{i_u, \varphi(i_u), \ldots, \varphi^{n-1}(i_u)\}$ in $Q'$, where $r_{i_u,t}: \varphi^{t}(i_u) \rightarrow \varphi^{t+1}(i_u)$. If $W'$ were to contain both $r_{i_u,t}$ and $r_{i_u,t}^{-1}$ then there would need to be arrows $a,\,b,\,c,\,d \in Q_1$ such that $ba$ and $dc$ are part of cycles in $W$ and we would need to choose $\xi$ so that if we let $a_0,\,b_0,\,c_0,\,d_0$ be the corresponding generators of the orbits of the arrows then  $t(a_0)=\varphi^{k_1}(i_u)$, $s(b_0)=\varphi^{k_2}(i_u)$, and $t(c_0)=\varphi^{l_1}(i_u)$, $s(d_0)=\varphi^{l_2}(i_u)$ with $k_1, l_2 \leq t$, $k_2, l_1 \geq t+1$. Indeed if this was the case then 
$$\xi(ba) = p_4\, b_0\, p_3\, r_{i_u,t}\, p_2\, a_0\, p_1$$
and
$$\xi(dc) = q_4\, d_0\, q_3\, r_{i_u,t}^{-1}\, q_2\, c_0\, q_1$$
for paths of isomorphism arrows $p_i,\, q_i$ (see \cref{pic:3.021} for an illustration of this). Hence if we simply choose $\xi$ so that the generators of arrows of two different paths in the cycles in $W$ that cross in the same orbit of vertices do not overlap in this way we remove instances of both $r_{i_u,t}$ and $r_{i_u,t}^{-1}$ appearing in $W'$. One way of ensuring this is for each orbit of vertices in $Q$, when choosing generators of the orbits of arrows whose sources lie in this orbit of vertices we have that the sources of the each of the generators are equal. This would amount to taking $s(b_0)=s(d_0)$ and hence at most either $r_{i_u,t}$ or $r_{i_u,t}^{-1}$ can appear in $W'$.
\end{proof}
\begin{figure}[H]
    \centering
    \includegraphics[scale=0.7]{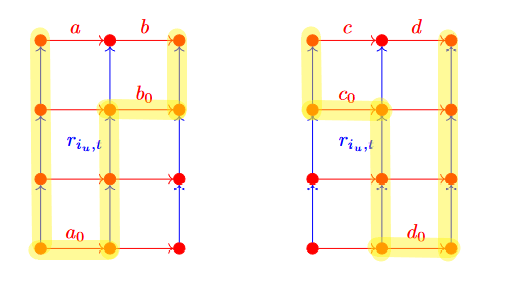}
    \caption{Two parts $ba$ and $dc$ of cycles that appear in $W$ that cross at the orbit of the vertex $i_u$. The chosen generators for the orbits are $a_0,\,b_0,\,c_0,\,d_0$ giving $\xi(ba)$ and $\xi(dc)$ as the paths highlighted in yellow. We can see that, due to these choices of generators, $r_{i_u,t}$ appears in $\xi(ba)$ and $r_{i_u,t}^{-1}$ appears in $\xi(dc)$ and so both appear in $W' = \xi(W)$. A choice of $\xi$ where this will not occur would be for example to take $b$ and $d$ to be the generators of their orbits. As the sources of $b$ and $d$ are equal this means no matter what choice we take for the generators of the orbits for $a$ and $c$, we can never see both $r_{i_u,t}$ and $r_{i_u,t}^{-1}$ in $W'$ (assuming no other parts of cycles in $W$ go through the orbit of $i_u$).}
    \label{pic:3.021} 
\end{figure}

\begin{lem}\label{lem3.9}
\begin{enumerate}[a)]
    \item[]
    \item Let $a \in Q'_1$ be a generating arrow which is dual to the edge in $\Delta$ that goes between the white vertex $v$ and the black vertex $u$. Then
    \begin{align}
        a \cdot \frac{\partial W'}{\partial a} = n \, \big(\xi(c_{v})-\xi(c_{u}) \big) \label{eq:3.1}
    \end{align}
    for the minimal cycles $c_v$ and $c_u$ in $W$.
    \item The relations $\partial W'/\partial r_{i_u,t}$ (or if relevant $\partial W'/\partial r_{i_u,t}^{-1}$) can be derived from the relations $\partial W'/\partial a$ for generating arrows $a \in Q'_1$.
\end{enumerate}
\end{lem}

\begin{proof}
a) \quad Write $\varphi^j(a)=a_j \in Q_1$ (so $a_0=a$). Then $\xi(a_j)=p_j a \, q_j$ where $p_j, q_j$ are paths in $\m CQ'$ made from isomorphism arrows. For each $j$ write $c_{v_j}$ and $c_{u_j}$ for the minimal cycles in $W$ such that $a_j$ is dual to the edge in $\Delta$ between the white vertex $v_j$ and black vertex $u_j$; as these cycles can be written up to cyclic permutation we write them so that $a_j$ appears at the end of these cycles. As $\varphi$ preserves the tiling $\Delta$ we have for any vertex $v \in \Delta$ that $\varphi(c_v)=c_{\varphi(v)}$. Hence $c_{v_j}=\varphi^j(c_{v_0})$ and $c_{u_j}=\varphi^j(c_{u_0})$. Then because we wrote the cycles so that $a_j$ appears at the end it follows that $\xi(c_{v_j})= p_j \xi(c_{v_0}) p_j^{-1}$ and $\xi(c_{u_j})= p_j \xi(c_{u_0}) p_j^{-1}$, and hence up to cyclic permutation $\xi(c_{v_j})= \xi(c_{v_0})$ and $\xi(c_{u_j})= \xi(c_{u_0})$. This is because $\xi(c_{v_0})$ (resp. $\xi(c_{u_0})$) ends with $a$ so is a cycle at the vertex $t(a)$ whilst $c_{v_j}$ (resp. $c_{u_j}$) is a cycle at the vertex $t(a_j)$, and $p_j$ is the unique path of isomorphism arrows between $t(a)$ and $t(a_j)$.
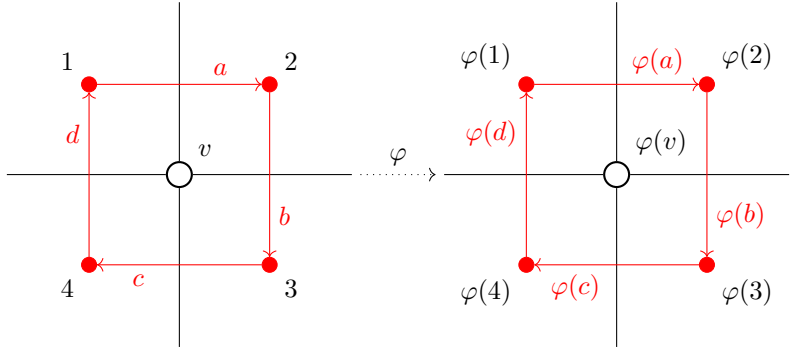
\begin{figure}[H]
    \centering
\begin{tikzpicture}
[black/.style={circle, draw=black!120, fill=black!120, thin, minimum size=3mm},
white/.style={circle, draw=black!120, thick, minimum size=3mm},
empty/.style={circle, minimum size=1pt, inner sep=1pt},
redsmall/.style={circle, draw=red!120, fill=red!120, thin, scale=0.6}]

\node[white,label=above right:{$v$}] (1) {};
\node[redsmall,label=above left:{$1$}] (2) [above left=of 1] {};
\node[redsmall,label=above right:{$2$}] (3) [above right=of 1] {};
\node[redsmall,label=below right:{$3$}] (4) [below right=of 1] {};
\node[redsmall,label=below left:{$4$}] (5) [below left=of 1] {};
\node[empty] (6) [above=of 1] {};
\node[empty] (7) [left=of 1] {};
\node[empty] (8) [right=of 1] {};
\node[empty] (9) [below=of 1] {};
\node[empty] (20) [above=of 6] {};
\node[empty] (21) [left=of 7] {};
\node[empty] (22) [right=of 8] {};
\node[empty] (23) [below=of 9] {};

\node[empty] (16) [right=of 22] {};
\node[empty] (25) [right=of 16] {};
\node[white,label=above right:{$\varphi(v)$}] (10) [right=of 25] {};
\node[redsmall,label=above left:{$\varphi(1)$}] (11) [above left=of 10] {};
\node[redsmall,label=above right:{$\varphi(2)$}] (12) [above right=of 10] {};
\node[redsmall,label=below right:{$\varphi(3)$}] (13) [below right=of 10] {};
\node[redsmall,label=below left:{$\varphi(4)$}] (14) [below left=of 10] {};
\node[empty] (15) [above=of 10] {};
\node[empty] (17) [right=of 10] {};
\node[empty] (18) [below=of 10] {};
\node[empty] (24) [above=of 15] {};
\node[empty] (26) [right=of 17] {};
\node[empty] (27) [below=of 18] {};

\draw[-] (1) -- (20);
\draw[-] (1) -- (21);
\draw[-] (1) -- (22);
\draw[-] (1) -- (23);
\draw[red,->] (2) -- (3) node [above, near end] {$a$};
\draw[red,->] (3) -- (4) node [right, near end] {$b$};
\draw[red,->] (4) -- (5) node [below, near end] {$c$};
\draw[red,->] (5) -- (2) node [left, near end] {$d$};

\draw[-] (10) -- (24);
\draw[-] (10) -- (16);
\draw[-] (10) -- (26);
\draw[-] (10) -- (27);
\draw[red,->] (11) -- (12) node [above, near end] {$\varphi(a)$};;
\draw[red,->] (12) -- (13) node [right, near end] {$\varphi(b)$};
\draw[red,->] (13) -- (14) node [below, near end] {$\varphi(c)$};
\draw[red,->] (14) -- (11) node [left, near end] {$\varphi(d)$};

\draw[dotted,->] (22) -- (16) node [above, midway] {$\varphi$};

\end{tikzpicture}
    \caption{An example of part of a brane tiling with the action of $\varphi$. Let $\varphi$ send the part of the brane tiling on the left to the part on the right. Take the generators of the orbits to be $a,\,b,\,c,\,d$ and write $p_i$ for the path in $\m CQ'$ made of isomorphism arrows between the vertices $i$ and $\varphi(i)$. Hence $\xi(\varphi(a)) = p_2 a p_1^{-1}$, $\xi(\varphi(b)) = p_3 b p_2^{-1}$, $\xi(\varphi(c)) = p_4 c p_3^{-1}$, and $\xi(\varphi(d)) = p_1 d p_4^{-1}$. Therefore $\xi(c_{\varphi(v)}) = \xi(\varphi(d)) \xi(\varphi(c)) \xi(\varphi(b)) \xi(\varphi(a)) = p_1 dcba p_1^{-1} = p_1 \xi(c_v) p_1^{-1}$ and all the intermediary paths of isomorphism arrows cancel, as required.}
    \label{pic:3.03}
\end{figure}
To calculate $\partial W'/\partial a$, the relevant terms in $W'$ are
$$\xi(c_{v_j})-\xi(c_{u_j}) = \xi(c_{v_0})-\xi(c_{u_0}) \,\, \r{(up to cyclic permutation)}$$
for $j \in \{0, \ldots, n-1\}$. Writing the cycles $\xi(c_{v_0})$ and $\xi(c_{u_0})$ so that $a$ is at the end, it follows that
$$ a \cdot \frac{\partial W'}{\partial a} = n \, \big(\xi(c_{v_0})-\xi(c_{u_0}) \big).$$

b) \quad Morally this is true because $r_{i_u,t} \cdot \partial W'/\partial r_{i_u,t}$ is a sum whose terms are just the cycles in $W'$ that contain $r_{i_u,t}$ with $r_{i_u,t}$ cyclically permuted to the front. But the cycles in $W'$ can be made from the relations $\partial W'/\partial a$ for generating arrows $a$ and so because $r_{i_u,t}$ is invertible in $\m CQ'$, $\partial W'/\partial r_{i_u,t}$ is $r_{i-u,t}^{-1}$ multiplied by terms that can be obtained from the relations $\partial W'/\partial a$. We just need to check which generating arrows $a$ are needed to obtain the relevant cycles in $W'$ which contain $r_{i_u,t}$.

So consider the relation $\partial W'/\partial r_{i_u,t}$ and without loss of generality we assume that $W'$ does not contain $r_{i_u,t}^{-1}$. For clarity write $r=r_{i_u,t}$. If $a \in Q_1$ has $r$ in its image under $\xi$ and if $b \in Q_1$ has $r^{-1}$ in its image under $\xi$ then we can write
$$\xi(a) = p'_a\, r^{x_a}\, p''_a\, a_0\, q'_a\, r^{y_a}\, q''_a$$
and
$$\xi(b) = p'_b\, r^{x_b}\, p''_b\, b_0\, q'_b\, r^{y_b}\, q''_b$$
for $a_0$ and $b_0$ the chosen generators of the orbits, $p'_a,\, p''_a,\, q'_a,\, q''_a$ and $p'_b,\, p''_b,\, q'_b,\, q''_b$ paths of isomorphism arrows that do not contain $r$ or $r^{-1}$, and $x_a,\, y_a \in \{0,1\}$ and $x_b,\, y_b \in \{-1,0\}$. If $a$ is dual to the edge in $\Delta$ between the white vertex $v_a$ and black vertex $u_a$ and we write the minimal cycles $c_{v_a}$ and $c_{u_a}$ so that $a$ is at the end, then $r \cdot \partial W'/\partial r$ will contain terms like
\begin{align}
    r\, p''_a\, a_0\, q'_a\, r^{y_a}\, q''_a \Big(\xi(a^{-1} c_{v_a}) - \xi(a^{-1} c_{u_a}) \Big) p'_a \label{eq:3.21}
\end{align}
and
\begin{align}
    r\, q''_a \Big(\xi(a^{-1} c_{v_a}) - \xi(a^{-1} c_{u_a}) \Big) p'_a\, r^{x_a}\, p''_a\, a_0\, q'_a. \label{eq:3.22}
\end{align} 
However since $W'$ does not contain $r^{-1}$ whenever an arrow $b$ appears in $W$ the $r^{-1}$ it gives under $\xi$ must be cancelled out by an $r$ given by some $a$ under $\xi$. Hence we must remove those terms from $r \cdot \partial W'/\partial r$ that are cancelled. What is removed is exactly of the form
\begin{align}
    r\, r^{-1}\, p''_b\, b_0\, q'_b\, r^{y_b}\, q''_b \Big(\xi(b^{-1} c_{v_b}) - \xi(b^{-1} c_{u_b}) \Big) p'_b \label{eq:3.23}
\end{align}
or
\begin{align}
    r\, r^{-1}\, q''_b \Big(\xi(b^{-1} c_{v_b}) - \xi(b^{-1} c_{u_b}) \Big) p'_b\, r^{x_b}\, p''_b\, b_0\, q'_b \label{eq:3.24}
\end{align}
and so $r \cdot \partial W'/\partial r$ is a sum of terms like \eqref{eq:3.21} and \eqref{eq:3.22} minus terms like \eqref{eq:3.23} and \eqref{eq:3.24}. Now \eqref{eq:3.21} can be written as
\begin{align*}
    r\, p''_a\, a_0\, q'_a\, r^{y_a}\, q''_a \Big(\xi(a^{-1} c_{v_a}) - \xi(a^{-1} c_{u_a}) \Big) p'_a  &= p_a^{\prime -1} \Big(p_a'\, r\, p''_a\, a_0\, q'_a\, r^{y_a}\, q''_a \Big(\xi(a^{-1} c_{v_a}) - \xi(a^{-1} c_{u_a}) \Big)\Big) p'_a\\
    &= p_a^{\prime -1} \Big(\xi(a) \xi(a^{-1}) \Big(\xi(c_{v_a}) - \xi(c_{u_a}) \Big)\Big) p'_a\\
    &= \frac{1}{n}\, p_a^{\prime -1} \Big(p_a\, a_0 \cdot \partial W'/\partial a_0\, p_a^{-1} \Big) p'_a
\end{align*}
where $p_a=p'_a r p''_a$ and the last line follows from part a). Similarly \eqref{eq:3.22}, \eqref{eq:3.23}, \eqref{eq:3.24} can all be written in terms of the derivatives of $W'$ with respect to the generating arrows $a_0$ or $b_0$. This gives the result.
\end{proof}

\begin{lem}\label{lem3.10}
$\xi$ induces an inclusion $\widetilde{\xi}: \r{Jac}(Q,W) \hookrightarrow \r{Jac}(Q',W')$.
\end{lem}

\begin{proof}
We first show the induced map is well-defined. Let $b \in Q_1$ and suppose $b=\varphi^k(a)$ for some arrow $a$ in the chosen generating set of $Q_1$. As in the proof of \cref{lem3.9} we write $\xi(b)=p_k a\, q_k$ where $p_k,\, q_k$ are paths of isomorphism arrows in $\m CQ'$. Then we see that
\begin{align}
    \xi \left( b \cdot \frac{\partial W}{\partial b} \right) &= \xi(c_{v_b}) - \xi(c_{u_b}) \nonumber \\
    &= p_k \big(\xi(c_{v_a}) - \xi(c_{u_a})\big)p_k^{-1}. \label{eq:3.3}
\end{align}
Combining \eqref{eq:3.3} with \eqref{eq:3.1} gives
$$\xi \left(b \cdot \frac{\partial W}{\partial b}\right) = \frac{1}{n}\, p_k\, a \cdot \frac{\partial W'}{\partial a} p_k^{-1}$$
and so from the equation $\xi(b)= p_k a \, q_k$ we get that
\begin{align}
    \xi\left(\frac{\partial W}{\partial b}\right) =\frac{1}{n}  q_k^{-1} \frac{\partial W'}{\partial a}\, p_k^{-1}. \label{eq:3.4}
\end{align}
As for injectivity, we must show that given some $x \in \m CQ$ such that $\xi(x) \in I_{W'}$ then $x \in I_W$. Write
$$\xi(x) = \sum_a p'_a \frac{\partial W'}{\partial a}\, q'_a$$
for paths $p'_a,\, q'_a \in \m CQ'$. Using \cref{lem3.6}  we can write this as
$$\xi(x) = \sum_a p'_a \frac{\partial W'}{\partial a}\, r_a\, \xi(q_a)$$
for paths $q_a \in \m CQ$ and paths of isomorphism arrows $r_a \in \m CQ'$. Let $\varphi^{k_a}(a)$ be the arrow in $Q_1$ such that $r_a^{-1}$ is the unique path of isomorphism arrows in $\m CQ'$ between $t(a)$ and $t(\varphi^{k_a}(a))$, then using \eqref{eq:3.4} we can write
$$\xi(x) = \sum_a p''_a \, \xi \left(\frac{\partial W}{\partial \varphi^{k_a}(a)} \, q_a \right)$$
where $p''_a =p'_a s_a$ and $s_a$ is the unique path of isomorphism arrows between $s(a)$ and $s(\varphi^{k_a}(a))$. Taking degrees we see from \cref{lem3.5} that $\r{deg}(p''_a) \equiv 0 \, \r{mod} \, n$ and so \cref{lem3.6} gives us a path $p_a \in \m CQ$ such that $\xi(p_a)=p''_a$. As $\xi$ is injective we get
$$x = \sum_a p_a\, \frac{\partial W}{\partial \varphi^{k_a}(a)}\, q_a$$
giving the result.
\end{proof}

We make the further assumption that we can choose generators and isomorphism arrows such that additionally $W'$ will be homogeneous of degree $n = \r{deg}(\varphi)$. We can see that in the case of the running example \Cref{eg3.1} that $W' = abreabre+2rdrc-2ardbrc-rere$ indeed has degree 2 in the isomorphism arrow $r$. We give some justification about why this assumption can be made.

\begin{lem}\label{lem3.11}
Let $\varphi$ be an automorphism of order $n$ of the Riemann surface $\Sigma_g$ and let $\Delta_0$ be a brane tiling of $\Sigma_g$ which is preserved by $\varphi$. Then we can extend $\Delta_0$ into a brane tiling $\Delta$ which is also preserved by $\varphi$ such that there exists a dimer for $\Delta$.
\end{lem}

\begin{proof}
In a brane tiling define the \emph{distance} between two distinct vertices $u$ and $v$, denoted by $\r{dist}(u,v)$, to be 1 if there exists a tile in the brane tiling for which both $u$ and $v$ are in the perimeter of it. Then recursively define the distance of two vertices to be $x+1$ if there exists another vertex $u'$ such that $\r{dist}(u,u')=1$ and $\r{dist}(u',v)=x$ or vice versa.

For our brane tiling $\Delta_0$ we first add vertices and their images under $\varphi, \ldots, \varphi^{n-1}$ until the number of black vertices equals the number of white vertices. Next we add edges and their images under $\varphi, \ldots, \varphi^{n-1}$ until we end up with a brane tiling which we call $\Delta_0'$. We then begin to construct a dimer $D$ for $\Delta_0'$ by choosing pairs of vertices in $\Delta_0'$ (one black and one white) which are connected by an edge. Once a vertex has been chosen in a pair it is then removed from further consideration for future pairings.

In this way we can see that if we are able to have all vertices in $\Delta_0'$ in one of these pairs then we will indeed have a dimer for $\Delta_0'$. However it might turn out that due to our choices there exists a (w.l.o.g) black vertex $u$ such that all the white vertices adjacent to it are already paired off. 

To remedy such an issue with constructing our dimer, we note that because the number of black and white vertices is equal there must exist at least one white vertex $v$ that has not been paired off. If $\r{dist}(u,v)=1$ we add an edge to $\Delta_0'$ (along with all its images under $\varphi, \ldots, \varphi^{n-1}$) connecting $u$ to $v$ and then take this to be in $D$. If $\r{dist}(u,v)>1$ we consider all the white vertices $v_1$ such that $\r{dist}(u,v_1)=1$. We then choose a white vertex $v_1$ such that $v_1$ is paired with the black vertex $u_1$ and $\r{dist}(u_1,v)<\r{dist}(u,v)$. Such a vertex $u_1$ must exist due to our definition of distance in the brane tiling. We then add an edge between $u$ and $v_1$ to the brane tiling (along with its images under $\varphi, \ldots, \varphi^{n-1}$) unless one already exists in $\Delta_0'$. Then we replace the pair $(u_1,v_1)$ in $D$ with $(u,v_1)$, thereby shifting the issue we are having with constructing our dimer onto $u_1$. We continue to repeat this, each time noting that the distance between the black vertex $u_i$ and our white vertex $v$ is decreasing. Hence after a finite number of steps we obtain a black vertex $u_y$ such that $\r{dist}(u_y,v)=1$ and so we can add an edge between $u_y$ and $v$ into the brane tiling and then add the pair $(u_y,v)$ into $D$. See \cref{pic:3.25} for an illustrative example. 

This will fully rectify the issue we had of not being able to include the vertex $u$ as a pair into $D$ and also not introduce any further issues since all vertices that had been paired up in $D$ beforehand still remain in $D$ (albeit some will be paired with different vertices now). Continuing to do this for each vertex will ensure that all vertices end up in $D$ and hence $D$ will be a dimer for the brane tiling $\Delta$ we end up with.
\begin{figure}
\begin{subfigure}{1.0\textwidth}
\centering
\begin{tikzpicture}
[black/.style={circle, draw=black!120, fill=black!120, thin, minimum size=3mm},
white/.style={circle, draw=black!120, thick, minimum size=3mm},
empty/.style={circle, minimum size=1pt, inner sep=1pt},
redsmall/.style={circle, draw=red!120, fill=red!120, thin, scale=0.6}]

\node[black] (1) {};
\node[white] (2) [right=of 1] {};
\node[white] (3) [below left=of 1] {};
\node[black] (4) [below right=of 2] {};
\node[black] (5) [below right=of 3] {};
\node[white,label=above:{$v$}] (6) [right=of 5] {};
\node[white] (7) [below left=of 5] {};
\node[black] (8) [below right=of 6] {};
\node[black] (9) [below right=of 7] {};
\node[white] (10) [right=of 9] {};
\node[white] (11) [right=of 4] {};
\node[white] (12) [right=of 8] {};
\node[black] (14) [below right=of 11] {};
\node[white] (15) [right=of 14] {};
\node[black] (16) [above right=of 11] {};
\node[white] (17) [right=of 16] {};
\node[black,label=above:{$u$}] (18) [below right=of 17] {};
\node[black] (19) [below right=of 15] {};
\node[black] (20) [below right=of 12] {};
\node[white] (21) [right=of 20] {};

\draw[green,-] (1) -- (2);
\draw[green,-] (3) -- (5);
\draw[green,-] (7) -- (9);
\draw[green,-] (10) -- (8);
\draw[green,-] (4) -- (11);
\draw[green,-] (12) -- (14);
\draw[green,-] (15) -- (19);
\draw[green,-] (16) -- (17);
\draw[green,-] (20) -- (21);
\draw[black,-] (1) -- (3);
\draw[black,-] (2) -- (4);
\draw[black,-] (5) -- (6);
\draw[black,-] (5) -- (7);
\draw[black,-] (9) -- (10);
\draw[black,-] (6) -- (4);
\draw[black,-] (6) -- (8);
\draw[black,-] (8) -- (12);
\draw[black,-] (11) -- (14);
\draw[black,-] (11) -- (16);
\draw[black,-] (12) -- (20);
\draw[black,-] (14) -- (15);
\draw[black,-] (17) -- (18);
\draw[black,-] (15) -- (18);
\draw[black,-] (21) -- (19);

\end{tikzpicture}
\end{subfigure}

\vspace{5em}

\begin{subfigure}{1.0\textwidth}
\centering
\begin{tikzpicture}
[black/.style={circle, draw=black!120, fill=black!120, thin, minimum size=3mm},
white/.style={circle, draw=black!120, thick, minimum size=3mm},
empty/.style={circle, minimum size=1pt, inner sep=1pt},
redsmall/.style={circle, draw=red!120, fill=red!120, thin, scale=0.6}]

\node[black] (1) {};
\node[white] (2) [right=of 1] {};
\node[white] (3) [below left=of 1] {};
\node[black] (4) [below right=of 2] {};
\node[black] (5) [below right=of 3] {};
\node[white,label=above:{$v$}] (6) [right=of 5] {};
\node[white] (7) [below left=of 5] {};
\node[black] (8) [below right=of 6] {};
\node[black] (9) [below right=of 7] {};
\node[white] (10) [right=of 9] {};
\node[white,label=above:{$v'$}] (11) [right=of 4] {};
\node[white] (12) [right=of 8] {};
\node[black] (14) [below right=of 11] {};
\node[white] (15) [right=of 14] {};
\node[black] (16) [above right=of 11] {};
\node[white] (17) [right=of 16] {};
\node[black,label=above:{$u$}] (18) [below right=of 17] {};
\node[black] (19) [below right=of 15] {};
\node[black] (20) [below right=of 12] {};
\node[white] (21) [right=of 20] {};

\draw[green,-] (1) -- (2);
\draw[green,-] (3) -- (5);
\draw[green,-] (7) -- (9);
\draw[green,-] (10) -- (8);
\draw[black,-] (4) -- (11);
\draw[green,-] (12) -- (14);
\draw[green,-] (15) -- (19);
\draw[green,-] (16) -- (17);
\draw[green,-] (20) -- (21);
\draw[black,-] (1) -- (3);
\draw[black,-] (2) -- (4);
\draw[black,-] (5) -- (6);
\draw[black,-] (5) -- (7);
\draw[black,-] (9) -- (10);
\draw[blue,-] (6) -- (4);
\draw[black,-] (6) -- (8);
\draw[black,-] (8) -- (12);
\draw[black,-] (11) -- (14);
\draw[black,-] (11) -- (16);
\draw[black,-] (12) -- (20);
\draw[black,-] (14) -- (15);
\draw[black,-] (17) -- (18);
\draw[black,-] (15) -- (18);
\draw[black,-] (21) -- (19);
\draw[green,dashed] (11) -- (18);

\end{tikzpicture}
\end{subfigure}
    \caption{The top picture shows an example of part of a brane tiling with the partial dimer $D$ in green. There are two vertices $v$ and $u$ which are not contained by an edge in $D$ with $\r{dist}(u,v)=2$. With the current setup it is not possible to include $v$ in $D$ to get a dimer. We rectify this in the second picture by changing the edges we include in $D$ (with the new edge in $D$ shown in blue). This shifts the problem onto the vertex $v'$ which is such that $\r{dist}(u,v')=1$. We can then add an edge that connects $v'$ and $u$ into $\Delta'_0$ (along with all its images under $\varphi, \ldots, \varphi^{n-1}$), and then include this new edge into $D$.}
    \label{pic:3.25} 
\end{figure}
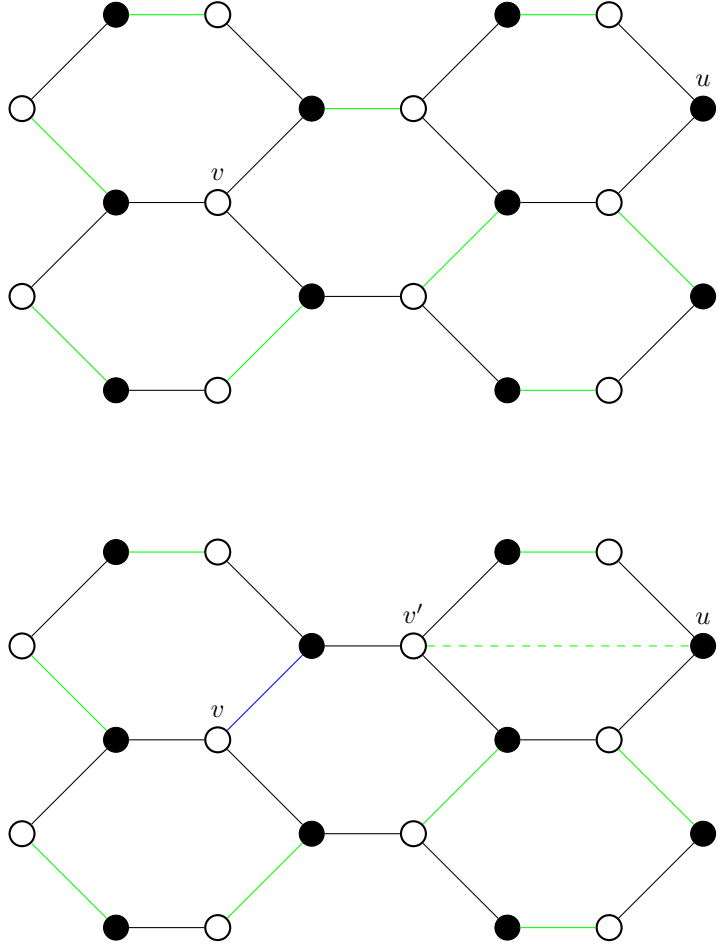
\end{proof}

We must always ensure that our brane tilings are preserved by $\varphi$ and that the orbits of the vertices in the dual quiver $Q$ have size $n$. By \cref{rem3.2} we can add to a brane tiling with this property and retain the property, so if we take a brane tiling $\Delta_0$ that is preserved by $\varphi$ then apply \cref{lem3.1} and then \cref{lem3.11} to it we end up with a brane tiling $\Delta$ for which both the orbit size of every vertex in $Q=Q_\Delta$ is $n$ and for which there exists a dimer $D$ for $\Delta$. To try to make $W'$ homogeneous of degree $n$ in the isomorphism arrows we consider the set $E$ of arrows in $Q$ that are dual to the edges in $D$. It is then hoped that there is a choice of $\xi$ for which the conditions explained in the proof of in \cref{lem3.8} hold and for which $\xi(e)$ has degree $n$ for each $e \in E$ and $\xi(a)$ has degree 0 for every other arrow $a \in Q_1 \setminus E$. This will then give a potential $W'$ which has the properties of not containing both $r_{i_u,t}$ and $r_{i_u,t}^{-1}$, and having degree $n$ in the isomorphism arrows since $W$ is homogeneous of degree 1 in the arrows in $E$.

Given these assumptions on $W'$ hold, similarly to \cite{d1} Proposition 4.2 and the preceding discussion, we define a homomorphism
\begin{align}
    \r{Jac}(\widetilde{Q'},W') \xrightarrow{\,\, \Psi \,\,} \r{Mat}_{m\times m}(\m C[\pi_1(M_{g,\varphi})]) \label{eq:3.01}
\end{align}
where $m$ is the number of vertices in $Q'$ (which is also the number of vertices in $Q$) and the tilde denotes that we are localising the path algebra $\m CQ'$ with respect to every arrow $a \in Q'_1$ before taking the quotient. To construct $\Psi$ explicitly consider a maximal tree $T \subset Q$ such that $\xi(t)$ has degree 0 for all $t \in T$. With regards to the dimer $D$ discussed above, this means that the intersection of $T$ and the set of arrows dual to edges in $D$ is empty. Fix a basepoint $\widetilde{\r{bp}} \in \Sigma_g$ that is invariant under $\varphi$ (note such a basepoint will be a vertex in the brane tiling $\Delta$ as per \cref{lem3.1}), fix a vertex $\r{bp} \in Q_0$, and fix a path $\delta: \widetilde{\r{bp}} \rightarrow \r{bp}$ in $\Sigma_g$. Then we can view
$$\pi_1\big(M_{g, \varphi},\, [\widetilde{\r{bp}},0]\big) \cong \pi_1\big(\Sigma_g,\, \widetilde{\r{bp}}\big) \rtimes_\varphi \m Z.$$
From now on, although technically all loops in $M_{g,\varphi}$ and $\Sigma_g$ have basepoint $\widetilde{\r{bp}}$, we work with loops at the basepoint $\r{bp} \in Q_0 \subset \Sigma_g$ and implicitly use $\delta$ to formally view them as loops at $\widetilde{\r{bp}}$.  For each $i \in Q_0$ let $i_s = \varphi^s(i)$ and let $t_i$ be the unique path $\r{bp} \rightarrow i$ in $\m C \widetilde{Q}'$ comprised solely from arrows in $T$ or their inverses. Let $E_{x,y}(z)$ denote the matrix with entries $z \in \m C[\pi_1(\Sigma_g) \rtimes_\varphi \m Z]$ in the the $(x,y)$-th position and zeroes everywhere else. Then define $\Psi: \r{Jac}(\widetilde{Q'},W') \rightarrow \r{Mat}_{m\times m}(\m C[\pi_1(\Sigma_g) \rtimes_\varphi \m Z])$ by sending
\begin{align*}
    a: i \rightarrow j &\longmapsto E_{j,i}\big(([t_j^{-1} a\, t_i],0) \big)\\
    r_{i_u, s}: i_{u,s} \rightarrow i_{u, s+1} &\longmapsto E_{i_{u, s+1}, i_{u,s}}\big(([\varphi^{-1}(t_{i_{u, s+1}}^{-1})\, t_{i_{u,s}}], -1) \big)
\end{align*}
where $a \in Q'_1$ is a generating arrow, $r_{i_u,s} \in Q'_1$ is an isomorphism arrow, $[l] \in \pi_1(\Sigma_g, \widetilde{\r{bp}})$ denotes the class of the loop $l \in \Sigma_g$ and we view paths $p \in \m CQ$ as paths in $\Sigma_g$ via the natural inclusion $Q \hookrightarrow \Sigma_g$.

\begin{rem}
It is not immediate that $\Psi$ as given above is well-defined.
\end{rem}

\begin{rem}
The assignment of an arrow $a: i \rightarrow j$ to $t_j^{-1} a\, t_i$ in the above homomorphism is equivalent to the contraction of the tree $T$ seen in [\cite{d1} Proposition 4.2].
\end{rem}

\begin{rem}
$\varphi^{-1}(t_{i_{q, p+1}}^{-1})\, t_{i_{q,p}}$ is not a loop in $\Sigma_g$ but in fact is a path $\r{bp} \rightarrow \varphi^{-1}(\r{bp})$. So, first fixing some path $\gamma: \r{bp} \rightarrow \varphi(\r{bp})$ in $\Sigma_g$, what we actually mean by $[\varphi^{-1}(t_{i_{q, p+1}}^{-1}) t_{i_{q,p}}] \in \pi_1(\Sigma_g, \r{bp})$ is the class of the loop $\varphi^{-1}(\gamma) \circ \varphi^{-1}(t_{i_{q, p+1}}^{-1}) t_{i_{q,p}}$. Since $\r{bp} \in Q_0$ we take $\gamma = t_{\varphi(\r{bp})}$. In a similar vein when multiplying in the semi-direct product we technically should write
$$([\beta], m) \cdot ([\alpha], k) = ([\gamma^{-1} \circ \varphi(\gamma^{-1}) \circ \cdots \circ \varphi^{k-1}(\gamma^{-1}) \circ \varphi^k(\beta) \circ \varphi^{k-1}(\gamma) \circ \cdots \circ \varphi(\gamma) \circ \gamma \circ \alpha], k+m)$$
for $k>0$ and
$$([\beta], m) \cdot ([\alpha], k) = ([\varphi^{-1}(\gamma) \circ \cdots \circ \varphi^{k+1}(\gamma^{-1}) \circ \varphi^k(\gamma) \circ \varphi^k(\beta) \circ \varphi^k(\gamma^{-1}) \circ \varphi^{k+1}(\gamma^{-1}) \circ \cdots \circ \varphi^{-1}(\gamma^{-1}) \circ \alpha], k+m)$$
for $k<0$ in order to ensure we are multiplying loops at $\r{bp}$. However it will turn out that in the multiplications we present later these intermediary paths, that are needed to correct for the basepoints of the loops involved, will mostly all cancel and so we omit them from the proceeding discussion to aid in notational clarity. This would not be necessary if we worked with the $\varphi$-invariant basepoint $\widetilde{\r{bp}}$, but it must be done to allow us to work within the quiver $Q'$.
\end{rem}
Note that $\Psi$ sends
$$r_{i_u,s}^{-1} \longmapsto E_{i_{u,s}, i_{u, s+1}}\big(([\varphi(t_{i_{u,s}}^{-1})\, t_{i_{u, s+1}}], 1) \big).$$

\begin{thm}\label{thm3.12}
Let $\Sigma_g$ be a Riemann surface of genus $g$ and let $\varphi$ be an orientation-preserving automorphism of $\Sigma_g$ of order $n$. Let $\Delta$ be a brane tiling of $\Sigma_g$ which is preserved under $\varphi$ such that the size of the order of each vertex in $Q=Q_\Delta$ is $n$. Choose generating arrows from $Q$ and isomorphism arrows to construct the quiver $Q'$ as before, such that the potential $W'$ is homogeneous of degree $n$ and does not contain both an isomorphism arrow $r_{i_u,t}$ and its inverse $r_{i_u,t}^{-1}$. Then the homomorphism of algebras \eqref{eq:3.01}
$$\r{Jac}(\widetilde{Q'},W') \xrightarrow[\,\,\,\, \sim \,\,\,\,]{\Psi} \r{Mat}_{m\times m}(\m C[\pi_1(M_{g,\varphi})])$$
is well-defined and an isomorphism.
\end{thm}

\begin{proof}
Using the definition of $\Psi$ we have that
\begin{align}
    \Psi \circ \widetilde{\xi}(p) = E_{y,x}\big(([t_y^{-1} p\, t_x], -\r{deg}(\widetilde{\xi}(p))) \big) \label{eq:3.02}
\end{align}
for any path $p: x \rightarrow y \in \r{Jac}(\widetilde{Q}, W)$. Indeed for an arrow $b \in Q_1$ we have that $b= \varphi^l(a)$ for some generating arrow $a \in Q_1$ and $l \in \{0, \ldots, n-1 \}$, and hence we can write $\widetilde{\xi}(b)$ as equal to
\begin{align*}
    & s_1^{-1} \ldots s_k^{-1}\, a\, r_k \ldots r_1\\
    \r{or} \quad & s_{l-k} \ldots s_1\, a\, r_k \ldots r_1\\
    \r{or} \quad & s_k \ldots s_1\, a\, r^{-1}_1 \ldots r^{-1}_k\\
    \r{or} \quad & s_1^{-1} \ldots s_{l-k}^{-1}\, a\, r^{-1}_1 \ldots r^{-1}_k
\end{align*}
for isomorphism arrows $r_i: x_{i-1} \rightarrow x_i,\, s_i: y_{i-1} \rightarrow y_i$, and some $0 \leq k \leq l$. In the first of these cases we get that $a: x_k \rightarrow y_k$ and $b: x_0 \rightarrow y_0$, and therefore
\begin{align*}
    \Psi(\widetilde{\xi}(b)) &= \Psi(s_1^{-1} \ldots s_k^{-1}\, a\, r_k \ldots r_1)\\
    &= E_{y_0, y_1}\big(([\varphi(t_{y_0}^{-1}) t_{y_1}], 1) \big) \, \cdots \, E_{y_{k-1}, y_k}\big(([\varphi(t_{y_{k-1}}^{-1}) t_{y_k}], 1) \big) \,\, E_{y_k, x_k}\big(([t_{y_k}^{-1} a\, t_{x_k}], 0) \big)\\
    & \quad \,\, E_{x_k, x_{k-1}}\big(([\varphi^{-1}(t_{x_k}^{-1}) t_{x_{k-1}}], -1) \big) \, \cdots \, E_{x_1, x_0}\big(([\varphi^{-1}(t_{x_1}^{-1}) t_{x_0}], -1) \big)\\
    &= E_{y_0, x_0}\big(([\varphi(t_{y_0}^{-1}) t_{y_1}], 1) \cdot \ldots \cdot ([\varphi(t_{y_{k-1}}^{-1}) t_{y_k}], 1) \cdot ([t_{y_k}^{-1} a\, t_{x_k}], 0) \cdot\\
    & \qquad \qquad \cdot ([\varphi^{-1}(t_{x_k}^{-1}) t_{x_{k-1}}], -1) \cdot \ldots \cdot ([\varphi^{-1}(t_{x_1}^{-1}) t_{x_0}], -1) \big)\\
    &= E_{y_0,x_0}\big(([t_{y_0}^{-1} \varphi^k(a)\, t_{x_0}], 0) \big)\\
    &= E_{y_0,x_0}\big(([t_{y_0}^{-1} b\, t_{x_0}], -\r{deg}(\widetilde{\xi}(b))) \big)
\end{align*}
as required. The other cases follow in a similar manner and since we have shown the statement for all arrows in $Q$ it readily extends to all paths in $\m C \widetilde{Q}$.

By \Cref{lem3.9}, to show that $\Psi$ is well-defined it suffices to check that
$$\Psi \left(\frac{\partial W'}{\partial a} \right) = 0$$
for all generating arrows $a \in Q_1$. But by \Cref{lem3.10} we have that
$$\frac{\partial W'}{\partial a} = n \, \xi \left(\frac{\partial W}{\partial a} \right)$$
and hence we just need to show that $\Psi(\widetilde{\xi}(\partial W/ \partial a)) = 0$. Writing $\partial W/\partial a = p - q$ for $a: i \rightarrow j$, \eqref{eq:3.02} tells us that
\begin{align*}
    \Psi \left(\widetilde{\xi} \left(\frac{\partial W}{\partial a} \right) \right) &= E_{i,j}\big(([t_i^{-1} p t_j^{-1}], -\r{deg}(\widetilde{\xi}(p)) \big) - E_{i,j}\big(([t_i^{-1} q t_j^{-1}], -\r{deg}(\widetilde{\xi}(q)) \big)\\
    &= E_{i,j}\big(([t_i^{-1} p t_j^{-1}], -\r{deg}(\widetilde{\xi}(p)) - ([t_i^{-1} q t_j^{-1}], -\r{deg}(\widetilde{\xi}(q))  \big).
\end{align*}
But [\cite{d1} Proposition 4.2 and Proposition 5.4] implies that $[p]=[q] \in \pi_1(\Sigma_g)$ and since $W'$ is homogeneous $\r{deg}(\widetilde{\xi}(p)) = \r{deg}(\widetilde{\xi}(q))$, hence
$$E_{i,j}\big(([t_i^{-1} p t_j^{-1}], -\r{deg}(\widetilde{\xi}(p)) - ([t_i^{-1} q t_j^{-1}], -\r{deg}(\widetilde{\xi}(q))  \big) = E_{i,j}\big(0\big) = 0$$
as required.

We now move onto showing that $\Psi$ is an isomorphism. For surjectivity we choose $\r{bp} \in Q'_0$ such that there exists an isomorphism arrow $r: \r{bp} \rightarrow \varphi(\r{bp})$ and recall we take $\gamma = t_{\varphi(\r{bp})}$. Then $r$ has image under $\Psi$ given by \begin{align*}
    E_{\varphi(\r{bp}),\r{bp}}\big(([\varphi^{-1}(\gamma) \circ \varphi^{-1}(t_{\varphi(\r{bp})}^{-1})], -1) \big) &= E_{\varphi(\r{bp}),\r{bp}}\big(([\varphi^{-1}(t_{\varphi(\r{bp})}) \circ \varphi^{-1}(t_{\varphi(\r{bp})}^{-1})], -1)\\
    &= E_{\varphi(\r{bp}),\r{bp}}\big(([1], -1)\big)
\end{align*}
where $[1] \in \pi_1(\Sigma_g)$ is the class of the constant path. For $i,j \in Q'_0$ consider the path 
$$p_{ij} = \widetilde{\xi}(t_i\, t_{\varphi(\r{bp})}^{-1})\, r\, \widetilde{\xi}(t_j^{-1})$$
in $\widetilde{Q}'$. Since $\r{deg}(t) = 0$ for all $t \in T$ we have that $\r{deg}(t_i) = 0$ for all $i \in Q'_0$, and so by \eqref{eq:3.02} applying $\Psi$ to $p_{ij}$ gives
\begin{align*}
    & E_{i,\varphi(\r{bp})}\big(([t_i^{-1} t_i t_{\varphi(\r{bp})}^{-1} t_{\varphi(\r{bp})}], 0)\big) \cdot E_{\varphi(\r{bp}),\r{bp}}\big(([1],-1)\big) \cdot E_{\r{bp},j}\big(([t_{\r{bp}}^{-1} t_j^{-1} t_j],0)\big)\\
    &= E_{i,j}\big(([\varphi^{-1}(\gamma) \cdot \varphi^{-1}(t_i^{-1} t_i t_{\varphi(\r{bp})}^{-1} t_{\varphi(\r{bp})}) \cdot \varphi^{-1}(\gamma^{-1}) \cdot t_{\r{bp}}^{-1} t_j^{-1} t_j],-1)\big)\\
    &= E_{i,j}\big(([1],-1)\big)
\end{align*}
because by definition $t_{\r{bp}} = 1$. In a similar way replacing the $r$ with $r^{-1}$ gives the matrices $E_{i,j}\big(([1],1)\big)$ in the image of $\Psi$ too. Hence these paths, varying over all vertices $i,j$ in $Q'$, will generate the $\m Z$ summand in the semi-direct product $\pi_1(\Sigma_g) \rtimes_\varphi \m Z$ for all coordinates $(i,j)$. Next let
$$\pi: \r{Mat}_{m \times m}(\m C[\pi_1(\Sigma_g) \rtimes_\varphi \m Z]) \rightarrow \r{Mat}_{m \times m}(\m C[\pi_1(\Sigma_g)])$$
be the projection map, and note that $\pi$ is not an algebra homomorphism but becomes one when we restrict to the subalgebra $\r{Mat}_{m \times m}(\m C[\pi_1(\Sigma_g) \rtimes_\varphi n \m Z]) = \r{Mat}_{m \times m}(\m C[\pi_1(\Sigma_g) \times n \m Z])$, since $\varphi^n = \r{id}_{\Sigma_g}$. Then \eqref{eq:3.02} tells us that the composition
$$\r{Jac}(\widetilde{Q},W) \xrightarrow{\,\,\widetilde{\xi}\,\,} \r{Jac}(\widetilde{Q}', W') \xrightarrow{\,\,\Psi\,\,} \r{Mat}_{m \times m}(\m C[\pi_1(\Sigma_g) \rtimes_\varphi \m Z]) \xrightarrow{\,\,\pi\,\,} \r{Mat}_{m \times m}(\m C[\pi_1(\Sigma_g)])$$
is the homomorphism that sends a path $p: x \rightarrow y \in \r{Jac}(\widetilde{Q}, W)$ to
$$E_{y,x}([t_y^{-1} p\, t_x]) \in \r{Mat}_{m \times m}(\m C[\pi_1(\Sigma_g)])$$
since $\r{Im}(\Psi \circ \widetilde{\xi})  \subset \r{Mat}_{m \times m}(\m C[\pi_1(\Sigma_g) \rtimes_\varphi n \m Z])$. This homomorphism is surjective by \Cref{thm2.22}. Therefore $\Psi$ is surjective.

To show injectivity suppose $\Psi(\alpha)=\Psi(\beta)$ for paths $\alpha,\beta:i \rightarrow j$ in $\m CQ'$. Then 
$$\Psi(\beta^{-1}\alpha) = E_{i,i}\big(([1], 0) \big).$$
Let $\pi_{\m Z}: \r{Mat}_{m \times m}(\m C[\pi_1(\Sigma_g) \rtimes_\varphi \m Z]) \rightarrow \r{Mat}_{m \times m}(\m C[\m Z])$ then from the definition of $\Psi$ it is clear that the composition $\pi_{\m Z} \circ \Psi$ sends a path $p: i \rightarrow j$ in $\m C \widetilde{Q}'$ to
$$E_{j,i}(-\r{deg}(p)).$$
Hence $\r{deg}(\beta^{-1} \alpha) = 0$ and so by \Cref{lem3.6} $\beta^{-1} \alpha$ lies in the image of $\widetilde{\xi}$. So we can view $\beta^{-1}\alpha$ as a path in $\r{Jac}(\widetilde{Q},W)$ and so by [\cite{d2} Lemma 2.7] we have that $\beta^{-1}\alpha=c_v^d$ for some minimal cycle $c_v$ around a vertex $v \in \Delta$ and some $d \in \m Z.$ Since $\r{deg}(\widetilde{\xi}(c_v)) = n$ from our assumption that the potential $W'$ is homogeneous of degree $n$, this implies that $\pi_{\m Z}(\Psi(\widetilde{\xi}(c_v)) \neq 0$. Therefore $d$ must be equal to 0 and hence $\beta^{-1}\alpha = c_v^0 = e_i$ the constant path at the vertex $i \in Q'_0$, and so $\alpha = \beta$ in $\m C \widetilde{Q}'$.
\end{proof}

\begin{eg}\label{eg:3.16}
We apply the homomorphism $\Psi$ to our running example of the genus 2 surface with $\varphi$ the rotation by $180^\circ$ from \cref{eg3.1}, using the brane tiling given in \cref{eg2.1}. 

Note that a finite presentation for the algebra $\m C[\pi_1(\Sigma_2) \rtimes_\varphi \m Z]$ is given by the generators $\{x,y,z\}$ and the relations
\begin{align}
    & xyx^{-1}y^{-1} z^{-1} xyx^{-1}y^{-1} z = 1 \label{eq:3.915}\\
    & x z^2 = z^2 x \nonumber\\
    & y z^2 = z^2 y. \nonumber
\end{align}

We saw from \eqref{eq:3.91} that for our brane tiling and that particular choice of generators and isomorphism arrows we get the potential
$$W' = abreabre + 2rdrc - 2ardbrc - rere.$$
This satisfies the conditions that $W'$ does not contain both $r$ and $r^{-1}$ and that $W'$ is homogeneous of degree $2 = \r{deg}(\varphi)$ in the isomorphism arrow $r$. We choose the maximal tree $T=\{e\} \subset Q$ since $\xi(e)=e$ has degree $0$. We have the following relations in $\r{Jac}(\widetilde{Q}',W')$
\begin{align*}
    \frac{\partial W'}{\partial a} &= 2breabre - 2rdbrc && \frac{\partial W'}{\partial b} = 2reabrea - 2rcard\\
    \frac{\partial W'}{\partial c} &= 2rdr - 2ardbr && \frac{\partial W'}{\partial d} = 2rcr - 2brcar\\
    \frac{\partial W'}{\partial e} &= 2abreabr - 2rer
\end{align*}
and recall from \cref{lem3.9} b) that we do not need to consider $\partial W'/\partial r$. 

Since the arrow $e$ is in our maximal tree $T$, we can think of $e$ as a means to go between the vertices in $\widetilde{Q}'$ and therefore as a means to go between the different coordinates in $\r{Mat}_{2\times 2}(\m  C[\pi_1(\Sigma_2) \rtimes_\varphi \m Z])$. Equivalently we can simply contract the arrow $e$ in $\widetilde{Q}'$ giving a quiver with 1 vertex and with all the arrows becoming loops at this vertex. So to simplify things, in the above relations we set $e=1$.

Then we have
\begin{align*}
    & \frac{\partial W'}{\partial a} = 0\\
    \Longrightarrow\,\,\, & brabr=rdbrc
\end{align*}
which combined with $\partial W'/\partial e=0$ gives
$$ardbrc=rr$$
and then combining with $\partial W'/\partial c=0$ gives
\begin{align}
    & rdrc = rr \nonumber\\
    \Longrightarrow\,\,\, & d = rc^{-1}r^{-1}. \label{eq:3.92}
\end{align}
Additionally
\begin{align}
    & \frac{\partial W'}{\partial b} = 0 \nonumber\\
    \Longrightarrow\,\,\, & b = a^{-1}carda^{-1}r^{-1} \nonumber\\
    \Longrightarrow\,\,\, & b = a^{-1}carrc^{-1}r^{-1}a^{-1}r^{-1}. \label{eq:3.93}
\end{align}
Using these substitutions we get
\begin{align*}
    & \frac{\partial W'}{\partial e} = 0\\
    \Longrightarrow\,\,\, & carrc^{-1}r^{-1}a^{-1}carrc^{-1}r^{-1}a^{-1}=rr
\end{align*}
and
\begin{align*}
    & \frac{\partial W'}{\partial a} = 0\\
    \Longrightarrow\,\,\, & a^{-1}carrc^{-1}r^{-1}a^{-1}carrc^{-1}r^{-1}a^{-1}=rrc^{-1}r^{-1}a^{-1}carrc^{-1}r^{-1}a^{-1}c
\end{align*}
which combined give
\begin{align}
    & a^{-1}rr = a^{-1}c^{-1}rrc \nonumber\\
    \Longrightarrow\,\,\, & crr = rrc. \label{eq:3.94}
\end{align}
Then combining $\partial W'/\partial b =0$ and $\partial W'/\partial d =0$ gives
\begin{align*}
    & rabra = b^{-1}rcrd\\
    \Longrightarrow\,\,\, & brabra = rcrd\\
    \Longrightarrow\,\,\, & a^{-1}carrrc^{-1}r^{-1}a^{-1}carrc^{-1}r^{-1} = rcrrc^{-1}r^{-1}
\end{align*}
and then using $rrc=crr$ implies that
$$a^{-1}carrrc^{-1}r^{-1}a^{-1}carrc^{-1}r^{-1} = rr.$$
Combining this with $\partial W'/\partial a =0$ and $\partial W'/\partial e =0$ gives
$$rra^{-1} = a^{-1}c^{-1}rrc$$
so again using $rrc=crr$ implies that
\begin{align}
    arr=rra. \label{eq:3.95}
\end{align}
Finally we have
\begin{align}
    & \frac{\partial W'}{\partial d} =0 \nonumber\\
    \Longrightarrow\,\,\, & rcr= a^{-1}carrrc^{-1}r^{-1}a^{-1}r^{-1}rcar \nonumber\\
    \Longrightarrow\,\,\, & rc=a^{-1}carrrc^{-1}r^{-1}a^{-1}ca \nonumber\\
    \Longrightarrow\,\,\, & 1=a^{-1}carc^{-1}ra^{-1}cac^{-1}r^{-1}. \label{eq:3.96}
\end{align}
again using $rrc=crr$ in the last step.

The relations \eqref{eq:3.92} and \eqref{eq:3.93} tell us that the generators $b$ and $d$ depend upon $a,c,r$, and the relations \eqref{eq:3.94}, \eqref{eq:3.95}, and \eqref{eq:3.96} are exactly the relations given in \eqref{eq:3.915}. Hence the map $a^{-1} \mapsto x$, $c \mapsto y$, $r \mapsto z$ is an isomorphism. It is easy to check that the map $\Psi$ described in \cref{thm3.12} is this isomorphism; for example at the $(1,1)$-coordinate $\Psi(a)=([a],0)$, $\Psi(c)=([e^{-1}c],0)$ and $\Psi(r)=([1],-1)$ and the loops $[a],\, [e^{-1}c] \in \Sigma_2$ can be seen using \cref{pic:3.0} as the first pair of the standard generators of $\pi_1(\Sigma_2)$, i.e. $x_1^{-1}$ and $y_1$, as required.
\end{eg}

\section{Future directions}
One of the benefits of a superpotential description of the fundamental group algebra $k[\pi_1(X)]$ is a more approachable way of calculating the (motivic) DT invariants associated to $\pi_1(X)$ (see \cite{dm} and \cite{mein} for good introductions to the Grothendieck ring of varieties and motivic DT invariants). [\cite{d1} Proposition 5.5] says that for a finitely generated algebra $A$ there is an equivalence of stacks 
$$\r{Rep}_d(A) \cong \r{Rep}_{dm}(\r{Mat}_{m\times m}(A))$$
and so it follows from \Cref{thm3.12} and \Cref{prop2.10} that we get a description of $\r{Rep}_d(\m C[\pi_1(M_{g,\varphi})])$ as a critical locus. Thus it is straightforward to define its motivic DT partition function. For an Artin stack $S$ locally of finite type over $\m C$ let $\c M_S^G = \r{K}_0^G(\r{Var}/S)[\m L^{-\frac{1}{2}}]$ denote the Grothendieck ring of naive $G$-equivariant motives over $S$ adjoined $\m L^{-\frac{1}{2}}$, where $\m L$ is the motive of $\m A^1$. Let $\c M^{\widehat{\mu}}_S = \lim_{\underset{r}{\to}} \c M^{\mu_r}_S$ where $\mu_r$ is the group of $r$-th roots of unity. Then explicitly this partition function is 
$$\Phi_{\pi_1(M_{g,\varphi})}(t) = \sum_{m=0}^\infty \Big[\r{Rep}_d(\m C[\pi_1(M_{g,\varphi})])\Big]_{\r{vir}} t^m$$
where for a finitely-generated algebra $A$ whose stack of representations can be written as a critical locus $\r{Rep}_d(A) \cong \r{crit}(f_d)$ for $f_d: X_d \rightarrow \m C$ and $X_d$ smooth we write
$$\Big[\r{Rep}_d(A)\Big]_{\r{vir}} = \int_{\r{Rep}_d(A)} [\phi_{f_d}] \in \c M^{\widehat{\mu}}_{\m C}$$
where $[\phi_{f_d}] \in \c M^{\widehat{\mu}}_{\r{Rep}_d(A)}$ is the (normalised) motivic vanishing cycle of $f_d$.

For the fundamental group algebras of the mapping tori of Riemann surfaces we are considering the isomorphism from \Cref{thm3.12} to the Jacobi algebra gives us the following critical locus structure
$$\r{Rep}_d(\m C[\pi_1(M_{g,\varphi})]) \cong \r{crit}(\widetilde{f}_{d})$$
where $\widetilde{f}_{d} = \r{Tr}(W')_{d}: \r{Rep}_{d}(\widetilde{Q}') \rightarrow \m C$, for $d = (d, \ldots, d) \in \m N^m$. Let $M_{d}(\widetilde{Q}')$ be the space of representations of $\m C\widetilde{Q}'$ which gives a smooth atlas of $\r{Rep}_{d}(\widetilde{Q}')$ since
$$\r{Rep}_{d}(\widetilde{Q}') \cong M_{d}(\widetilde{Q}')/G$$
for $G = \prod_{i \in Q'_0} \r{GL}_d$, and let $f_{d}: M_{d}(\widetilde{Q}') \rightarrow \m C$ be the lift of $\widetilde{f}_{d}$ to this atlas. It follows that
$$\Big[\r{Rep}_d(\m C[\pi_1(M_{g,\varphi})])\Big]_{\r{vir}} = \m L^{\frac{r}{2}}[G]^{-1}\int_{\r{crit}(f_d)} [\phi_{f_{d}}]$$
where $r= \r{dim}(G)$. Hence the problem of calculating motivic DT partition function for $\m C[\pi_1(M_{g,\varphi})]$ comes down to calculating the motivic vanishing cycles $\int [\phi_{f_{d}}]$. We now give some ideas on how one could go about doing this using motivic dimensional reduction \cite{bbs}, \cite{dm}, \cite{np} and power structures \cite{gzlmh}, \cite{bm}.

\subsection{Calculating motivic DT invariants}
In order to calculate the motivic vanishing cycle $\int [\phi_{f_d}]$ we would like to use 
[\cite{dm} Theorem 5.9] which says that
$$\int_{\r{crit}(f_d)} [\phi_{f_d}] = \int_{\r{crit}(f_d)} [\phi_{f_d}^{\r{eq}}] := \m L^{-\frac{\r{dim}\, X_d}{2}} \Big(\big[f_d^{-1}(0)\big] - \big[f_d^{-1}(1)\big]\Big)$$
for the equivariant vanishing cycle $\int [\phi_{f_d}^{\r{eq}}]$ of $f_d$. But we can't apply [\cite{dm} Theorem 5.9] directly as the variety $M_{d}(\widetilde{Q}')$ cannot locally be written in the form $\m A^r \times Z$ with a $\m G_d$-action such that $f_d$ is equivariant and the induced action on $Z$ is trivial. To remedy this we present a conjecture that utilises power structures to overcome this (see \cite{gzlmh}, \cite{bm} for an explanation of power structures on the Grothendieck ring of varieties).

Let $B$ be a finitely generated algebra with potential $W$ and let $A=\r{Jac}(B,W)$ be its Jacobi algebra. Let $\widetilde{\omega} \in B$ be such that its image $\omega \in A$ is central. Then define two new partition functions $\Phi_A^{\omega-\r{nilp}}(t)$ and $\Phi_A^{\omega-\r{inv}}(t)$ for the substacks $\r{Rep}^{\omega-\r{nilp}}(A)$ of representations of $A$ for which $\omega$ acts nilpotently and $\r{Rep}^{\omega-\r{inv}}(A)$ of representations of $A$ for which $\omega$ acts invertibly, by
\begin{align*}
    &\Phi_A^{\omega-\r{nilp}}(t) = \sum_{d=0}^\infty \Big[\r{Rep}_d(A)\Big]_{\r{vir}}^{\omega-\r{nilp}}t^d \\
    &\Phi_A^{\omega-\r{inv}}(t) = \sum_{d=0}^\infty \Big[\r{Rep}_d(A)\Big]_{\r{vir}}^{\omega-\r{inv}} t^d
\end{align*}
where $\big[\r{Rep}_d(A)\big]_{\r{vir}}^{\omega-\r{nilp}}$ (resp. $\big[\r{Rep}_d(A)\big]_{\r{vir}}^{\omega-\r{inv}}$) denotes the pushforward to $\c M^{\hat \mu}_{\m C}$ of the pullback of $\big[\r{Rep}_d(A)\big]_{\r{relvir}}$ to $\r{Rep}_d^{\omega-\r{nilp}}(A)$ (resp. $\r{Rep}_d^{\omega-\r{inv}}(A)$).

\begin{rem}
If we have a function $f:X \rightarrow \m C$ with critical locus $Z = \r{crit}(f)$ and an open subset $U \subset X$ then 
$$[\phi_f]|_{Z \cap U} = [\phi_{f|_U}].$$
Then $\widetilde{\omega}$ acting invertibly on representations of $B$ is an open condition and
$$\r{Rep}_{d}^{\omega-\r{inv}}(A) = \r{crit}(\r{Tr}(W)_d) \cap \r{Rep}_d^{\widetilde{\omega}-\r{inv}}(B).$$
Hence defining the virtual motive
$$\Big[\r{Rep}_d^{\omega-\r{inv}}(A)\Big]_{\r{vir}} := \int_{\r{crit}(\r{Tr}(W)_d)} [\phi_{f_d^{\r{inv}}}]$$
where $f_d^{\r{inv}} = \r{Tr}(W)_d|_{\r{Rep}_d^{\widetilde{\omega}-\r{inv}}(B)}$, we get that
$$\Big[\r{Rep}_d(A)\Big]_{\r{vir}}^{\omega-\r{inv}} = \Big[\r{Rep}_d^{\omega-\r{inv}}(A)\Big]_{\r{vir}}.$$
If we let $A_\omega$ denote the localisation of $A$ with respect to $\omega$ then it is clear that $\r{Rep}_d^{\omega-\r{inv}}(A) \cong \r{Rep}_d(A_\omega)$ hence
$$\Big[\r{Rep}_d(A)\Big]_{\r{vir}}^{\omega-\r{inv}} = \Big[\r{Rep}_d^{\omega-\r{inv}}(A)\Big]_{\r{vir}} = \Big[\r{Rep}_d(A_\omega)\Big]_{\r{vir}}$$
and so
\begin{align}
    \Phi_A^{\omega-\r{inv}}(t) = \Phi_{A_\omega}(t). \label{eq:4.1}
\end{align}
\end{rem}

We construct a new space of representations for our fundamental group algebras which will arise from a partially localised quiver algebra. As before consider a brane tiling $\Delta$ for a Riemann surface $\Sigma_g$ with automorphism $\varphi: \Sigma_g \xrightarrow{\sim} \Sigma_g$ of order $n$, giving dual quiver $Q_\Delta = Q$ and potential $W_\Delta = W$. Choose generating arrows $a \in Q_1$ for the action of $\varphi$ and choose isomorphism arrows $r_{i_u,t}: \varphi^t(i_u) \rightarrow \varphi^{t+1}(i_u)$ for all orbits of vertices and $t = 0, \ldots, n-2$ giving the quiver $Q^\#$ as explained in the proof of \cref{lem3.2}. Then define the algebra $\m C\widehat{Q}'$ to be the localisation of $\m CQ^\#$ with respect to the generating arrows $\{a\}$ (recall that $\m CQ'$ was the localisation of $\m CQ^\#$ with respect to the isomorphism arrows $r_{i_u,t}$). Since we assumed the potential $W'$ does not contain any inverse arrows (more correctly it did not contain both $r_{i_u,t}$ and $r_{i_u,t}^{-1}$ so if it contained an inverse then we simply change the arrow $r_{i_u,t}$ in $Q^\#$ to its inverse $r_{i_u,t}^{-1}$) then $W'$ also gives a potential on $\m C\widehat{Q}'$.

\begin{conj}\label{conj4.1}
Let $B =  \m C\widehat{Q'}$ be the path algebra of $\widehat{Q}'$ with potential $W'$ and let $A= \r{Jac}(B,W')$ be the Jacobi algebra of $(B,W')$. Let $\omega \in A$ be a central element and consider the three partition functions
\begin{align*}
    &\Phi_A(t) = \sum_{d=0}^\infty \Big[ \r{Rep}_d(A)\Big]_{\r{vir}} t^{d}\\
    &\Phi_A^{\omega-\r{nilp}}(t) = \sum_{d=0}^\infty \Big[\r{Rep}_d(A)\Big]_{\r{vir}}^{\omega-\r{nilp}} t^{d} \\
    &\Phi_A^{\omega-\r{inv}}(t) = \sum_{d=0}^\infty \Big[\r{Rep}_d(A)\Big]_{\r{vir}}^{\omega-\r{inv}} t^{d}\,.
\end{align*}
Then 
$$\Phi_A(t)=\big(\Phi_A^{\omega-\r{nilp}}(t) \big)^{\m L}$$
and
$$\Phi_A^{\omega-\r{inv}}(t)= \big(\Phi_A^{\omega-\r{nilp}}(t) \big)^{\m L-1}.$$
\end{conj}

Since $\r{Jac}(\widetilde{Q}',W')$ is a localisation of $A$ (with respect to the isomorphism arrows $\{r_{i_u,t}\}$) we try find a central element $\omega \in A$ such that $A_{\omega} = \r{Jac}(\widetilde{Q}',W')$. We expect to be able to do this because $\m C[\pi_1(M_{g, \varphi})] \cong \m C[\pi_1(\Sigma_g) \rtimes_\varphi \m Z]$ and so $\omega = ([1],n) \in \pi_1(\Sigma_g) \rtimes_\varphi \m Z$ is central, then using \cref{thm3.12} we have $\r{Jac}(\widetilde{Q}',W') \cong \r{Mat}_{m \times m}(\m C[\pi_1(M_{g, \varphi})])$ so by construction $A$ should have such elements. We can then apply \Cref{conj4.1} to write $\Phi_{\r{Jac}(\widetilde{Q}',W')}(t) = \Phi_{A_\omega}(t) = \Phi_A^{\omega-\r{inv}}(t)$ in terms of $\Phi_A(t)$. We get that
$$\Phi_{\r{Jac}(\widetilde{Q}',W')}(t) = \big(\Phi_{A}(t) \big)^{\frac{\m L-1}{\m L}}.$$
Then because \Cref{thm3.12} implies that
$$\Phi_{\m C[\pi_1(M_{g,\varphi})]}(t^m) = \Phi_{\r{Jac}(\widetilde{Q}',W')}(t)$$
where $m = |Q_0|$, it follows that
\begin{align}
    \Phi_{\m C[\pi_1(M_{g,\varphi})]}(t^m) = \big(\Phi_{A}(t) \big)^{\frac{\m L-1}{\m L}}\,.
\end{align}
Writing $d = (d, \ldots, d) \in \m N^m$ we have that $\r{Rep}_{d}(B) \cong M_{d}(B)/G$ with
$$M_{d}(B) = \prod_{r_{i_u,t}} \r{Mat}_{d \times d}(\m C) \times \prod_{a \in Q^\#_1} \r{GL}_d(\m C)$$
which is globally of the form $\m A^r \times Z$. Let $f_{d} =  \r{Tr}(W')_{d} : M_{d}(B) \rightarrow \m C$ and define a $\m G_m$-action on $M_{d}(B)$ by scaling the non-invertible matrices only. This gives us a $\m G_m$-action such that the induced action on $Z$ is trivial, and because $W'$ is homogeneous of degree $n$ in the isomorphism arrows $\{r_{i_u,t}\}$ it follows that $f_{d}$ is equivariant of degree $n$. Hence we can apply [\cite{dm} Theorem 5.9] to the virtual motives in the partition function $\Phi_A(t)$ reducing the problem to calculating the motives $[f_{d}^{-1}(0)]$ and $[f_{d}^{-1}(1)]$ for each $d$.

To see this in action let us consider our running example, namely $\Sigma_2$ with $\varphi$ being the rotation by $180^\circ$ and with the brane tiling given in \Cref{eg2.1} and quivers $Q, Q'$ and potentials $W, W'$ given in \Cref{eg3.1} and \cref{eg:3.16}. Let $B= \m C \widehat{Q}'$ and $A= \r{Jac}(\widehat{Q}', W')$ be the algebras described above. In particular the generators in $B$ are $\{a^{\pm 1}, b^{\pm 1}, c^{\pm 1}, d^{\pm 1}, e^{\pm 1}, r\}$. 

\begin{lem}
The element $\omega=rere+erer \in A$ is central and $A_\omega = \r{Jac}(\widetilde{Q'},W')$.
\end{lem}

\begin{proof}
We first show that $A_\omega =\r{Jac}(\widetilde{Q'},W').$ It suffices to show that $r^{-1}$ exists in $A_\omega$. And so because
$$er\omega^{-1}e = er(e^{-1}r^{-1}e^{-1}r^{-1}+r^{-1}e^{-1}r^{-1}e^{-1})e = err^{-1}e^{-1}r^{-1}e^{-1}e =r^{-1}$$
and $e, r, \omega^{-1}$ are all elements of $A_\omega$, then $r^{-1} \in A_\omega$.

Then to show $\omega \in A$ is central, since $\omega$ is not a 0-divisor it suffices to show that $\omega \in A_\omega$ is central. Now 
$$A_\omega= \r{Jac}(\widetilde{Q'},W') \xrightarrow[\sim]{\Psi} \r{Mat}_{2\times 2}(\m C[\pi_1(M_{2,\varphi})])$$
where, viewing $\pi_1(M_{2,\varphi}) \cong \pi_1(\Sigma_2)\rtimes_\varphi \m Z$, $\omega$ is sent under this isomorphism to the matrix 
$$\omega' = \begin{pmatrix}
([1],-2) & 0 \\
0 & ([1],-2)
\end{pmatrix}.$$
Since $\varphi^2 = \varphi^{-2} = \r{id}_{\Sigma_2}$ we have for $([\alpha], m) \in \pi_1(\Sigma_2) \rtimes_\varphi \m Z$
$$([\alpha],m) \cdot ([1],-2) = ([\varphi^{-2}(\alpha)], m-2) = ([\alpha], -2+m) = ([1],-2) \cdot ([\alpha],m).$$
Hence $\omega'$ is central and therefore so is $\omega$.
\end{proof}

It follows that we can apply the conjecture in this case for the particular choice of $\omega = rere + erer$. We have that $\r{Rep}_{(d,d)}(B) \cong \big(\r{Mat}_{d \times d}(\m C) \times \r{GL}_d^5 \big)/\r{GL}_d^2$ and $\r{Rep}_{(d,d)}(A) \cong \r{crit}(\widetilde{f_d})$ where $\widetilde{f_d} = \r{Tr}(W')_d: \r{Rep}_{(d,d)}(B) \rightarrow \m C$. Therefore
\begin{align*}
    \Phi_{\m C[\pi_1(M_{2,\varphi})]}(t^2) &= \sum_{d=0}^\infty \Big[\r{Rep}_d(\m C[\pi_1(M_{2,\varphi})])\Big]_{\r{vir}} t^{2d} \\
    &= \sum_{d=0}^\infty \Big[\r{Rep}_{2d}(A_\omega)\Big]_{\r{vir}} t^{2d} \\
    &= \sum_{d=0}^\infty \Big[\r{Rep}_{(d,d)}(A_\omega)\Big]_{\r{vir}} t^{2d} \\
    &= \sum_{d=0}^\infty \Big[\r{Rep}_{(d,d)}(A)\Big]_{\r{vir}}^{\omega-\r{inv}} t^{2d} \\
    &= \Phi_A^{\omega-\r{inv}}(t)
\end{align*}
where we take the series in the variable $t^2$ because a $d$-dimensional representation of $\m C[\pi_1(M_{2, \varphi})]$ corresponds to a $(d,d)$-dimensional representation of the Jacobi algebra $\r{Jac}(\widetilde{Q'}, W)$ which has two vertices. So by \Cref{conj4.1}
\begin{align*}
    \Phi_{\m C[\pi_1(M_{2,\varphi})]}(t^2) &= (\Phi_A(t))^{\frac{\m L-1}{\m L}}\\
    &= \Big(\Phi_A(t)^{\m L^{-1}}\Big)^{[\m C^*]} \\
    &=  \Phi_A(\m L^{-1}t)^{[\m C^*]} \\
    &=  \Phi_A(t)^{[\m C^*]}|_{t \mapsto \m L^{-1}t}.
\end{align*}
We can then apply [\cite{dm} Theorem 5.9] to find the motivic vanishing cycles for $A$ because our regular function $\widetilde{f}_d$ lifts to $f_d:\r{Mat}_{d \times d}(\m C) \times \r{GL}_d^5 \rightarrow \m C$ where $\r{Mat}_{d \times d}(\m C) \times \r{GL}_d^5$ is of the form $\m A^r \times Z$ and $f_d$ is equivariant of degree 2 when $\m G_m$ acts by scaling on $\r{Mat}_{d \times d} (\m C)$. We get that
\begin{align*}
    \Big[\r{Rep}_{(d,d)}(A)\Big]_{\r{vir}} &= \m L^{2d^2}[\r{GL}_d]^{-2} \int_{\r{crit}(f_d)} [\phi_{f_d}]\\
    &= \m L^{2d^2}[\r{GL}_d]^{-2} \int_{\r{crit}(f_d)} [\phi_{f_d}^{\r{eq}}]\\
    &= \m L^{-d^2} [\r{GL}_d]^{-2} \Big([f_d^{-1}(0)]-[f_d^{-1}(1)] \Big)
\end{align*}
and so we can find the motivic DT partition function of $\m C[\pi_1(M_{2,\varphi})]$ by studying the fibres of $f_d$ over $0$ and $1$ for each $d \in \m N$.

We end with some remarks on why \Cref{conj4.1} should be true, taking inspiration from [\cite{bbs} Section 2.4 and Proposition 2.6] and [\cite{dr} Section 3]. Let $\alpha =  (\alpha_1, \ldots, \alpha_d)$ be a partition of $d$ (ie. $\sum_{i=1}^d \alpha_i \cdot i = d$), let $\omega \in A$ be a central element and let $\widetilde{\omega} \in B$ be a lift of $\omega$. Denote by $\r{Rep}_d^\alpha(B)$ (resp. $\r{Rep}_d^\alpha(A)$) the closed substack of $d$-dimensional representations of $B$ for which $\widetilde{\omega}$ has generalised eigenvalues of shape $\alpha$ (resp. the closed substack of $d$-dimensional representations of $A$ for which $\omega$ has generalised eigenvalues of shape $\alpha$). Also denote by $\r{Rep}_d^\alpha(B)'$ (resp. $\r{Rep}_d^\alpha(A)'$) the open substack of $d$-dimensional representations of $B$ such that the generalised eigenvalues of $\widetilde{\omega}$ are distinct when split according to $\alpha$ (resp. the open substack of $d$-dimensional representations of $A$ such that the generalised eigenvalues of $\omega$ are distinct when split according to $\alpha$). Put another way we have
\begin{align*}
    \r{Rep}_d^\alpha(B)' =  \bigsqcup_{\alpha' \leq \alpha} \r{Rep}_d^{\alpha'}(B) \\
    \r{Rep}_d^\alpha(A)' =  \bigsqcup_{\alpha' \leq \alpha} \r{Rep}_d^{\alpha'}(A)
\end{align*}
where $\alpha' \leq \alpha$ runs over all sub-partitions of $\alpha$. We obtain stratifications
\begin{align}
    \r{Rep}_d(B) =  \bigsqcup_{\alpha \vdash d} \r{Rep}_d^\alpha(B) \nonumber\\
    \r{Rep}_d(A) =  \bigsqcup_{\alpha \vdash d} \r{Rep}_d^\alpha(A). \label{eq:4.01}
\end{align}
Fixing a presentation of $B$ and of $A$, a $d$-dimensional representation $\rho$ of $A$ can be characterised by a vector space $V$ of dimension $d$ along with $d \times d$-matrices $\rho(a_i)$ for each of the finite number of generators $a_i \in A$ that act on $V$ subject to the relations in $A$. Consider the generalised eigenspace decomposition of $V$ with respect to $\omega$
$$V= \bigoplus_{i=1}^r V_i$$
where
$$V_i = \r{Im}\Big(\prod_{j \neq i} (\rho(\omega) - \lambda_j \r{Id}_d)^{s_j}\Big)$$
with $\prod_j (x-\lambda_j)^{s_j}$ the characteristic polynomial of $\rho(\omega)$. Since $\omega \in A$ is central it follows that $\rho(a)$ respects this decomposition for any $a \in A$ i.e. $\rho(a)(V_i) \subset V_i$ for all $i$. Hence the matrix $\rho(a)$ is a block-diagonal matrix, and so for a partition $\alpha \vdash d$ we get an isomorphism
\begin{align}
    \r{Rep}_d^\alpha(A) \cong Z \subset \prod_{i=1}^d \big(\r{Rep}_i^{(0, \ldots, 0,1)}(A) \big)^{\alpha_i}/S_\alpha \label{eq:4.02}
\end{align}
where $(0, \ldots 0, 1) \vdash i$ is the partition consisting of the whole of $i$ and so $\r{Rep}_i^{(0, \ldots, 0,1)}(A)$ is the stack of $i$-dimensional representations of $A$ for which $\omega$ has a single generalised eigenvalue, $S_\alpha = \prod_{i=1}^a S_i^{\alpha_i}$ is a product of symmetric groups that acts on $\prod_{i=1}^d \big(\r{Rep}_i^{(0, \ldots, 0,1)}(A) \big)^{\alpha_i}$ by permuting the factors, and $Z$ is the substack for which the generalised eigenvalue of $\omega$ is distinct for each factor in the product. 

We now make the following two assumptions. The first is that we have an isomorphism of stacks
\begin{align}
    \r{Rep}_d^\alpha(A) \xrightarrow[\sim]{\xi} \bigg(\prod_{i=1}^d \big(\r{Rep}_i^{\omega-\r{nilp}}(A)\big)^{\alpha_i} \times \bigg(\prod_{i=1}^d\m A^{\alpha_i} \setminus \Delta\bigg)\bigg)/S_\alpha \label{eq:4.04}
\end{align}
where $\Delta \subset \prod_{i=1}^d \m A^{\alpha_i}$ is the big diagonal ie. the set of tuples $(x_j)$ with  $x_j \in \m A^1$ such that $x_{j_1}= x_{j_2}$ for some $j_1 \neq j_2$, and $S_\alpha$ acts on both $\prod_{i=1}^d \big(\r{Rep}_i^{\omega-\r{nilp}}(A)\big)^{\alpha_i}$ and $\prod_{i=1}^d\m A^{\alpha_i} \setminus \Delta$ by permuting the factors. For the second assumption consider the projections
$$\bigg(\prod_{i=1}^d \big(\r{Rep}_i(B)\big)^{\alpha_i} \times \bigg(\prod_{i=1}^d\m A^{\alpha_i} \setminus \Delta\bigg)\bigg)/S_\alpha \xrightarrow{\pi_B} \prod_{i=1}^d \big(\r{Rep}_i(B)\big)^{\alpha_i}/S_\alpha$$
and
$$\bigg(\prod_{i=1}^d \big(\r{Rep}_i(A)\big)^{\alpha_i} \times \bigg(\prod_{i=1}^d\m A^{\alpha_i} \setminus \Delta\bigg)\bigg)/S_\alpha \xrightarrow{\pi_A} \prod_{i=1}^d \big(\r{Rep}_i(A)\big)^{\alpha_i}/S_\alpha$$
and the natural inclusion
$$\prod_{i=1}^d \big(\r{Rep}_i^{\omega-\r{nilp}}(A)\big)^{\alpha_i}/S_\alpha \xhookrightarrow{\,\, j' \,\,} \prod_{i=1}^d \big(\r{Rep}_i(A)\big)^{\alpha_i}/S_\alpha.$$
Then we have the non-commuting triangle
\begin{equation}\label{eq:6.0}
\begin{tikzcd}[column sep=5em, row sep=5em]
\r{Rep}_d^\alpha(A) \arrow[d, "\xi", "\sim"' sloped] \arrow[r, hook, "\iota"] & \prod_{i=1}^d \big(\r{Rep}_i(A)\big)^{\alpha_i}/S_\alpha\\
Z_\alpha/S_\alpha \arrow[ru, "j"]
\end{tikzcd}    
\end{equation}
where $V_\alpha = \prod_{i=1}^d \big(\r{Rep}_i^{\omega-\r{nilp}}(A)\big)^{\alpha_i}$ and $Z_\alpha = V_\alpha \times \Big(\prod_{i=1}^d\m A^{\alpha_i} \setminus \Delta \Big)$, and $j$ is the composition $\pi_A \circ (j' \times \r{id}_{\prod_i \m A^{\alpha_i} \setminus \Delta})$. We assume that we have the equality of motives over $\r{Spec}(\m C)$
\begin{align}
    \int_{\r{Rep}_d^\alpha(A)} \iota^* [\phi_{\sum_i \alpha_i \overline{\r{Tr}(W')}_i}] &= \int_{\r{Rep}_d^\alpha(A)} \xi^* j^* [\phi_{\sum_i \alpha_i \overline{\r{Tr}(W')}_i}] \label{eq:6.1}
\end{align}
where the bar denotes that these functions are taken on the quotients modulo $S_\alpha$. Note that we have
\begin{align*}
   \xi^* j^* [\phi_{\sum_i \alpha_i \overline{\r{Tr}(W')}_i}] &= \xi^* (j' \times \r{id})^* \pi_A^* [\phi_{\sum_i \alpha_i \overline{\r{Tr}(W')}_i}]\\
    &= \xi^* (j' \times \r{id})^* \m L^{\frac{\sum_i \alpha_i}{2}}[\phi_{\sum_i \alpha_i \overline{\r{Tr}(W')}_i \circ \pi_B}]\\
    &= \xi^*\Big([\phi_{\sum_i \alpha_i \overline{\r{Tr}(W')}_i}]|_{V_\alpha/S_\alpha} \cdot \Big[\Big(\prod_{i=1}^d \m A^{\alpha_i} \setminus \Delta \Big)/S_\alpha \Big] \Big)
\end{align*}
where the second equality follows from [\cite{dm}, Proposition 5.3 (3)], and the third equality follows from the motivic Thom-Sebastiani isomorphism ([\cite{dl1}, Theorem 5.2.2] or [\cite{dm} Proposition 5.8 (3)]) and [\cite{dm}, Proposition 5.8 (5)]. Hence the assumption \eqref{eq:6.1} is equivalent to
\begin{align}
    \int_{\r{Rep}_d^\alpha(A)} \iota^* [\phi_{\sum_i \alpha_i \overline{\r{Tr}(W')}_i}] &= \int_{V_\alpha/S_\alpha} \Big( [\phi_{\sum_i \alpha_i \overline{\r{Tr}(W')}_i}]|_{V_\alpha/S_\alpha} \Big) \cdot \Big[\Big(\prod_{i=1}^d \m A^{\alpha_i} \setminus \Delta \Big)/S_\alpha \Big]. \label{eq:6.2}
\end{align}

Consider the closed substack of $\r{Rep}_d^\alpha(B)'$
$$X_\alpha = \prod_{i=1}^d \Big(\r{Rep}_i(B)\Big)^{\alpha_i}/S_\alpha \cap \r{Rep}_d^\alpha(B)'.$$
Then by using the following holomorphic Morse-Bott style lemma we have that locally around each point in $\r{Rep}_d^\alpha(A)' \subset \r{Rep}_d^\alpha(B)'$ the function $f_d^\alpha = \r{Tr}(W')_d^\alpha: \r{Rep}_d^\alpha(B)' \rightarrow \m C$ can be written as
\begin{align}
    f_d^\alpha = f_d^\alpha|_{X_\alpha} + \sum_j x_j^2 \label{eq:4.06}
\end{align}
where $j$ runs over all the coordinates that cut out $X_\alpha$ as a subspace of $\r{Rep}_d^\alpha(B)'$, and recall for functions $f: X \rightarrow \m A^1$ and $g: Y \rightarrow \m A^1$ we write $f+g$ for the composition $X \times Y \xrightarrow{f \times g} \m A^1 \times \m A^1 \xrightarrow{+} \m A^1$.

\begin{lem}[\cite{joy} Proposition 2.22]\label{lem4.2}
Let $Y$ be a smooth complex variety of dimension $d$, $f: Y \rightarrow \m C$ a regular function, and $X\subset Y$ a smooth subvariety of dimension $m$ such that $\r{crit}(f)= \r{crit}(f|_X)$. Then analytically locally around any $x \in \r{crit}(f)$ we may write
$$f=f|_X + \sum_{i=m+1}^d x_i^2$$
for local coordinates $x_1, \ldots, x_d$.
\end{lem}

We can apply this lemma since
\begin{align*}
    \r{crit}(f_d^\alpha|_{X_\alpha}) &= \bigg(\prod_{i=1}^d \Big(\r{Rep}_i(A)\Big)^{\alpha_i}\bigg)/S_\alpha \cap \r{Rep}_d^\alpha(A)'\\
    &= \r{Rep}_d^\alpha(A)'\\
    &= \r{crit}(f_d^\alpha)
\end{align*}
because $\r{Rep}_d^\alpha(A)' \subset \prod_{i=1}^d \big(\r{Rep}_i(A)\big)^{\alpha_i}/S_\alpha$ as $\omega \in A$ is central.

Now consider the Cartesian diagram
\[\begin{tikzcd}[column sep=6em, row sep=6em]
\r{Rep}_d^\alpha(B) \arrow[r, hook, "\widetilde{\iota}_\alpha''"] & \r{Rep}_d^\alpha(B)' \arrow[r, hook, "\widetilde{\iota}_\alpha'"] & \r{Rep}_d(B)\\
\r{Rep}_d^\alpha(A) \arrow[u, hook] \arrow[r, hook, "\iota_\alpha''"] \arrow[rr, bend right=20, hook, "\iota_\alpha"] & \r{Rep}_d^\alpha(A)' \arrow[u, hook] \arrow[r, hook, "\iota_\alpha'"] & \r{Rep}_d(A) \arrow[u, hook]
\end{tikzcd}\]
where recall $\widetilde{\iota}_\alpha'$ and $\iota_\alpha'$ are open inclusions and $\widetilde{\iota}_\alpha''$, $\iota_\alpha$ and $\iota_\alpha''$ are closed inclusions. Also consider the Cartesian square
\[\begin{tikzcd}[column sep=6em, row sep=6em]
X_\alpha \arrow[r, hook, "\widetilde{j}_\alpha"] & \prod_{i=1}^d \big(\r{Rep}_i(B)\big)^{\alpha_i}/S_\alpha\\
\r{Rep}_d^\alpha(A)' \arrow[u, hook] \arrow[r, hook, "j_\alpha"] & \prod_{i=1}^d \big(\r{Rep}_i(A)\big)^{\alpha_i}/S_\alpha \arrow[u, hook]
\end{tikzcd}\]
where $\widetilde{j}_\alpha$ and $j_\alpha$ are open inclusions. Note that $j_\alpha \circ \iota_\alpha'' = \iota$ from \eqref{eq:6.0}. Then using \eqref{eq:4.06} it follows, at least locally, that
\begin{align}
    [\r{Rep}_d(A)]_{\r{vir}} &= \sum_{\alpha \vdash d} \int_{\r{Rep}_d^\alpha(A)} \iota_\alpha^* [\phi_{\r{Tr}(W')_d}] \nonumber\\
    &= \sum_{\alpha \vdash d} \int_{\r{Rep}_d^\alpha(A)} \iota_\alpha^{\prime \prime *} \iota_\alpha^{\prime *} [\phi_{\r{Tr}(W')_d}] \nonumber\\
    &= \sum_{\alpha \vdash d} \int_{\r{Rep}_d^\alpha(A)} \iota_\alpha^{\prime \prime *} [\phi_{\r{Tr}(W')_d \circ \widetilde{\iota}_\alpha'}] \nonumber\\
    &= \sum_{\alpha \vdash d} \int_{\r{Rep}_d^\alpha(A)} \iota_\alpha^{\prime \prime *} [\phi_{\r{Tr}(W')_d^\alpha}] \nonumber\\
    &= \sum_{\alpha \vdash d} \int_{\r{Rep}_d^\alpha(A)} \iota_\alpha^{\prime \prime *} [\phi_{f_d^\alpha|_{X_\alpha}+ \sum_j x_j^2}] \nonumber\\
    &= \sum_{\alpha \vdash d} \int_{\r{Rep}_d^\alpha(A)} \iota_\alpha^{\prime \prime *} \Big([\phi_{f_d^\alpha|_{X_\alpha}}] \cdot [\phi_{\sum_j x_j^2}]\Big) \nonumber\\
    &= \sum_{\alpha \vdash d} \int_{\r{Rep}_d^\alpha(A)} \iota_\alpha^{\prime \prime *} [\phi_{f_d^\alpha|_{X_\alpha}}] \nonumber\\
    &= \sum_{\alpha \vdash d} \int_{\r{Rep}_d^\alpha(A)} \iota_\alpha^{\prime \prime *} [\phi_{\sum_i \alpha_i \overline{\r{Tr}(W')}_i \circ \widetilde{j}_\alpha}] \nonumber\\
    &= \sum_{\alpha \vdash d} \int_{\r{Rep}_d^\alpha(A)} \iota_\alpha^{\prime \prime *} j_\alpha^* [\phi_{\sum_i \alpha_i \overline{\r{Tr}(W')}_i}] \nonumber\\
    &= \sum_{\alpha \vdash d} \int_{\r{Rep}_d^\alpha(A)} \iota^* [\phi_{\sum_i \alpha_i \overline{\r{Tr}(W')}_i}] \nonumber\\
    &= \sum_{\alpha \vdash d} \int_{V_\alpha/S_\alpha} \Big([\phi_{\sum_i \alpha_i \overline{\r{Tr}(W')}_i}]|_{V_\alpha/S_\alpha} \Big) \cdot \Big[\Big(\prod_{i=1}^d \m A^{\alpha_i} \setminus \Delta \Big)/S_\alpha \Big]. \label{eq:4.07}
\end{align}
where the first equality follows from the cut and paste relations and the stratification \eqref{eq:4.01}, the third equality from the fact that $\widetilde{\iota}_\alpha'$ and $\iota_\alpha'$ are open and the square is Cartesian, the fifth equality from \eqref{eq:4.06}, the sixth equality from the motivic Thom-Sebastiani theorem, the seventh equality from the fact that the motivic vanishing cycle of a quadratic term $\int_{\m C} [\phi_{x^2}]$ is equal to 1, the eight equality from the fact that the function $\overline{\r{Tr}(W')}_d$ when restricted to block-diagonal matrices of the form $\alpha$ is equal to the function $\sum_{i=1}^d \alpha_i \overline{\r{Tr}(W')}_i$, the ninth equality from the fact that $\widetilde{j}_\alpha$ and $j_\alpha$ are open and the square is Cartesian, and the final equality from \eqref{eq:6.2}.

Let $q_\alpha: K^{S_\alpha}_0(\r{Var}/\m C)\big[[\r{GL}_n]^{-1} : n \in \m N \big] \rightarrow K_0(\r{Var}/\m C)\big[[\r{GL}_n]^{-1} : n \in \m N \big]$ be the quotient map (see [\cite{bbs} (1.5)] or [\cite{dr} (1.3)]) from the ring of $S_\alpha$-equivariant motives over $\r{Spec}(\m C)$ to the ring of motives over $\r{Spec}(\m C)$. Then if the function $g: Y \rightarrow \m A^1$ is $S_\alpha$-invariant and acts freely on $Y$ [\cite{bitt} Proposition 8.6] says that
$$q_\alpha \int_{Y} [\phi_g] = \int_{Y/S_\alpha} [\phi_{\overline{g}}]$$
where $\overline{g}: Y/S_\alpha \rightarrow \m A^1$ is the induced function on the quotient. It follows that
\begin{align}
    \int_{V_\alpha/S_\alpha}  \Big([&\phi_{\sum_i \alpha_i \overline{\r{Tr}(W')}_i}]|_{V_\alpha/S_\alpha} \Big) \cdot \Big[\Big(\prod_{i=1}^d \m A^{\alpha_i} \setminus \Delta \Big)/S_\alpha \Big] = \nonumber\\
    &= q_\alpha \bigg( \int_{V_\alpha} \Big([\phi_{\sum_i \alpha_i \r{Tr}(W')_i}]|_{V_\alpha} \Big) \cdot \Big[\prod_{i=1}^d \m A^{\alpha_i} \setminus \Delta \Big] \bigg)\nonumber\\
    &= q_\alpha \Bigg(\prod_{i=1}^d \bigg(\int_{\r{Rep}_i^{\omega-\r{nilp}}(A)} [\phi_{\r{Tr}(W')_i}]|_{\omega-\r{nilp}}\bigg)^{\alpha_i} \cdot \bigg[\prod_{i=1}^d \m A^{\alpha_i} \setminus \Delta \bigg] \Bigg). \label{eq:4.09}
\end{align}
Hence combining \eqref{eq:4.07} and \eqref{eq:4.09} gives us that locally
\begin{align}
    [\r{Rep}_d(A)]_{\r{vir}} &= \sum_{\alpha \vdash d} q_\alpha \Bigg(\prod_{i=1}^d \bigg(\int_{\r{Rep}_i^{\omega-\r{nilp}}(A)} [\phi_{\r{Tr}(W')_i}]|_{\omega-\r{nilp}}\bigg)^{\alpha_i} \cdot \bigg[\prod_{i=1}^d \m A^{\alpha_i} \setminus \Delta \bigg] \Bigg) \nonumber\\
    &= \sum_{\alpha \vdash d} q_\alpha \Bigg(\prod_{i=1}^d \Big(\big[\r{Rep}_i(A)\big]_{\r{vir}}^{\omega-\r{nilp}}\Big)^{\alpha_i} \cdot \bigg[\prod_{i=1}^d \m A^{\alpha_i} \setminus \Delta \bigg] \Bigg)\,. \label{eq:4.10}
\end{align}
Since power structures are linked to $\lambda$-ring structures, and the $\lambda$-ring structure on $K_0^{\hat{\mu}}(\r{Var}/\m C)$ is the exotic one from [\cite{dm} Lemma 4.1] we additionally require that
$\int [\phi_{\r{Tr}(W')_i}]|_{\omega-\r{nilp}}$ lies in the $\lambda$-subring 
$$K_0(\r{Var}/\m C)\big[[\r{GL}_n]^{-1} : n \in \m N \big] \subset K^{\hat{\mu}}_0(\r{Var}/\m C)\big[[\r{GL}_n]^{-1} : n \in \m N \big].$$
Then if \eqref{eq:4.10} can be upgraded to a global statement, from the definition of the power structure (see for example [\cite{bbs} equation (1.6)]) for the standard $\lambda$-ring structure on $K_0(\r{Var}/\m C)$, this tells us exactly that 
$$\big(\Phi_A^{\omega-\r{nilp}}(t)\big)^{\m L} = \Phi_A(t)$$
and restricting to the $\omega$-invertible locus implies that
$$\big(\Phi_A^{\omega-\r{nilp}}(t)\big)^{\m L -1} = \Phi_A^{\omega-\r{inv}}(t)$$
as per \Cref{conj4.1}.

\newpage
\bibliographystyle{plain}
\bibliography{references}
\end{document}